\documentclass[10pt]{amsart}
\usepackage{amsmath,amssymb,verbatim}
\usepackage[utf8]{inputenc}
\usepackage{enumerate}
\usepackage{pdflscape}  
\usepackage{amsthm}
\usepackage{mathrsfs}
\usepackage{color}
\usepackage[normalem]{ulem}
\usepackage{cancel}
\usepackage{tikz-cd}
\usepackage{tikz}
\usetikzlibrary{matrix}
\usepackage[all]{xy}
\usepackage{mathtools}
\usepackage{cleveref}
\usepackage{arydshln}
\usepackage{enumitem}
\usepackage{amsaddr}
\usepackage{comment}

\topskip=\baselineskip 
\newtheorem{theorem}{Theorem}[section]
\newtheorem{lemma}[theorem]{Lemma}
\newtheorem{proposition}[theorem]{Proposition}
\newtheorem{example}[theorem]{Example}

\newtheorem{remark}[theorem]{Remark}

\newtheorem{remarks}[theorem]{Remarks}

\newcounter{claimcounter} 

\newtheorem{maintheorem}{Theorem}

\makeatletter
\@addtoreset{equation}{section}
\newcommand{\mylabel}[2]{#2\def\@currentlabel{#2}\label{#1}}
\makeatother

\newcommand{\Gal}{{\rm Gal}}

\newcommand{\Aut}{\mbox{\rm Aut}}
\newcommand{\Inn}{\mbox{\rm Inn}}

\newcommand{\Imagen}{\mbox{\rm Im }}

\newcommand{\INV}{\mbox{\rm MCINV}}
\newcommand{\N}{{\mathbb N}}
\newcommand{\Z}{{\mathbb Z}}
\newcommand{\Q}{{\mathbb Q}}
  
\newcommand{\C}{{\mathbb C}}

\newcommand{\GEN}[1]{\left\langle #1 \right\rangle}

\newcommand{\U}{\mathcal{U}}

\newcommand{\Ese}[2]{\mathcal{S}\left(#1\mid #2\right)}

\DeclareMathOperator{\lcm}{lcm}

\newcommand{\qand}{\quad \text{and} \quad}

\DeclareMathOperator{\Core}{Core}
\DeclareMathOperator{\Deg}{Deg}
\DeclareMathOperator{\Res}{Res}

\title{The isomorphism problem for rational group algebras of finite metacyclic groups}

\author{Àngel García-Blázquez and Ángel del Río}
\thanks{Partially supported Grant PID2020-113206GB-I00 funded by MICIU/AEI/10.13039/501100011033 and by Grant 22004/PI/22 of Fundación Séneca de la Región de Murcia.}

\address{Departamento de Matem\'{a}ticas, Universidad de Murcia, 30100, Murcia, Spain} \email{angel.garcia11@um.es, adelrio@um.es}


\begin{document}

\begin{abstract}
We prove that if two finite metacyclic groups have isomorphic rational group algebras, then they are isomorphic. This contributes to understand where the line separating positive and negative solutions to the Isomorphism Problem for group algebras lies.
\end{abstract}

\maketitle

\section{Introduction}

The Isomorphism Problem for group algebras over a commutative ring $R$ asks whether the isomorphism type of $G$ is determined by the isomorphism type of the group algebra $RG$.
More precisely:

\begin{quote}
\textbf{Isomorphism Problem for group algebras}: Let $R$ be a commutative ring and let $G$ and $H$ be groups. If $RG$ and $RH$ are isomorphic as $R$-algebras, are $G$ and $H$ isomorphic groups?
\end{quote}

This problem has been studied extensively, with special emphasis on the cases where the groups are finite and the coefficient ring is either the integers or a field. One of the first results is due to G. Higman who proved that if $G$ and $H$ are finite abelian groups and $\Z G\cong \Z H$, then $G\cong H$ \cite{Higman1940Thesis,Higman1940Paper}.
The same result with the integers replaced by the rationals was proved by S. Perlis and G. L. Walker \cite{PerlisWalker1950}. However, if $G$ and $H$ are finite abelian groups of the same order, then $\C G\cong \C H$.
This illustrates a general principle for the Isomorphism Problem: ``the smaller the coefficient ring the greater the chances for positive answer''. This is a consequence of the fact that if $S$ is an $R$-algebra, then $SG\cong S\otimes_R RG$. Hence, if $RG\cong RH$, then $SG\cong SH$.

So, $\Z$ is the ring with greatest chances to get positive answers to the Isomorphism Problem.
A. Whitcomb obtained a positive solution for integral group rings of finite metabelian groups \cite{Whitcomb}. This was extended by W. Roggenkamp and L. Scott to finite abelian-by-nilpotent groups \cite{RoggenkampScott1987}, and to nilpotent-by-abelian groups by W. Kimmerle (see \cite[Section~XII]{RoggenkampTaylor}).
However, M. Hertweck found two non-isomorphic finite solvable groups with isomorphic integral group rings \cite{Hertweck2001}.
See \cite[Section~3]{MargolisdelRioSurvey} for an overview on the Isomorphism Problem for integral group rings.

In an influential survey paper on representations of finite groups, Brauer posed the following variants of the Isomorphism Problem for group algebras over fields: ``When two non-isomorphic groups have isomorphic group algebras'', ``If two groups $G_1$ and $G_2$ have isomorphic group algebras over every ground field $\Omega$, are $G_1$ and $G_2$ isomorphic?''.
A negative answer to the latter was given by E. Dade, in the form of two non-isomorphic finite metabelian groups $G$ and $H$ with $FG\cong FH$ for every field $F$ \cite{Dade71}.
This contrasts with Whitcomb Theorem.
For more information on the Isomorphism Problem see \cite[Chapter~14]{Passman1977}, \cite[Chapter~III]{Sehgal1978} and the recent survey \cite{Margolis2022}.

The aim of this paper is contributing to draw the line between positive and negative solutions for the Isomorphism Problem for rational group algebras of finite groups.
Perlis-Walker Theorem shows that abelian groups are on the positive side, while Dade example, shows that some metabelian groups are on the negative side. Actually, it is easy to find examples of non-isomorphic metabelian groups with isomorphic rational group algebras. For example, if $p$ is an odd prime,
\begin{equation}\label{Counterexample}
G=\GEN{a,b \mid a^{p^2}=b^p=1, a^b=a^{1+p}} \hspace{.3cm} \text{and} \hspace{.3cm}
H=\GEN{a,b,c \mid a^p=b^p=c^p=[a,b]=[a,c]=1, [b,c]=a},
\end{equation}
then $G$ and $H$ are metabelian and $\Q G\cong \Q H$ (see \Cref{ExCounterexample}).
The aim of this paper is to show that metacyclic groups lie on the positive side. More precisely, we prove the following theorem:

\begin{maintheorem}\label{Main}
Let $G$ and $H$ be finite metacyclic groups. If $\Q G\cong \Q H$, then $G\cong H$.
\end{maintheorem}

Observe that the example before the theorem shows that it is not sufficient to assume that one of the groups is metacyclic.
A first step to the proof of \Cref{Main} was obtained in \cite{GarciadelRioNilpotent} where we proved the theorem in the particular case where the groups are nilpotent. This will be an important tool in our proof of \Cref{Main}.

We briefly explain the structure of the paper.
In \Cref{SectionNotation} we introduce the basic notation and the main tools used in our proof of \Cref{Main}.
This includes a recent classification of finite metacyclic groups in terms of group invariants \cite{GarciadelRioClasification}, alternative to Hempel's classification \cite{Hempel2000}.
By this classification, the isomorphism type of a finite metacyclic group is uniquely determined by a four tuple $\INV(G)=(m,n,s,\Delta)$, where $m,n$ and $s$ are integers and $\Delta$ is a cyclic subgroup of the group of units of $\Z/m'\Z$ for some divisor $m'$ of $m$ (see \Cref{SSecMetacyclic} for details). The goal is proving that $\INV(G)$ is determined by the isomorphism type of the group algebra $\Q G$. In \Cref{SectionSketch} we explain a program to achieve this goal and, in the subsequent sections we complete it step-by-step. After proving in \Cref{SectionRDetermined} that the isomorphism type of $\Q G$ determines some group invariants of $G$, not explicitly included in $\INV(G)$, in \Cref{SectionEpsilonDetermined} we prove that it determines $s$.
In \Cref{SectionmnrDetermined}, it is proved that the isomorphism type of $\Q G$ also determines $m$ and $n$. Lastly, we prove that it determines $\Delta$ in \Cref{SectionDeltaDetermined}. This will complete the proof of \Cref{Main}.

We thank the anonymous referee for a careful reading and many suggestions, which have improved the paper considerably.

\section{Notation and preliminaries}\label{SectionNotation}

In this section we introduce the basic notation, prove some technical results and review some tools and lemmas needed for the remainder of the paper. We divide the section in the following subsections: Number theory, Group theory, The finite metacyclic groups and Wedderburn decomposition of rational group algebras.

\subsection{Number theory}\label{SectionNumberTheory}

We use the following notation for $m,q,p,t\in \Z$ with $m,q,p>0$, $q\mid m$, $p$ prime, $\gcd(t,m)=1$, $\pi$ a set of prime integers, and $A$ a finite set:
\begin{center}
\begin{tabular}{rl}
$\pi(m)=$& the set of prime integers dividing $m$,\\
$m_p=$ & greatest power of $p$ dividing $m$, \\
$m_\pi=$ & $\prod_{p\in \pi} m_p$, \\
$v_p(m)=$&$\log_p(m)$, \\
$\U_m=$ & group of units of $\Z/m\Z$,\\
$[t]_m=$& the element of $\U_m$ represented by $t$, \\
$o_m(t)=$& $|[t]_m|$, the multiplicative order or $t$ module, $m$,\\
$\GEN{t}_m=$&$\GEN{[t_m]}$, the subgroup of $\U_m$ generated by $[t]_m$,\\
$\Res_q:\U_m\to \U_q$, & the map $[t]_m\mapsto [t]_q$,\\
$\zeta_m$,& a complex primitive root of unity,\\
$\Q_m=$& $\Q(\zeta_n)$, \\
$|A|=$ & cardinality of $A$, \\
$\pi(A)=$ & $\pi(|A|)$.
\end{tabular}
\end{center}

Suppose that $n$ is a power of a prime $p$ and let $\Gamma$ be a cyclic $p$-subgroup of $\U_n$. If $p$ is odd, then $\Gamma=\GEN{1+r}_n$ for some $p\mid r\mid n$ and this is the only subgroup of $\U_n$ with order $\frac{n}{r}$.
Suppose that $p=2$. If $n\mid 2$, then $\U_n=1$. If $4\mid n$, then $\U_n$ has $2(\log_2(n)-1)$ cyclic subgroups, namely the following ones
    $$\GEN{1+r}_n \qand \GEN{-1+r}_n, \text{ with } 4\mid r \mid n.$$

If $x$ and $n$ are integers with $x\ne 0$ and $n>0$, then we denote:
$$\Ese{x}{n}=\sum_{i=0}^{n-1} x^i =
\begin{cases} n, & \text{ if } x=1; \\
\frac{x^n-1}{x-1}, & \text{otherwise}.
\end{cases}$$

The notation $\Ese{x}{n}$ occurs in the following statement:
\begin{equation}\label{Potencia}
\text{If } g^h=g^x \text{ with } g \text{ and } h
\text{ elements of a group, then } (hg)^n = h^n g^{\Ese{x}{n}}.
\end{equation}

The following lemma collects some properties of the operator $\Ese{-}{-}$ (see  \cite[Lemma~2.1]{GarciadelRioClasification}).

\begin{lemma}\label{PropEse}
Let $p$ be a prime integer, let $R$ and $m$ be integers with $m>0$ and $R>1$ and $v_p(R-1)\ge 1$.
Then
\begin{enumerate}
\item \label{vpRm-1}
$v_p(R^m-1)=\begin{cases}
v_p(R-1)+v_p(m), & \text{if } p\ne 2 \text{ or } v_p(R-1)\ge 2; \\
v_p(R+1)+v_p(m), & \text{if } p=2, v_p(R-1)=1 \text{ and } 2\mid m; \\
1, & \text{otherwise}.
\end{cases}$

\item \label{vpEse}
$v_p(\Ese{R}{m})=\begin{cases}
v_p(m), & \text{if } p\ne 2 \text{ or } v_p(R-1)\ge 2; \\
v_p(R+1)+v_p(m)-1, & \text{if } p=2, v_p(R-1)=1 \text{ and } 2\mid m; \\
0, & \text{otherwise}.
\end{cases}$

\item \label{op}
$o_{p^m}(R)=\begin{cases}
p^{\max(0,m-v_p(R-1))}, & \text{if } p\ne 2 \text{ or } v_p(R-1)\ge 2; \\
1, & \text{if } p=2, v_p(R-1)=1 \text{ and } m\le 1; \\
2^{\max(1,m-v_2(R+1))}, & \text{otherwise}.
\end{cases}$

\end{enumerate}
\end{lemma}

\subsection{Group theory}\label{SubsectionGroupTheory}

By default all the groups in this paper are finite.
In this section $G$ is a group and $g,h\in G$.
We use the standard group theoretical notation:
$Z(G)=$ center of $G$, $G'=$ commutator subgroup of $G$, $\Aut(G)=$ group of automorphisms of $G$, $|g|=$ order of $g$, $g^h=h^{-1}gh$, $[g,h]=g^{-1}g^h$,  $H\le G$ means ``$H$ is a subgroup of $G$'' and  $N\unlhd G$ means ``$N$ is a normal subgroup of $G$''.
If $H\le G$, then $[G:H]$ denotes the index of $H$ in $G$,
$N_G(H)$ the normalizer of $H$ in $G$, $C_G(H)$ the centralizer of $H$ in $G$ and $\Core_G(H)$ the core of $H$ in $G$, i.e. the greatest subgroup of $H$ that is normal in $G$.

If $\pi$ is a set of primes, then $g_{\pi}$ and $g_{\pi'}$ denote the $\pi$-part and $\pi'$-part of $g$, respectively.
We often use $G_{\pi}$ (respectively, $G_{\pi'}$) to denote some Hall $\pi$-subgroup (respectively, Hall $\pi'$-subgroup) of $G$.
When $p$ is a prime we rather write $G_p$ and $g_p$ than  $G_{\{p\}}$ and
$g_{\{p\}}$.

We denote
\begin{equation}\label{piDef}
\pi_G=\{p\in \pi(G) : G \text{ has a normal } p\text{-complement} \} \qand
\pi'_G=\pi(G)\setminus \pi_G.
\end{equation}
Observe that if $p\in \pi_G$ if and only if $G$ has a unique Hall $p'$-subgroup $G_{p'}$, and hence $G_{\pi'_G}=\cap_{p\in \pi_G} G_{p'}$ is the unique Hall $\pi'_G$-subgroup of $G$.

Let $m$ be a positive integer and let $A$ be a cyclic group of order $m$. Then we have isomorphisms
	$$\Aut(A)\stackrel{\sigma^A}{\leftarrow} \U_m\rightarrow \Aut(\Q_m)$$
which associate $[t]_m\in \U_m$ with the automorphism of $A$ given by $a\mapsto a^t$ and the automorphism of $\Q_m$ given by $\zeta_m\mapsto \zeta_m^t$.
Sometimes we abuse the notation and identify elements of $\Aut(A)$, $\U_m$ and $\Aut(\Q_m)$ via these isomorphisms. For example, if $X$ is a subset (respectively, an element $x$) of $\Aut(A)$ or $\U_m$, then $\Q_m^X$ (respectively, $\Q_m^x$) denotes the subfield of $\Q_m$ formed by the elements fixed by the images of the elements $X$ (respectively, of $x$) in $\Aut(\Q_m)$.

If $A$ is a cyclic normal subgroup of a group $G$ and $m=|A|$, then we define
\begin{eqnarray*}
T_G(A)&=&(\sigma^A)^{-1}(\Inn_G(A)), \\
r^G(A)&=& \text{greatest divisor } r \text{ of } m \text{ such that } \Res_{r_{2'}}(T_G(A))=1 \text{ and }\Res_{r_2}(T_G(A))\subseteq \GEN{-1}_{r_2}; \\
\epsilon^G(A)&=&
\begin{cases}
-1, & \text{if } \Res_{r^G(A)_2}(T_G(A))\ne 1; \\
1, & \text{otherwise};
\end{cases} \\
k^G(A)&=&|\Res_{m_{\nu}}(T_G(A))|, \text{ with } \nu = \pi(m)\setminus \pi(r^G(A)).
\end{eqnarray*}
where $\Inn_G(A)$ denotes the restriction to $A$ of the inner automorphisms of $G$.
Observe that if $\nu=\emptyset$, then $m_{\nu}=1$.
Moreover, $\epsilon^G(A)=-1$ if and only if $A$ has an element of order $4$ which is conjugate in $G$ to its inverse. Indeed, if denote $r=r^G(A)$, then $\epsilon^G(A)=-1$ if and only if $\Res_{r_2}(T_G(A))=\GEN{-1}_{r_2}\ne 1$ and this happen if and only if $A$ has an element $a$ of order $r_2$ such that $a$ and $a^{-1}$ are different and conjugate in $G$. It is now clear that the latter holds if and only if $A$ has an element of order $4$ conjugate in $G$ to its inverse.

If $G$ is a group, then a cocyclic subgroup of $G$ is a normal subgroup $N$ of $G$ such that $G/N$ is cyclic . We close this subsection computing the cocyclic subgroups of a $2$-generated abelian $p$-group. This will be used in \Cref{SectionmnrDetermined}.

\begin{lemma}\label{Cocyclic}
Let $L=\GEN{g}\times \GEN{h}$ be a an abelian $p$-group with $|g|\ge |h|$ and let
$$C_L = \left\{(i,y,x)\in \Z^3 : i \in \{1,2\},\quad  1\le x\le y \qand \begin{cases} y\mid |h|, & \text{if } i=1; \\
y\mid |g|, p\mid x \text{ and } p\mid y\mid |h|x, & \text{if } i=2 \end{cases} \right\}$$
Then
$$(i,y,x)\mapsto K_{i,y,x}=\begin{cases} \GEN{gh^x,h^y}, & \text{if } i=1; \\
\GEN{g^xh,g^y}, & \text{if } i=2;\end{cases}$$
defines a bijection from $C_L$ to the set of cocyclic subgroups of $L$.
Moreover, for every $(i,y,x)\in C_L$ we have
$$[L:K_{i,y,x}]=y,\quad
\GEN{g}\cap K_{i,y,x}=\begin{cases} \GEN{g^{\frac{y}{x_p}}}, & \text{if } i=1; \\
\GEN{g^y}; & \text{if } i=2; \end{cases}
\qand
\GEN{h}\cap K_{i,y,x}=\begin{cases} \GEN{h^y}, & \text{if } i=1; \\
\GEN{h^{\frac{y}{x_p}}}; & \text{if } i=2. \end{cases}$$
\end{lemma}

\begin{proof}
Clearly, if $(i,y,x)\in C_L$, then $K_{i,y,x}$ is a cocyclic subgroup of $L$.

Let $K$ be a cocyclic subgroup of $L$. Suppose that $K\not\subseteq\GEN{g^p,h}$. Then $K$ contains $\GEN{gh^z}$ for some integer $z$. Moreover, $L=\GEN{gh^z}\times \GEN{h}$.
Therefore $K=\GEN{gh^z,h^y}=\GEN{gh^x,h^y}=K_{1,y,x}$, with $y = [L:K]\mid |h|$ and $x$ the unique integer in the interval $[1,y]$ such that $x\equiv z \mod y$.
Moreover, $K_{1,y,x}=\GEN{gh^x}\times \GEN{h^{y}}$ and $K_{1,y,x}\cap \GEN{h}=\GEN{h^{y}}$.
Let $u$ be a positive integer. Then $g^u\in K_{1,y,x}$ if and only if there are integers $a$ and $b$ such that $g^u = g^ah^{ax+by}$. In that case, $u\equiv a \mod |g|$. As $y\mid |h|$ and $|h| \mid |g|$, it follows that $g^u\in K_{1,y,x}$ if and only if there is an integer $b$ such that $|h| \mid xu+by$ if and only if $y \mid xu$ if and only if $\frac{y}{\gcd(x,y)}\mid u$ if and only if $\frac{y}{x_p}\mid u$.
This shows that $\GEN{g}\cap K_{1,y,x}= \GEN{g^{\frac{y}{x_p}}}$.

Suppose otherwise that $K\subseteq \GEN{g^p,h}$. Then $K\cap \GEN{g}=\GEN{g^y}$ for some $y$ with $p\mid y$ and $y \mid |g|$, and $K$ contains $\GEN{g^xh}$ for some integer $x$ with $p\le x \le y$ and $p\mid x$.
Then $L=\GEN{g^xh,g}$ and $|g\GEN{g^xh,g^y}|=y$ and hence $K=\GEN{g^xh,g^y,g^\delta}$ for some $\delta\mid y$. However, as $K\cap \GEN{g}=\GEN{g^y}$ it follows that $K=\GEN{g^xh,g^y}$.
As $g^{x|h|}\in \GEN{g}\cap K=\GEN{g^y}$, $y\mid x|h|$. Thus $(2,y,x)\in C_L$ and $K=K_{2,y,x}$.
Furthermore, $[L:K]=[\GEN{g}:K\cap \GEN{g}]=[\GEN{g},\GEN{g^y}]=y$, since  $L=\GEN{K,g}$. Conversely, suppose that $(2,y,x)\in C_L$
Let $u$ be a positive integer. Then $g^u\in K_{2,y,x}$ if and only if $g^u=g^{ax+by}h^a$ for some integers $a$ and $b$. In that case $|h|\mid a$ and hence $y\mid ax$ because $y\mid x|h|$. Moreover $u\equiv ax+by \mod |g|$ and as $y\mid |g|$ we have that $y\mid u$.
Therefore $\GEN{g}\cap K_{2,y,x}= \GEN{g^{y}}$.
Finally, $h^u\in K_{2,y,x}$ if and only if there are integers $a$ and $b$ with $h^u=g^{ax+by}h^{a}$. Then $a=u+c|h|$ for some integer $c$ and $ux+xc|h|+ by \equiv 0 \mod |g|$. As $y\mid |g|$ and $y\mid x|h|$ we deduce that $y \mid ux$. Thus $\GEN{h}\cap K_{2,y,x}\subseteq \GEN{h^{\frac{y}{x_p}}}$ and as $h^{\frac{y}{x_p}}=(g^xh)^{\frac{y}{x_p}}(g^y)^{-\frac{x}{x_p}}\in K_{2,y,x}$, we conclude that $\GEN{h}\cap K_{2,y,x}=\GEN{h^{\frac{y}{x_p}}}$.

Let $(i_1,y_1,x_1),(i_2,y_2,x_2)\in C_L$ with $K_{i_1,y_1,x_1}=K_{i_2,y_2,x_2}$. It remains to prove that $(i_1,y_1,x_1)=(i_2,y_2,x_2)$.
First of all $i_1=i_2$, as $K_{1,y_1,x_1}\not\subseteq \GEN{g^p,h}$ and $K_{2,y_2,x_2}\subseteq \GEN{g^p,h}$.
Suppose that $i_1=i_2=1$.
Then $\GEN{h^{y_1}}=K_{1,y_1,x_1}\cap \GEN{h} = K_{1,y_2,x_2}\cap \GEN{h}=\GEN{h^{y_2}}$ and therefore $y_1=y_2$.
Moreover, $gh^{x_1}=(gh^{x_2})^a(h^{y_2})^b=g^a h^{ax_2+by_2}$ for some integers $a$ and $b$. Then $a\equiv 1 \mod |g|$ and, as $|h|\mid |g|$ and $y_2\mid |h|$, we have that $x_1\equiv x_2 \mod y_1$. Then $x_1=x_2$, as $1\le x_1,x_2\le y_1=y_2$.
Suppose that $i_2=2$. Since $y_1, y_2\mid |g|$ and $\GEN{g^{y_1}}=K_{2,y_1,x_1}\cap \GEN{g}=K_{2,y_2,x_2}\cap \GEN{g}=\GEN{g^{y_2}}$, we have $y_1=y_2$.
Moreover,
$g^{x_1}h=(g^{x_2}h)^a(g^{y_2})^b=g^{ax_2+by_2}h^a$ for some integers $a$ and
$b$. Then $a=1 + c|h|$ for some integer $c$ and as $y_2\mid x_2|h|$ and $y_1\mid |g|$, it follows that $x_1\equiv ax_2=x_2+cx_2|h| \equiv x_2 \mod y_2$. Thus $x_1=x_2$.
\end{proof}

\subsection{The finite metacyclic groups}\label{SSecMetacyclic}

Let $G$ be a group.
A \emph{metacyclic factorization} of $G$ is an expression $G=AB$ where $A$ is a normal cyclic subgroup of $G$ and $B$ is a cyclic subgroup of $G$.
Clearly a group is metacyclic if and only if it has a metacyclic factorization.

In the remainder of this subsection $G$ is a finite metacyclic group.
A metacyclic factorization $G=AB$ is said to be \emph{minimal} in $G$ if, in the lexicographical order, $(|A|,r^G(A),[G:B])$ is minimal in
	$$\{(|C|,r^G(C),[G:D]) \mid G=CD, \text{ metacyclic factorization}\}.$$
In that case we denote
    $$m^G=|A|, \quad n^G=[G:A], \quad s^G=[G:B], \quad r^G=r^G(A), \quad \epsilon^G=\epsilon^G(A) \qand k^G=k^G(A).$$
By the results of \cite{GarciadelRioClasification} the above integers are independent of the minimal metacyclic factorization and $\pi'_G=\pi(m^G)\setminus \pi(r^G)$. Thus, they are invariants of $G$. Another invariant the plays an important role in this paper is the following:
$$R^G=T_G(G'_{\pi'_G}).$$

We need to define a last invariant. For that we set $m=m^G$, $n=n^G$, $s=s^G$, $r=r^G$, $\epsilon=\epsilon^G$, $k=k^G$, $\pi=\pi_G$ and $\pi'=\pi'_G$ and denote
    $$m'=m_{\pi'} \prod_{p\in \pi(r)} m'_p$$
with $m'_p$ defined as follows for $p\in \pi(r)$:
    \begin{equation}\label{m'}
    \begin{split}
    & \text{if }
    \epsilon=1 \text{ or } p\ne 2
    \text{, then }
    m'_p = \min\left(m_p,k_pr_p,\max\left(r_p,s_p,r_p\frac{s_p k_p}{n_p}\right)\right);\\
    & \text{otherwise }
    m'_2=\begin{cases} r_2, & \text{if } k_2\le 2  \text{ or } m_2\le 2r_2; \\
    \frac{m_2}{2}, & \text{if } 4\le k_2<n_2, 4r_2\le m_2,  \text{ and if } s_2\ne n_2r_2 \text{, then } 2s_2=m_2<n_2r_2; \\
    m_2, & \text{otherwise}.
    \end{cases}
    \end{split}
    \end{equation}
Then we denote
$$\Delta^G = \Res_{m'}(T_G(A)).$$
Again, by the results in \cite{GarciadelRioClasification}, $\Delta^G$ is independent of the minimal metacyclic factorization and following the notation in loc. cit. we denote
	$$\INV(G)=(m^G,n^G,s^G,\Delta^G).$$
We recall the main results of \cite{GarciadelRioClasification}:

\begin{theorem}\label{Invariants}
Two finite metacyclic groups $G$ and $H$ are isomorphic if and only if $\INV(G)=\INV(H)$.
\end{theorem}

\begin{theorem}\label{Parameters}
Let $m,n,s\in \N$ and let $\Delta$ be a cyclic subgroup of $\U_{m'}$ with $m'\mid m$.
Let
\begin{eqnarray*}
r&=& \text{greatest divisor of } m \text{ such that } \Res_{r_{2'}}(\Delta)=1 \text{ and }\Res_{r_2}(\Delta)\subseteq \GEN{-1}_{r_2}; \\
\epsilon&=&
\begin{cases}
-1, & \text{if } \Res_{r_2}(\Delta)\ne 1; \\
1, & \text{otherwise};
\end{cases} \\
k&=&|\Res_{m_{\pi'}}(\Delta)|, \text{ with } \pi' = \pi(m)\setminus \pi(r).
\end{eqnarray*}
Then the following conditions are equivalent:
\begin{enumerate}
 \item $(m,n,s,\Delta)=\INV(G)$ for some finite metacyclic group $G$.
 \item\label{Param} \begin{enumerate}
        \item\label{ParamB} $s$ divides $m$, $|\Delta|$ divides $n$ and $m_{\pi'}=s_{\pi'}=m'_{\pi'}$.
        \item\label{Param'} \eqref{m'} holds for every $p\in \pi(r)$.
        \item\label{Param-} If $\epsilon=-1$, then $\frac{m_2}{r_2}\le n_2$, $m_2\le 2s_2$ and $s_2\ne n_2r_2$. If moreover $4\mid n$, $8\mid m$ and $k_2<n_2$, then $r_2\le s_2$.
		\item\label{Param+} For every $p\in \pi(r)$ with $\epsilon^{p-1}=1$, we have
		$\frac{m_p}{r_p}\le s_p\le n_p$ and if $r_p> s_p$, then  $n_p<s_pk_p$;
	\end{enumerate}
\end{enumerate}
\end{theorem}

Moreover, if $(m,n,s,\Delta)$ satisfy the conditions of \Cref{Parameters}, then with the same notation for $m', r, \epsilon$ and $k$, one can construct a metacyclic group $G$ with $MCINV(G) = (m,n,s,\Delta)$ as follows:
Fix an automorphism $\gamma$ of a cyclic group $\GEN{a}$ of order $m$ satisfying the following conditions:
	$$\Res_{m'}(\GEN{\gamma})=\Delta \quad
	\text{ and } \Res_{m_p}(\GEN{\gamma})=\GEN{\epsilon^{p-1}+r_p}_{m_p} \text{ for every } p\in \pi(r),$$
where we are abusing the notation by identifying, $\Aut(\GEN{a^{m/d}})$ and $\U_d$ for $d\mid m$,  as explained in \Cref{SectionNumberTheory}.
Then the group given by the following presentation:
	\begin{equation}\label{Presentation}
	    G = \GEN{a,b \mid a^m = 1, b^n=a^s, a^b=\gamma(a)},
	\end{equation}
satisfies $MCINV(G)=(m,n,s,\Delta)$ and $G= \GEN{a}\GEN{b}$ is a minimal metacyclic factorization of $G$.

The following lemma is a direct consequence of \Cref{Parameters}.

\begin{lemma}\label{conditionsA-D}
Let $p$ be a prime integer and let $G$ be a metacyclic $p$-group.
If $\INV(G)=(p^\mu,p^\nu,p^\sigma,\GEN{e+p^\rho}_{p^\mu})$, then
 \begin{enumerate}[label=(\Alph*)]
 \item\label{rhomu1} If $\rho=0$, then $\mu=0$.
 \item\label{rhomu2} If $p=2$ and $\rho=1$, then $\mu=1$.
 \item\label{epsilon1} If $e=1$, then $\rho\le \sigma\le \mu \le \rho+\sigma$ and $\sigma\le\nu$.
 \item\label{epsilon-1} If $e=-1$, then
 \begin{enumerate}
 	\item $p=2\le \rho\le \mu$, $\nu\ge 1$, $\mu-1\le \sigma\le \mu\le \rho+\nu\ne \sigma$ and
 	\item if $2\ge \nu$ and $3\ge \mu$, then $\rho\le \sigma$.
 \end{enumerate}
 \end{enumerate}
\end{lemma}

\subsection{Wedderburn decomposition of rational group algebras}

If $F$ is a field and $S$ is a finite dimensional central simple $F$-algebra, then $\Deg(S)$ denotes the degree of $S$, i.e. $\dim_F S=\Deg(S)^2$ (cf. \cite{Pierce1982}).

Let $F/K$ be a finite Galois field extension and let $G=\Gal(F/K)$.
Let $\U(F)$ denote the multiplicative group of $F$.
If $f:G\times G\rightarrow \U(F)$ is a $2$-cocycle, then $(F/K,f)$ denotes the crossed product
	$$(F/K,f)=\sum_{\sigma\in G} t_\sigma F, \quad xt_\sigma=t_\sigma \sigma(x), \quad  t_{\sigma}t_{\tau}=t_{\sigma\tau}f(\sigma,\tau), \quad (x\in F, \sigma,\tau\in G).$$

Suppose that $G$ is cyclic of order $n$ and generated by $\sigma$ and let $a\in \U(K)$. Then there is a cocycle $f:G\times G\rightarrow \U(K)$ given by
	$$f(\sigma^i,\sigma^j)=\begin{cases} 1, & \text{if } 0\le i,j,i+j<n; \\
	a, & \text{if } 0\le i,j<n\le i+j; \end{cases}$$
and the crossed product algebra $(F/K,f)$ is said to be a \emph{cyclic algebra}. We denote this algebra $(F,\sigma,a)$, and it can be described as follows:
	$$(F,\sigma,a)=\sum_{i=0}^{n-1} u^i F = F[u\mid xu=u\sigma(x), u^n=a]$$

If $S$ is a semisimple ring, then $S$ is a direct product of central simple algebras and such expression of $S$ is called the \emph{Wedderburn decomposition} of $S$ and its simple factors  are called the \emph{Wedderburn components} of $S$.
The Wedderburn components of $S$ are the direct summands of the form $Se$ with $e$ a primitive central idempotent of $S$.

Let $G$ be a finite group. By Maschke's Theorem, $\Q G$ is semisimple. Moreover,the center of each Wedderburn component $S$ of $\Q G$ is isomorphic to the field of character values $\Q(\chi)$ of any irreducible character $\chi$ of $G$ satisfying $\chi(S)\ne 0$ (see e.g. \cite[Theorem~3.3.1(4)]{JespersdelRioGRG1}).
It is well known that $\Q(\chi)$ is a finite abelian extension of $\Q$ inside $\C$ and henceforth it is the unique subfield of $\C$ isomorphic to $\Q(\chi)$.
Sometimes we will abuse the notation and consider $Z(S)$ as equal to $\Q(\chi)$.

An important tool for us is a technique introduced in \cite{OlivieridelRioSimon2004} to describe the Wedderburn decomposition of $\Q G$ for $G$ a metabelian group.
See also \cite[Section~3.5]{JespersdelRioGRG1} and the implementation in the GAP \cite{GAP4} package Wedderga \cite{Wedderga4.10.5}.
We recall here its main ingredients.

If $H$ is a subgroup of $G$, then denote $\widehat{H}=|H|^{-1}\sum_{h\in H} h$, as an element in $\Q G$.
It is clear that $\widehat{H}$ is an idempotent of $\Q G$ and it is central in $\Q G$ if and only if $H$ is normal in $G$.

If $N\unlhd G$ then we denote
	$$\varepsilon(G,N)=\begin{cases} \widehat{G}, & \text{if } G=N; \\
	\prod_{D/N\in M(G/N)} (\widehat{N}-\widehat{D}), & \text{otherwise}. \end{cases}$$
where $M(G/N)$ denotes the set of minimal normal subgroups of $G/N$.
Clearly $\varepsilon(G,N)$ is a central idempotent of $\Q G$.

If $(L,K)$ is a pair of subgroups of $G$ with $K\unlhd L$, then we denote
	$$e(G,L,K) = \sum_{gC_G(\varepsilon(L,K))\in G/C_G(\varepsilon(L,K))} \varepsilon(L,K)^g.$$
Then $e(G,L,K)$ belongs to the center of $\Q G$. If moreover, $\varepsilon(L,K)^g\varepsilon(L,K)=0$ for every $g\in G\setminus C_G(\varepsilon(L,K))$, then $e(G,L,K)$ is an idempotent of $\Q G$.

A \emph{strong Shoda pair} of $G$ is a pair $(L,K)$ of subgroups of $G$ satisfying the following conditions:
\begin{itemize}
\item[(SS1)] $K\subseteq L \unlhd N_G(K)$,
\item[(SS2)] $L/K$ is cyclic and maximal abelian in $N_G(K)/K$,
\item[(SS3)] $\varepsilon(L,K)^g\varepsilon(L,K)$ for every $g\in G\setminus C_G(\varepsilon(L,K))$.
\end{itemize}

Suppose that $(L,K)$ is a strong Shoda pair of $G$ with $m=[L:K]$ and let $N=N_G(K)$.
Hence $L/K\cong C_m$ and the action of $N$ by conjugation on $L$ induces an injective homomorphism $\alpha: N/L \rightarrow \Aut(\Q_m)$ via the isomorphisms $\Aut(L/K)\cong \U_m\cong \Aut(\Q_m)$ introduced in \Cref{SectionNumberTheory}.
Let $F_{N,L,K}=(\Q_m)^{\Imagen \alpha}$. Then we have a short exact sequence
	$$1\rightarrow L/K \cong \GEN{\zeta_m} \rightarrow N/K \rightarrow N/L\cong \Gal(\Q_m/F_{N,L,K})\rightarrow 1$$
which induces an element $\overline{f}\in H^2(N/L,\U(\Q_m))$. More precisely, choose a generator $hK$ of $L/K$ and a set of representatives $\{c_u : u\in N/L\}$ of $L$-cosets in $N$, we define $f(u,v)=\zeta_m^k$, if $c_uc_v=c_{uv}h^k$. This defines an element of $H^2(N/L,\U(\Q_m))$ because another election yields to another $2$-cocycle differing in a $2$-coboundary.
Associated to $\overline{f}$ one has the crossed product algebra
\begin{equation*}
\begin{split}
& A(N,L,K)=(\Q_m/F_{N,L,K},\overline{f}) = \oplus_{u\in N/L} t_u \Q_m, \\
& x t_u=t_u\alpha(x), \quad t_ut_v=t_{uv}f(u,v), \quad (x\in \Q_m, u,v\in N/K).
\end{split}
\end{equation*}

\begin{theorem}\label{SSPMetacyclic}
Let $G$ be a metabelian group, let $A$ be an abelian normal subgroup of $G$ containing $G'$ and let $M=C_G(A)$. Then the following statements hold:
\begin{enumerate}
	\item Every Wedderburn component of $\Q G$ is of the form $\Q Ge(G,L,K)$ for subgroups $L$ and $K$ of $G$ satisfying the following conditions:
	\begin{enumerate}
		\item\label{SSPMeta1} $L$ is a maximal element in the set $\{B\le G : M\le B \text{ and } B'\le K\le B\}$.
		\item\label{SSPMeta2} $L/K$ is cyclic.
	\end{enumerate}
	\item Suppose that $L$ and $K$ satisfy \eqref{SSPMeta1} and \eqref{SSPMeta2} and let $e=e(G,L,K)$, $N=N_G(K)$ and $n=[G:N]$. Then $(L,K)$ is a strong Shoda pair of $G$, $e$ is a primitive central idempotent of $G$, $\Deg(\Q Ge)=[G:L]$, $\{g\in G : ge=e\}=\Core_G(K)$, $\Q Ge\cong M_n(A(N,L,K))$ and $Z(\Q Ge)\cong F_{N,L,K}$.
\end{enumerate}
\end{theorem}
\begin{proof}
Clearly $M$ is a maximal abelian normal subgroup of $G$. Then (1) follows from \cite[Theorem~4.7]{OlivieridelRioSimon2004} and (2) follows from \cite[Corollary~3.6]{OlivieridelRioSimon2004} and \cite[Proposition~3.4]{OlivieridelRioSimon2004}. Alternatively, use Theorem~3.5.12, Corollary 3.5.11 and Theorem~3.5.5 in \cite{JespersdelRioGRG1}.
\end{proof}

\begin{example}\label{ExCounterexample}{\rm
Consider the groups $G$ and $H$ in \eqref{Counterexample}.
Then $G/G'\cong H/H'\cong C_p^2$ and hence $\Q G \widehat{G'}\cong \Q H \widehat{H'} \cong \Q(C_p)^2$.
Moreover, $L=\GEN{G',b}$ and $K=\GEN{b}$ satisfy conditions (a) and (b) of \Cref{SSPMetacyclic} both as subgroups of $G$ and $H$.
Then
	$$\Q G e(G,L,K) \cong \Q H e(H,L,K) \cong M_p(\Q(\zeta_p))$$
and we conclude that $\Q G \cong \Q(C_p)^2 \times M_p(\Q(\zeta_p)) \cong \Q H$.
}\end{example}

Suppose that $G$ is a finite metacyclic group and $G=AB$ is a metacyclic factorization of $G$. Then every Wedderburn component of $\Q G$ is of the form $\Q Ge(G,L,K)$ for $L$ and $K$ subgroups of $G$ with $A\subseteq L$ and satisfying the conditions of \Cref{SSPMetacyclic}. Then $L=\GEN{a,b^d}$, $N=N_G(K)=\GEN{a,b^n}$ where $[G:N]=n\mid d\mid [G:A]$. Moreover, $L/K$ is cyclic, say generated by $uK$, and normal in $N/K$ so that $(uK)^{b^nK}=u^xK$ and $(uK)^{\frac{d}{n}}=a^yK$ for some integers $x$ and $y$.
By \Cref{SSPMetacyclic},
\begin{equation}\label{WCMetacyclic}
A(N,L,K)\cong (\Q_m,\sigma_x,\zeta_m^y) =
\Q_m[\overline{u} \mid \zeta_m \overline{u} = \overline{u} \zeta_m^x, \overline{u}^k=\zeta_m^y],
\end{equation}
where $m=[L:K]$ and $\sigma_x$ is the automorphism of $\Q_m$ given by $\sigma_x(\zeta_m)=\zeta_m^x$. The center of $A(N,L,K)$ is $\Q_m^x$.

\section{Sketch of the proof of \Cref{Main}}\label{SectionSketch}

For the proof of \Cref{Main} we fix two finite metacyclic groups $G$ and $H$ such that the rational group algebras $\Q G$ and $\Q H$ are isomorphic. By \Cref{Invariants}, proving that $G$ and $H$ are isomorphic is equivalent to showing that $\INV(G)=\INV(H)$.
This can be  expressed by saying that $\INV(G)$ is determined by the isomorphism type of $\Q G$.
We will work most of the time with the group $G$ and will show how the different entries of $\INV(G)$ are determined by the isomorphism type of $\Q G$.
In the way we need to prove that other invariants of $G$ are determined by the isomorphism type of $\Q G$, which we abbreviate as ``determined by $\Q G$''.
For example, $|G|$ is determined by $\Q G$, because $|G|=\dim_{\Q} \Q G$. So, as $m^Gn^G=|G|$ to prove that $m^G$ and $n^G$ are determined by $\Q G$ it suffices to show that so is one of them.
By \cite[Theorem~A]{GarciadelRioNilpotent}, $\pi_G$, $\pi'_G$ and the isomorphism type of the Hall $\pi_G$-subgroups of $G$ are determined by $\Q G$.
Thus we can simplify the notation by setting $\pi=\pi_G=\pi_H$ and $\pi'=\pi'_G=\pi'_H$, and for every $p\in \pi$, the isomorphism type of the Sylow $p$-subgroup of $G$ is determined by $\Q G$.

It is easy to see that the kernel of the natural homomorphism $\Q G\rightarrow \Q(G/G')$ is the minimal ideal $I$ of $\Q G$ with $\Q G/I$ commutative. Therefore any isomorphism $\Q G\rightarrow \Q H$ maps that ideal of $\Q G$ to the corresponding ideal of $\Q H$, and hence $\Q(G/G') \cong \Q(H/H')$. Then, $G/G'\cong H/H'$, by the Perlis-Walker Theorem \cite{PerlisWalker1950}. This shows that the isomorphism type of $G/G'$ is determined by $\Q G$, and in particular so is $[G:G']$ and $|G'|=\frac{|G|}{[G:G']}$. Moreover, $m_{\pi'}=|G'|_{\pi'}$, by \cite[Lemma~3.1]{GarciadelRioClasification}, so that $m_{\pi'}$ is determined by $\Q G$.
Therefore $n_{\pi'}=\frac{|G|_{\pi'}}{m_{\pi'}}$ is determined by $\Q G$.
We collect this information for future use:

\begin{proposition}\label{piIgual}
If $G$ and $H$ are finite metacyclic groups with $\Q G\cong \Q H$, then $G/G'\cong H/H'$, $\pi_G=\pi_H$, $\pi'_G=\pi'_H$, $(m^G)_{\pi'}= (m^H)_{\pi'}$, $(n^G)_{\pi'}=(n^H)_{\pi'}$ and for every $p\in \pi_G$, the $p$-Sylow subgroups of $G$ and $H$ are isomorphic.
\end{proposition}

In \Cref{SectionRDetermined}, we prove that $R^G$ is determined by $\Q G$ and then $k^G$ is determined by $\Q G$, since $k^G=|R^G|$.
In \Cref{sDetermined}, we prove that $s^G$ is determined by $\Q G$ and, in \Cref{EpsilonDetermined}, that so is $\epsilon^G$.
In \Cref{mnrDetermined}, we prove that $m^G$, $n^G$ and $r^G$ are determined by $\Q G$.
Finally we prove that $\Delta^G$ is determined by $\Q G$ in \Cref{DeltaDetermined}.
Summarizing,  $\INV(G)=(m^G,n^G,s^G,\Delta^G)=(m^H,n^H,s^H,\Delta^H)=\INV(H)$ and hence $G\cong H$ by \Cref{Invariants}.
The idea in all the cases is to look at some simple components of $\Q G$ which encode the corresponding invariant. We will illustrate this with one example at the beginning of each section. To compute the Wedderburn decomposition of the rational group algebras appearing in the examples we use GAP package Wedderga \cite{Wedderga4.10.5}.

Along the paper we fix a minimal metacyclic factorization $G=\GEN{a}\GEN{b}$ of $G$.
We also fix the notation $m=m^G$, $n=n^G$, $s=s^G$, $\Delta=\Delta^G$, $r=r^G$,
$\epsilon=\epsilon^G$, $k=k^G$, $R=R^G$, and $m'$ is as defined in \eqref{m'}.
Then $\Inn_G(\GEN{a})$ is cyclic, say generated by $\gamma$, and $G$ is given by the presentation in \eqref{Presentation}.
By \cite[Lemma~3.1]{GarciadelRioClasification}
\begin{equation}\label{Estructura}
    G'_{\pi'}=\GEN{a_{\pi'}} \qand G=\GEN{a_{\pi'}}\rtimes \left( \GEN{b_{\pi'}} \times \prod_{p\in \pi} \GEN{a_p,b_p} \right).
\end{equation}

\begin{remark}\label{Observations}
The following observations will be useful in the remainder of the paper:
\begin{itemize}
\item $k^G=|R^G|$.
\item $2\not\in \pi'$ and hence $m_{\pi'}$ is odd. In fact, $\min(\pi(G))\in \pi_G$. Indeed, if $p=\min(\pi(G))$, then for every $q\in \pi(G)\setminus \{q\}$, $A_p$ does not have automorphisms of order $q$. Thus $[B_{p'},A_p]=1$ and hence $A_{p'}B_{p'}$ is a $p$-normal complement of $G$.
\item $[b,a]\in \GEN{a^2}$, because if $m$ is even then $a^b=a^x$ with $x$ odd.
\item $\epsilon=-1$ if and only if $4\mid m$ and $[\GEN{a_2}:G'_2]$ is not multiple of $4$.
\item If $K$ is a subgroup of $G$ then $A\cap K=A\cap \Core_G(K)$, as $\GEN{a}$ is normal in $G$.
\item Every primitive central idempotent of $\Q G$ is of the form $e(G,L,K)$ with $L$ and $K$ satisfying conditions \eqref{SSPMeta1} and \eqref{SSPMeta2} of \Cref{SSPMetacyclic} for $A=\GEN{a}$.
\end{itemize}
\end{remark}

\section{$\Q G$ determines $R^G$}\label{SectionRDetermined}

In this section we prove that $\Q G$ determines $R$. For that we will look at the simple components of maximal degree of $\Q G$ with center $F$ satisfying the following conditions:
\begin{itemize}
	\item[\mylabel{A1}{(A1)}] $F$ can be embedded in a subfield of $\Q_{m_{\pi'}}$.
	\item[\mylabel{A2}{(A2)}] The only roots of unity of $F$ are $1$ and $-1$.
\end{itemize}

We illustrate this with the following example:

\begin{example}{\rm
Consider the following group
	$$G=\GEN{a,b \mid a^{15}=b^{12}=1, a^b=a^2}.$$
For this group $m_{\pi'}=15$ and $R^G=\GEN{2}_{15}$. The latter can be discovered looking at $\Q G$ as follows: Consider the Wedderburn decomposition of $\Q G$, which we compute with command  \texttt{WedderburnDecompositionInfo} of the package Wedderga:
$$\Q G = \mathbf{2\Q} \oplus 2\Q_3\oplus \Q_4 \oplus \Q_{12} \oplus \mathbf{M_2(\Q)}\oplus \mathbf{(\Q_3,\sigma_{-1},-1)} \oplus 2M_2(\Q_3) \oplus \mathbf{M_4(\Q)} \oplus M_4(\Q_3)\oplus \mathbf{M_4(\Q_{15}^2)} \oplus M_4(\Q_{15}^4).$$
Recall that for $\Q_m^r$ denotes the subfield of $\Q_m$ of elements fixed by the automorphism mapping $\zeta_m$ to $\zeta_m^r$.
The centers of the components which satisfy conditions \ref{A1} and \ref{A2} are all contained in $\Q_{15}^2$.
The Galois correspondent of this field, as subfield of $\Q_{15}$ is precisely $\GEN{2}_{15}$, the desired $R^G$.
Notice that the last component of the previous decomposition of $\Q G$ does not satisfy condition \ref{A2} because $\Q_{15}^4$ contains a third root of unity.
}
\end{example}

Along this section we use the following notation:
	$$L_0=C_G(G'_{\pi'}), \quad k=|R|=[G:L_0], \quad L_1=\GEN{a,b^{2k}} \qand F_0=(\Q_{m_{\pi'}})^R.$$
As $a\in L_0$, $L_0=C_G(a_{\pi'})=\GEN{a,b^k}$.
Moreover, $L_1\subseteq L_0$, $[L_0:L_1]\le 2$ and $L_0=L_1$ if and only if $[L_0:\GEN{a}]$ is odd.

\begin{lemma}\label{SSPabk}
	\begin{enumerate}
		\item\label{F0A1A2} $F_0$ satisfies \ref{A1} and \ref{A2}.
		\item\label{Deg=k} Let $K=\GEN{a_{\pi},b^k}$.
		Then $(L_0,K)$ is a strong Shoda pair of $G$ and if $S=\Q Ge(G,L_0,K)$,
		then $\Deg(S)=k$ and $Z(S)\cong F_0$.

		\item\label{Deg=2k} Suppose that $k$ is odd, $\epsilon=-1$ and $a_2^2\not\in \GEN{b^{4}}$, and let
		$$\overline{K}=\begin{cases}
		\GEN{a_{\pi}^4,b^{2k}}, & \text{if } a_2^2\not\in \GEN{b^{2}}; \\
		\GEN{a_{\pi}^4,b^{4k}}, & \text{otherwise}.
		\end{cases}$$
		Then $(L_1,\overline{K})$ is a strong Shoda pair of $G$ and if $\bar{S}=\Q  Ge(G,L_1,\overline{K})$, then $\Deg(\bar{S})=2k$ and $Z(\bar{S})\cong F_0$.
	\end{enumerate}
\end{lemma}

\begin{proof}
\eqref{F0A1A2} Clearly, $F_0$ satisfies \ref{A1} and as $2\not\in \pi'$, $F_0$ does not have a fourth root of unity. If $F_0$ does not satisfy condition \ref{A2}, then it contains a root of unity of order $p$ with $p$ an odd prime. Then $p$ divides $m_{\pi'}$, so that $p\in \pi'$ and $G'_{\pi'}\cap Z(G)=\GEN{a_{\pi'}}\cap Z(G)$ has an element of order $p$, in contradiction with \cite[Lemma~3.1]{GarciadelRioClasification}.

\eqref{Deg=k} Clearly $L_0/K$ is cyclic generated by $a_{\pi'}K$, $\GEN{a_{\pi'}}\cap K=1$, by \eqref{Estructura}, and $[g,a_{\pi'}]\in \GEN{a_{\pi'}}$ for every $g\in G$.
Then $[g,a_{\pi'}]\not\in K$ for every $g\in G\setminus L_0$.
Moreover, $K\unlhd G$, because $a_{\pi}^b\in \GEN{a_{\pi}}$ and $[b^k,a]=[b^k,a_{\pi}]\in \GEN{a_{\pi}}$.
This proves that $(L_0,K)$ satisfies the hypothesis of \Cref{SSPMetacyclic} and hence $(L_0,K)$ is a strong Shoda pair of $G$.
By \Cref{SSPMetacyclic}, $\Deg(S)=[G:L_0]=k$ and as $[L_0:K]=m_{\pi'}$ and $K$ is normal in $G$, the center of $S$  is isomorphic to $F_0$.

\eqref{Deg=2k} Suppose now that the conditions of \eqref{Deg=2k} hold.
The assumption $\epsilon=-1$ implies that $4\mid m$ and $a^b=a^t$ with $t\equiv -1 \mod 4$, or equivalently $\GEN{a}$ has a non-central element of order $4$.
In particular, $|b\GEN{a}|$ is even.
Since $k$ is odd, we have that $[G:L_1]=2k$.
Using that $[b,a_2]\in \GEN{a_2^2}$ and $[b^k,a_{\pi'}]=1$, it follows that $[b^{2k},a]\in \GEN{a_{\pi}^4}$, and therefore $\overline{K}\unlhd G$ and $L_1'\subseteq \overline{K}$.

We claim that $L_1/\overline{K}$ is cyclic generated by $a\overline{K}$.
This is clear from the definition of $\overline{K}$, if $a_2^2\not\in \GEN{b^{2k}}$.
Otherwise, as $a_2^2 \not\in \GEN{b^{4k}}$ by hypothesis, $a_2^2\in \GEN{b^{2k}}\setminus \GEN{b^{4k}}$ and hence $a_2^2=b_2^{2ki}$ for some odd integer $i$.
Therefore $b_2^{2k}\in \GEN{a}$. As $b_{2'}^k\in \overline{K}$ it follows that $b^{2k}\in \GEN{a,\overline{K}}$. Then $L_1/\overline{K}=\GEN{a\overline{K}}$, as desired.

In order to prove that $(L_1,\overline{K})$ satisfies the conditions of \Cref{SSPMetacyclic} it remains to prove that if $B$ is a subgroup of $G$ containing $L_1$ properly, then $B'\not\subseteq \overline{K}$.
Assume otherwise. Then there is $g\in G\setminus L_1$ with $[L_1,g]\subseteq \overline{K}$.
If $g\not\in \GEN{a,b^k}$, then $1\ne [a_{\pi'},g]\in \overline{K}\cap \GEN{a_{\pi'}}=1$, a contradiction.
Thus $g\in \GEN{a,b^k}$ and $\GEN{L_1,g}=\GEN{a,b^k}$, so that $[b^k,a]\in \overline{K}$ and therefore $[b^k,a_2]\in \overline{K}$.
On the other hand, $[b^k,a_2]=a_2^{t^k-1}$ and $v_2(t^k-1)=1$, because $k$ is odd
(cf. \Cref{PropEse}.\eqref{vpRm-1}).
Then $a_2^2\in \overline{K}$. Moreover, $\GEN{a_{\pi},b}=\GEN{b_{\pi'}}\times \prod_{p\in \pi} \GEN{a_p,b_p}$, a nilpotent group.
Thus $\overline{K}$ is nilpotent and hence $a_2^2\in \overline{K}_2$.
Suppose first that $a_2^2\not\in \GEN{b^{2k}}$.
Then $\overline{K}_2=\GEN{a_2^4,b_2^{2k}}$ and hence $a_2^2=a_2^{4i}b_2^{2kj}$ for some integers $i,j$.
Hence $\GEN{a_2^2}=\GEN{a_2^{2-4i}}=\GEN{b_2^{2kj}}\subseteq \GEN{b_2^{2k}}$, yielding a contradiction.
Thus $a_2^2\in \GEN{b^{2k}}$, and as $k$ is odd we have that
$a_2^2\in \GEN{b_2^2}\setminus \GEN{b_2^4}$, by assumption.
Then $\GEN{a_2^2}=\GEN{b_2^2}$ and, in particular, $a_2^2$ commutes with $b$. This implies that $a_2$ has order $4$, because $\epsilon=-1$.
Thus $a_2^2=b_2^2$, so that the two generators $a_{\pi}^{4k}$ and $b^{4k}$ of $\overline{K}$ have odd order. As $\overline{K}$ is nilpotent, we deduce that $\overline{K}$ is a $2'$-group.
This yields a contradiction with the fact that $a_2^2$ is an element of order $2$ in $\overline{K}$.

Then $(L_1,\overline{K})$ satisfies the conditions of \Cref{SSPMetacyclic} and hence it is a strong Shoda pair of $G$.
By \Cref{SSPMetacyclic}, $\Deg(\bar{S})=[G:L_1]=2k$ and as $[L_1:\overline{K}]=4m_{\pi'}$ and $L_1/\overline{K}$ is generated by $a\overline{K}$, the center  of $\bar{S}$ is $F=(\Q_{4m_{\pi'}})^{\Res_{4m_{\pi'}}(T_G(\GEN{a})}$.
Moreover,  $F_0=F\cap \Q_{m_{\pi'}} \subseteq F\subseteq \Q_{4m_{\pi'}}$ and $[\Q_{4m_{\pi'}}:F_0]=[\Q_{4m_{\pi'}}:\Q_{m_{\pi'}}]\;[\Q_{m_{\pi'}}:F_0]=2k=[G:\bar L_0]=[\Q_{4m_{\pi'}}:F]$ and therefore $F=F_0$.
\end{proof}

In \Cref{SSPabk} we have encountered some Wedderburn components of $\Q G$ with center satisfying conditions \ref{A1} and \ref{A2}.
In order to analyze which other Wedderburn components of $\Q G$ satisfy the same properties we need the following two lemmas. In their proofs we often use that if $(L,K)$ is a strong Shoda pair of $G$ and $e=e(G,L,K)$, then $\{g\in G : ge=e\}=\Core_G(K)$ (cf. \Cref{SSPMetacyclic}).

\begin{lemma}\label{PCIOdd-A2}
Let $(L,K)$ be a strong Shoda pair of $G$ with $a\in L$.
Let $C=\Core_G(K)$ and $S=\Q Ge(G,L,K)$.
Let $p$ be an odd prime such that the center of $S$ does not have elements of order $p$.
Then $b_p^k\in C$.
If, moreover, $p\in \pi$, then $a_p\in C$.
\end{lemma}

\begin{proof}
Let $e=e(G,L,K)$ and  $F=Z(\Q Ge)$.
Since $S$ is generated by $Ge$ as $\Q$-algebra, the assumption implies that $Ge$ does not have central elements of order $p$ and hence if $g$ is a $p$-element of $G$ with $ge\in Z(Ge)$, then $g\in C$.
If $p\in \pi'$, then $b_p^k\in Z(G)$ and hence $b_p^k\in C$.
Suppose that $p\in \pi$. Then $Ge$ has a $p$-normal complement because so does $G$.
If $a_p\not\in C$, then $a_pe\ne e$, so  $p\in \pi_{Ge}$, and hence $\GEN{a}e$ has an element of order $p$ which is central in $Ge$, which is not possible.
Thus $a_p\in C$.
Then $b_p^ke\in Z(Ge)$ and hence $b_p^k\in C$.
\end{proof}

Of course every field of characteristic $0$ has a root of unity of order 2 and therefore a similar lemma for $p=2$ makes no sense. However we have the following:

\begin{lemma}\label{PCI2-A1}
Let $(L,K)$ be a strong Shoda pair of $G$ with $a\in L$.
Let $C=\Core_G(K)$ and $S=\Q Ge(G,L,K)$.
Suppose that $Z(S)$ can be embedded in $\Q_t$ for some odd integer $t$.
Then
	\begin{enumerate}
		\item $a_2^4,b_2^{4k}\in C$ and $b_2^{2k}\in L$.
		\item If $\epsilon=1$, then $a_2^2\in C$.
		\item If $a_2^2 \in C$ or $k$ is even, then $b_2^{2k}\in C$ and $b_2^k\in L$.
		\item If $2\nmid k$, $\epsilon=-1$ and $a_2^2\in \GEN{b^{2k}}\setminus C$, then $\GEN{a_2,b_2}$ is the quaternion group of order $8$.
	\end{enumerate}
\end{lemma}

\begin{proof}
We use the same notation as in the proof of \Cref{PCIOdd-A2}.
Now $F\cong (\Q_h)^T$ with $h=[L:K]$ and $T$ a cyclic subgroup of $\U_h$, and by assumption $F$ can be embedded in $\Q_t$ with $t$ an odd integer. Then $(\Q_h)^T\subseteq \Q_t$.
The latter implies that $F$ does not have elements of order $4$ and hence neither does $Z(Ge)$. As in the proof of \Cref{PCIOdd-A2}, this implies that if $g\in G$ with $ge\in Z(Ge)$, then $g^4 \in C$.

(1) Suppose that $a_2^4\not\in C$.
As $\GEN{a_2^4}$ is normal in $G$, this implies that $a_2^4\not\in K$ and hence $h$ is multiple of $8$.
Therefore $\Q_h$ contains a primitive $8$-th root of unity.
As $\Gal(\Q_8/\Q)$ is not cyclic, it follows that $(\Q_h)^T$ contains a subfield of $\Q_8$ other than $\Q$ and this is not compatible with $(\Q_h)^T\subseteq \Q_t$, because $t$ is odd.
Therefore $a_2^4\in C$.
Then $b_2^{2k}e\in Z(Ge)$ and therefore $b_2^{4k}\in C$.
Moreover, as $b_2^{2k}e\in Z(Ge)$, $[b_2^{2k},a]\in C\subseteq K$ and hence $\GEN{L, b^{2k}_2}/K$ is an abelian subgroup of $N_G(K)/K$. Thus  $b^{2k}\in L$, since $L/K$ is maximal abelian in $N_G(K)/K$.

(2)	Suppose that $a_2^2\not\in C$. By (1) the order of $a_2e$ is $4$ and the hypotheses imply that $a_2e\not\in Z(Ge)$ so that $a_2^be=a_2^{-1}e$.
Hence $4\mid |a|$ and $\GEN{a}$ has an element of order $4$ which is not central in $G$.
Thus $\epsilon=-1$.

(3)	Suppose that $a_2^2\in C$ or $k$ is even.  Then $b_2^{k}e$ is central in $Ge$, so that $b_2^{2k}\in C$ and $[b_2^k,a]\in C\subseteq K$. Therefore $b_2^k\in L$, because $L/K$ is maximal abelian in $N_G(K)$.

(4)	Suppose that $2\nmid k$, $\epsilon=-1$ and $a_2^2\in \GEN{b^{2k}}\setminus C$. Then $a_2^2$ commutes with $b$ and as $\epsilon=-1$ the order of $a_2^2$ is $2$, i.e. $a_2$ has order $4$ and $a_2^b=a_2^{-1}$.
	Furthermore, as $a_2^2\in \GEN{b_2^{2k}}\setminus C$ but $b_2^{4k}\in C$, it follows that $a_2^2\in\GEN{b_2^{2k}}\setminus \GEN{b_2^{4k}}$ and hence $a_2^2=b_2^{2k}=b_2^2$.
	This shows that $\GEN{a_2,b_2}$ is the quaternion group of order $8$.
\end{proof}

\begin{lemma}\label{SSPMax}
	Let $S$ be a Wedderburn component of $\Q G$ with center $F$.
	Suppose that $F$ satisfies conditions \ref{A1} and \ref{A2} and the degree of $S$ is maximum among the degrees of the Wedderburn components of $\Q G$ with center satisfying \ref{A1} and \ref{A2}. Then $F$ can be embedded in $F_0$ and $$\Deg(S)=\begin{cases} 2k, & \text{if } \epsilon=-1, 2\nmid k \text{ and } a_2^2\not\in \GEN{b^{4}}; \\ k, & \text{otherwise}.\end{cases}$$
\end{lemma}

\begin{proof}
	By \Cref{SSPMetacyclic}, $S=\Q Ge(G,L,K)$ for a strong Shoda pair $(L,K)$ of $G$ with $a\in L$ and $\Deg(S)=[G:L]$. Let  $C=\Core_G(K)$.
	Observe that $(L,K)$ satisfies the hypothesis of \Cref{PCIOdd-A2} for every $p\in \pi\setminus \{2\}$ and the hypothesis of \Cref{PCI2-A1}, because $2\not\in \pi'$, since the minimal prime dividing $|G|$ is always in $\pi$, and hence $m_{\pi'}$ is odd.
	Therefore $a_{\pi}^4,b^{4k}\in C$ and $b^{2k}\in L$.
	The latter implies that $L_1\subseteq L$.
	Therefore $[G:L]$ divides $[G:L_1]$ and
 $$[G:L_1]=\begin{cases} 2k, & \text{if } 2\mid [L_0:\GEN{a}]; \\ k, & \text{otherwise}.\end{cases}$$
	In view of \Cref{SSPabk}, the maximality of $\Deg(S)$ implies that $L$ is either $L_0$ or $L_1$.
	Therefore $\Deg(S)=[G:L]\in\{k,2k\}$. By \Cref{SSPabk}(\ref{Deg=2k}), if $2\nmid k$, $\epsilon=-1$ and $a_2^2\not\in\GEN{b^{4k}}$, then $\Deg(S)=2k$.
	Conversely, suppose that $\Deg(S)=2k$. Then $b_2^k\not\in L$ and hence $a_2^2\not\in C$ and $2\nmid k$, by \Cref{PCI2-A1}(3). Therefore $\epsilon=-1$ by \Cref{PCI2-A1}(2) and $a_2^2\not\in \GEN{b^{4k}}$, by \Cref{PCI2-A1}(1).
This proof the statement about $\Deg(S)$.

The following observations will be relevant for the remainder of the proof.
By \Cref{Observations}, $[b,a]\in \GEN{a^2}$, and moreover $a_2^2\in K$ if and only if $a_2^2\in C$. By \eqref{Estructura}, $L_0=\GEN{a_{\pi'}}\times \GEN{b_{\pi'}^k}\times \prod_{p\in \pi} \GEN{a_p,b_p^k}$ and in particular $L_0$ is nilpotent.

We consider the following subgroup of $G$:
	$$D=\begin{cases}
	\GEN{a_{\pi}^2,b^{2k}}, & \text{if } a_2^2\in K; \\
	\GEN{a_{\pi}^4,b^{2k}}, & \text{if } a_2^2\not\in K \text{ and } 2\mid k; \\
	\GEN{a_{\pi}^4,b^{4k}}, & \text{otherwise}.\end{cases}$$
By \Cref{PCIOdd-A2} and \Cref{PCI2-A1}, we have that $D\subseteq C$.
Moreover, $D$ is normal in $G$ because $\GEN{a_{\pi}}$ is normal in $G$ and, using \Cref{PropEse}.\eqref{vpRm-1}, it is easy to see that given an integer $j$, $[b^{jk},a]\in \GEN{a_{\pi}^j}$ and, if $k$ is even, then $[b^{jk},a]\in \GEN{a_{\pi}^{2j}}$.
As $\GEN{a_{\pi\setminus \{2\}},b_{2'}^k} \subseteq D\subseteq C \subseteq K\subseteq L \subseteq L_0$, the Hall $2'$-subgroup of $L/D$ is cyclic generated by $a_{2'}D$ and therefore the Hall $2'$-subgroup of $L/K$ is generated by $a_{2'}K$.
Furthermore $(L/K)_2$ is a cyclic quotient of $L/D$.

We consider separately four cases:

\textbf{Case 1}. Suppose that $a_2^2\in K$.

Then the Sylow $2$-subgroup of $L/D$ is elementary abelian of order at most 4 and hence the Sylow $2$-subgroup of $L/K$ has order at most $2$.
Thus $[L:K]\in\{l,2l\}$ with $l\mid m_{\pi'}$.
If $g=a^ib^{kj}$ with $i$ and $j$ integers, then $[a,g]\in \GEN{a_{\pi}^2}\subseteq D$ and
$[b,g]=[b,a^i]=a^{2ix}\equiv a^{2ix}b^{2jkx}\equiv g^{2x} \mod D$ for some integer $x$.
This shows that every subgroup of $L$ containing $D$ is normal in $G$.
In particular, $K$ is normal in $G$.
Recall that $\gamma$ is a generator of $\Inn_G(\GEN{a})$.
 By \Cref{SSPMetacyclic}, $F\cong \Q_l^{\Res_l(\gamma)}\subseteq (\Q_{m_{\pi'}})^R=F_0$, as desired.

	\textbf{Case 2}. Suppose that $a_2^2\not\in K$ and $2\mid k$.

	Then $L/D$ has a cyclic normal Hall $2'$-subgroup (generated by $a_{2'}D$), and its Sylow $2$-subgroups have order dividing $8$ and exponent $4$. If the Sylow $2$-subgroup of $L/D$ is not abelian, then $a_2^2\in L'D\subseteq K$, in contradiction with the hypothesis $a_2^2\not\in K$.
	Thus the Sylow $2$-subgroup of $L/D$ is isomorphic to $C_4$ or $C_4\times C_2$. As $k$ is even, $[b^k,a]\in \GEN{a_\pi^4}\subseteq D$. Hence, arguing as in the previous case it follows that every subgroup of $L$ containing $D$ is normal in $G$. In particular $K$ is normal in $G$ and as $a_2^2\not\in K$ it follows that $a_{\{2\}\cup \pi'}K$ is a generator of $L/K$ and $F\cong \Q_l^{\Res_{l}(\gamma)}$ with $l\mid 4m_{\pi'}$.
	As, by assumption, $F$ can be embedded in $\Q_{m_{\pi'}}$, so does $\Q_l^{\Res_{l}(\gamma)}$ and as both $\Q_{m_{\pi'}}$ and $\Q_l^{\Res_{l}(\gamma)}$ are Galois extensions of $\Q_l$ it follows that $\Q_l^{\Res_{l}(\gamma)}\subseteq (\Q_{m_{\pi'}})^R=F_0$. Thus $F$ can be embedded in $F_0$, as desired.

	\textbf{Case 3}. Suppose that $a_2^2\in \GEN{b^{2k}}\setminus K$ and $2\nmid k$.

	Then $\epsilon=-1$, by \Cref{PCI2-A1}.(2) and
    $\GEN{a_2,b_2}\cong Q_8$, by \Cref{PCI2-A1}.(4). The latter also implies that $a_2^2\not\in \GEN{b^4}$ and hence the conditions of \Cref{SSPabk}.\eqref{Deg=2k} hold. Therefore, $L=\GEN{a,b^{2k}}=\GEN{a,D}$, so that $L/D$ is cyclic generated by $aM$.
	Moreover, the $\pi$-parts of $D$ and $K$ coincide because $a_2^2\not\in K$.
	Therefore $[L:K]=4l$ with $l$ a divisor of $m_{\pi'}$. Now it is easy to see that $K$ is normal in $G$ and arguing as in the previous case we deduce that $F$ is isomorphic to a subfield of $(\Q_{m_{\pi'}})^R=F_0$.

	\textbf{Case 4}.
	In the remaining cases $a_2^2\not\in (\GEN{b^{2k}}\cup K)$ and $2\nmid k$.

As in the previous case $\epsilon=-1$ and $L=\GEN{a,b^{2k}}$.
On the other hand by the definition of $D$ and the assumption $a_2^2 \not\in \GEN{b^{2k}}$ we have that the Sylow $2$-subgroup of $L/D$ is $\GEN{a_2M}\times \GEN{b_2^2M}\cong C_4\times C_2$.
We claim that $K\subseteq \GEN{a^2, b^{2k}}$ and $K$ is normal in $G$.
If the former fails, then $K$ contains an element of the form $g=a_2b^{2kl}$.
Then, $g^2 \equiv a_2^2 b^{4kl} \equiv a_2^2 \mod D$, so $a_2^2\in K$, in contradiction with the hypothesis. Then, $K\subseteq \GEN{a^2, b^{2k}}$. In order to prove that $K$ is normal in $G$ take $g=a^{2i}b^{2kj}$ with $i$, $j$ integers. Then, $[a,g] = [a_\pi, b^{2kj}] \in \GEN{a_{\pi}^4}\subseteq D\subseteq K$ and there is an integer $x$ such that $[b,g] = [b, a^{2i}] = a^{4ix} \equiv a^{4ix}b^{4kjx}\equiv (a^{2i}b^{2kj})^{2x} \equiv g^{2x} \mod D$.
So $[a,g], [b,g]\in \GEN{g, D}\subseteq K$ and then $K$ is a normal subgroup of $G$.
As $L=\GEN{a,b^{2k}}$, $L/K$ is cyclic, the $2'$-Hall subgroup of $L/K$ is generated by $a_{2'}K$ and $a_2^2\not\in K$, the $2$-Sylow subgroup of $L/D$ is isomorphic to $C_4\times C_2$ with $a_2M$ of order $4$, and $K\subseteq \GEN{a^2,b^{2k}}$. It follows that $L/K$ is generated by $\GEN{aK}$ and $[L:K]$ divides $4m_{\pi'}$.
Then we can argue as in the previous cases.
\end{proof}

We are ready to prove the main result of this section:

\begin{proposition}\label{RDetermined}
	If $G$ and $H$ are finite metacyclic groups with $\Q G\cong \Q H$, then $R^G=R^H$ and hence $k^G=k^H$.
\end{proposition}

\begin{proof}
Suppose that $\Q G\cong \Q H$ and let $L_G=(\Q_{m_{\pi'}})^{R^G}$ and $L_H=(\Q_{m_{\pi'}})^{R^H}$.
Observe that we are abusing the notation, by considering $R^G$ and $R^H$ as subsets of $\Gal(\Q_{m_{\pi'}}/\Q)$, rather than as subsets of $\U_{m_{\pi'}}$, i.e. we are considering the standard isomorphism $\U_{m_{\pi'}}\cong \Gal(\Q_{m_{\pi'}})$ as an equality.
By \Cref{SSPabk} and \Cref{SSPMax}, among the Wedderburn components of $\Q G$ (respectively, $\Q H$) whose center satisfy conditions \ref{A1} and \ref{A2} there is one with maximum degree and center $L_G$ (respectively, $L_H$). Denote those Wedderburn components $S_G$ and $S_H$. As $\Q G\cong \Q H$, $\Deg(S_G)=\Deg(S_H)$.
Then, by \Cref{SSPMax}, $Z(S_G)\cong L_G\subseteq L_H$	and
	$Z(S_H)\cong L_H \subseteq L_G$.
	As $L_G$ and $L_H$ are Galois extensions of $\Q$ it follows that $L_G=L_H$. Then $R^G=\Gal(\Q_{m_{\pi'}}/L_G)=\Gal(\Q_{m_{\pi'}}/L_H)=R^H$, by Galois Theory.
\end{proof}

\section{$\Q G$ determines $s^G$ and $\epsilon^G$}\label{SectionEpsilonDetermined}

In this section we first prove that $s^G$ is determined by $\Q G$ and latter that so is $\epsilon^G$. Recall that we have fixed notation $s=s^G, \epsilon=\epsilon^G, m=m^G, \dots$
In the proof that $s$ is determined by $\Q G$ we work prime by prime, so we fix a prime $p$ and we will prove that $s_p$ is determined by $\Q G$.
As $s_{\pi'}=m_{\pi'}=|G'_{\pi'}|$ and, by \Cref{piIgual}, $G'$ and $\pi'$ are determined by $\Q G$, if $p\in \pi'$, then $s_p$ is determined by $\Q G$. Thus, we may assume that $p\in \pi$.
Recall that $G_p=\GEN{a_p,b_p}$ is a Sylow $p$-subgroup of $G$.
Let
	$$\INV(G_p)=(p^\mu,p^\nu,p^\sigma,\GEN{e+p^\rho}_{p^\mu}).$$
Then $\mu,\nu,\sigma,\rho$ and $e=\epsilon^{G_p}$ satisfy the conditions of \Cref{conditionsA-D}, $G_p$ is given by the following presentation
    $$G_p=\GEN{c,d \mid c^{p^\mu}=1, c^d = c^{e+p^\rho}, d^{p^\nu}=c^{p^\sigma}},$$
and $G_p=\GEN{c}\GEN{d}$ is a minimal metacyclic factorization of $G_p$.
As $p\in \pi$, by \Cref{piIgual}, the isomorphism type of $G_p$ is determined by $\Q G$ and hence so are $\mu,\nu,\sigma,\rho$ and $e$.
Observe that $T_G(\GEN{c})=T_{G_p}(\GEN{c})=\GEN{e+p^{\rho}}_{p^\mu}$, since $p\in \pi$.

To prove that $s_p$ is determined by $\Q G$ we compare, in the proof of \Cref{DeGpAG}, the following metacyclic factorizations of a Sylow subgroup of $G$:
	$$G_p=\GEN{c}\GEN{d}=\GEN{a_p}\GEN{b_p}.$$
	Ultimately we prove that $s_p=p^{\sigma}$. As $\sigma$ is determined by $\Q G$, so is $s_p$. A byproduct of this analysis is that if $e=1$ then $\epsilon=1$. However, if $e=-1$, then $p=2$ and $\epsilon$ might be $1$ or $-1$. The latter only happens in few cases and the idea is to consider two groups $G$ and $H$ in such special situation with $\epsilon^G=1$ and $e=\epsilon^{G_p}=\epsilon^{H_p}=\epsilon^H=-1$ and prove that the Wedderburn decompositions of $\Q G$ and $\Q H$ differ in the number of simple components satisfying the following conditions:
	\begin{itemize}
		\item[\mylabel{B1}{(B1)}] $\Deg(S)=k$
		\item[\mylabel{B2}{(B2)}] The center of $S$ does not contain roots of unity of order $p$ for every $p\in \pi\setminus \{2\}$.
	\end{itemize}
	We illustrate this with the following example:

\begin{example}{\rm
Consider the following groups:
	$$G=\GEN{a,b \mid a^{168}=b^6=1, a^b=a^{157}}, \quad
	H=\GEN{a,b \mid a^{84}=1, b^{12}=a^{42}, a^b=a^{31}}.$$
	Observe that if $c=a^{21}$ and $d=b^3$, with $a,b\in G$, and $c_1=a^{21}$ and $d_1=b^3$, with $a,b\in H$ then
	$$G_2=\GEN{c,d \mid c^4=1, d^4=c^2, c^d=c^{-1}}\cong
	H_2=\GEN{c_1,d_1 \mid c^8=d_1^2=1, c_1^{d_1}=c_1^5}.$$
	Thus $\epsilon^G=1$, while $e=\epsilon^H=-1$.
Moreover
\begin{eqnarray*}
    \Q G &=& 4\Q \oplus 16\Q_3 \oplus 2\Q_4 \oplus 8 \Q_{12} \oplus M_2(\Q_4)
    \oplus  4 M_2(\Q_{12})  \oplus \\
    & & 2(\Q(\zeta_{42}),\sigma_{31}, -1) \oplus 2M_6(\Q_3) \oplus M_6(\Q_{84}^{25}) \oplus \mathbf{2 M_6(\Q) \oplus 2(\Q_7,\sigma_3,-1) \oplus M_6(\Q_{28}^9)} \\
    \Q H &=& 4\Q \oplus 16\Q_3 \oplus 2\Q_4 \oplus 8 \Q_{12} \oplus 2M_2(\Q_4)
    \oplus 4 M_2(\Q_{12}) \oplus  \\
    & & 2M_6(\Q_3)  \oplus M_6(\Q_{12}) \oplus M_6(\Q_{168}^{61})\oplus \mathbf{2 M_6(\Q) \oplus M_6(\Q_{56}^5)}.
\end{eqnarray*}
	In this case $k=6$ and $\pi=\{2,3\}$, so we have to count simple components of degree $6$ not containing a primitive third root of unity. $\Q G$ has five of them, namely two copies of $M_6(\Q)$, two copies of $(\Q_7,\sigma_3,-1)$ and one copy of $M_6(\Q_{28}^{\GEN{9}})$, while, $\Q H$ has only tree copies satisfying the conditions, namely two copies of $M_6(\Q)$ and one copy of $M_6(\Q_{56}^{\GEN{5}})$.}
\end{example}

\begin{lemma}\label{DeGpAG}
	\begin{enumerate}
		\item\label{mnnu} $m_pn_p=p^{\mu+\nu}$ and $p^\mu\mid m_p$.
		\item\label{m=pmu} $p^\mu=m_p$ if and only if $p^\nu=n_p$.
		\item\label{rp} If $p^\mu=m_p$ then $p^\rho=r_p$.
		\item\label{sp} $s_p=p^{\sigma}$.
		\item\label{e1} Suppose that $e^{p-1}=1$.
		Then
		\begin{enumerate}
			\item\label{e12} If $p=2$, then $\epsilon=1$.
			\item\label{e1mr} $\frac{m_p}{r_p}=p^{\mu-\rho}$ and $\exp(G_p)=\frac{m_pn_p}{s_p}=p^{\mu+\nu-\sigma}$.
			\item\label{e1mnomui} If $m_p\ne p^{\mu}$, then $k_p>1$, $\mu\ne 0$, $\rho=\sigma$ and  $k_ps_p>n_p$.
		\end{enumerate}
		\item\label{e-1} Suppose that $e=-1$ and $p=2$. Then
		\begin{enumerate}
			\item\label{e-1eps-1} $\epsilon=-1$ if and only if $m_2=2^{\mu}$.
			\item\label{e-1eps1} If $\epsilon=1$, then $2=n_2=k_2<2^\nu$, $\sigma=1$, $\mu=2$, $m_2=2^{\nu+1}$ and $r_2=2^\nu$.
		\end{enumerate}
	\end{enumerate}
\end{lemma}

\begin{proof}
Recall that $G=\GEN{a}\GEN{b}$ and $G_p=\GEN{c}\GEN{d}$ are minimal metacyclic factorizations. Moreover, $G_p=\GEN{a_p}\GEN{b_p}$ is a metacyclic factorization, $|a|=m$, $[G:\GEN{b}]=s$, $|c|=p^\mu$ and $[G_p:\GEN{d}]=p^\sigma$.
Since $G_p=\GEN{c}\GEN{d}$ is a minimal factorization of $G_p$ $p^{\mu}=|c|\mid |a_p|=m_p$.
Moreover $p^{\mu+\nu}=|G_p|=m_pn_p$ and hence $m_p=p^\mu$ if and only if $n_p=p^\nu$.
This proves \eqref{mnnu} and \eqref{m=pmu}. For the remainder of the proof we distinguish cases.

    \textbf{Case 1}. Suppose that $e^{p-1}=1$.

Then $\mu,\nu,\sigma$ and $\rho$ satisfy the conditions  \ref{rhomu1}, \ref{rhomu2} and \ref{epsilon1} of \Cref{conditionsA-D}.

We first prove that if $p=2$, then $\epsilon=1$.
This is clear, if $\mu\le 1$.
Otherwise $4\mid 2^\rho = [\GEN{c}:G'_2]$.
As $|c|=2^{\mu} \mid m_2$, $[\GEN{a_2}:G'_2]$ is also a multiple of $4$ and hence $\epsilon=1$ (see Remark~\ref{Observations}), as desired. This proves \eqref{e12}.

Moreover $\frac{m_p}{r_p}=|(G')_p|=|(G_p)'|=p^{\mu-\rho}$, since $p\in \pi$.
In particular, if $m_p=p^\mu$, then $r_p=p^\rho$.
This proves \eqref{e1mr} and \eqref{m=pmu} in this case.

Let $g\in G_p$. Then $g=b_p^xa_p^y$ for some positive integers $x$ and $y$. If $a_p^{b_p^x}=a_p^z$, then using \Cref{PropEse}.\eqref{vpEse} and \eqref{Potencia} we deduce that $g^{\frac{m_pn_p}{s_p}} = a^{y\Ese{z}{\frac{m_pn_p}{s_p}}}=1$, as $s_p\mid n_p$ by \Cref{Parameters}.\eqref{Param+}.
As the order of $b_p$ is $\frac{m_pn_p}{s_p}$, this proves that the exponent of $G_p$ is $\frac{m_pn_p}{s_p}$ and a similar argument with $c$ and $d$ shows that the exponent of $G_p$ is $p^{\mu+\nu-\sigma}$.
As we already know that $m_pn_p=p^{\mu+\nu}$ we deduce that $s_p=p^\sigma$.
This proves \eqref{sp} in this case.

We now prove \eqref{e1mnomui}. Assume that $m_p\ne p^\nu$.  Suppose that $k_p=1$ or  $\mu=0$.
Then $[c,a_{\pi'}]=1$ and hence $G=\GEN{a_{p'}c}\GEN{b_{p'}d}$ is a metacyclic factorization of $G$.
As $G=\GEN{a}\GEN{b}$ is a minimal metacyclic factorization, we have $m_p=|a_p|\le |c|=p^{\mu}\le m_p$ in contradiction with the assumption.
Now suppose that $\rho\ne \sigma$. Then $1\le \rho<\sigma\le \mu$, by conditions \ref{rhomu1} and \ref{epsilon1}.
Moreover, $p^{\rho}\mid \frac{m_p}{p^{\mu}}p^{\rho}=r_p$.
As $G_p/\GEN{a_p}$ is cyclic, we have that $a_p=d^yc^z$ with either $p\nmid y$ or $p\nmid z$ and $\GEN{c^{p^{\rho}}}=G'_p\subseteq \GEN{a_p}$.
If $p\mid z$, then $p\nmid y$ and hence $c^{p^\rho}\in \GEN{a_p}=\GEN{dc^x}$ for some $x$.
However $\GEN{dc^x}\cap \GEN{c}=\GEN{c^{p^{\sigma}+x\Ese{1+p^\rho}{p^\nu}}}$ and $\sigma\le \nu$.
Thus $p^{\sigma}+x\Ese{1+p^\rho}{p^\mu}$ is multiple of $p^{\sigma}$, by \Cref{PropEse}.\eqref{vpEse}, so that it does not divide $p^{\rho}$.
Hence $c^{p^\rho}\not\in \GEN{a_p}$, a contradiction.
Therefore $p\nmid z$ and this implies that $\GEN{a_p}=\GEN{d^xc}$ for some integer $0\le x<p^\nu$.
If $x=0$, then $\GEN{c}=\GEN{a_p}$ and hence $m_p=p^{\mu}$.
Suppose otherwise that $x>0$ and let $u=v_p(x)$ and $c^{d^x}=c^{(1+p^\rho)^x}$.
Then $|a_p\GEN{c}|=|d^xc\GEN{c}|=p^{\nu-u}$ and $c^{p^{\rho}}\in (G_p)'\subseteq \GEN{a_p}\cap \GEN{c} = \GEN{(d^xc)^{p^{\nu-u}}}=\GEN{c^{xp^{\sigma-u}+\Ese{(1+p^\rho)^x}{p^{\nu-u}}}}$.
Therefore $v_p(xp^{\sigma-u}+\Ese{(1+p^\rho)^x}{p^{\nu-u}})\le \rho<\sigma=v_p(xp^{\sigma-u})$ and hence $v_p(xp^{\sigma-u}+\Ese{(1+p^\rho)^x}{p^{\nu-u}})=v_p(\Ese{(1+p^\rho)^x}{p^{\nu-u}})=\nu-u$.
Thus again $m_p=|a_p|=p^{\mu}$, in contradiction with the hypothesis. Thus $\rho=\sigma$.
Moreover, $m_p>p^\mu$, by \eqref{mnnu}, and $r_p>p^\rho$ by \eqref{e1mr}.
Therefore, by \eqref{sp}, $s_p=p^\sigma= p^{\rho}<r_p$.
Then $n_p<s_pk_p$, by \Cref{Parameters}.\eqref{Param+}.
This proves \eqref{e1mnomui}.

    \textbf{Case 2}. Suppose that $e=-1$ and $p=2$.

Then $\mu$, $\nu$, $\sigma$ and $\rho$ satisfy the conditions \ref{rhomu1}, \ref{rhomu2} and \ref{epsilon-1} of \Cref{conditionsA-D}.

We claim that $s_2\le 2^\sigma$. By means of contradiction suppose that $s_2>2^{\sigma}$. As $G_2/\GEN{a_2}$ is cyclic, either $\GEN{a_2} = \GEN{d^ic}$ or $\GEN{a_2} = \GEN{dc^{2i}}$ for some integer $i$.
If $G_2$ contains a cyclic normal subgroup contained in $C_{G_2}(a_{2'})$ and of the form $\GEN{d^ic}$ for some integer $i$, then $G=\GEN{a_{2'}d^ic}\GEN{b_{2'}d}$ is a metacyclic factorization of $G$ and therefore $s_2=[G_2:\GEN{b_2}]\le [G_2:\GEN{d}]=2^{\sigma}$, against the assumption.
Thus,  $\GEN{a_2}=\GEN{dc^{2i}}$ for some integer $i$.
Moreover, $\GEN{c^2}=G'_2\subseteq \GEN{dc^{2i}}$. Thus, $\GEN{c^2}=\GEN{dc^{2i}}\cap \GEN{c}=\GEN{(dc^{2i})^{2^\nu}}=\GEN{c^{2^{\sigma}+2i\Ese{-1+2^{\rho}}{2^{\nu}}}}$.
As $\rho\ge 2$ and $\mu\ge 1$, by \Cref{PropEse}.\eqref{vpEse}, $v_2(2i\Ese{-1+2^{\rho}}{2^{\nu}})=v_2(i)+\rho+\nu\ge 3$ and hence $\sigma=v_2(2^{\sigma}+2i\Ese{-1+2^{\rho}}{2^{\nu}})=1$. As $1\le \mu-1\le \sigma=1$, it follows that $\mu=2$ and hence $(dc)^2=d^2$. Thus $|dc|=|d|$ and $\GEN{a}\GEN{b_{2'}dc}$ is a metacyclic factorization of $G$. Then $s_2=[G_2:\GEN{b_2}]\le [G_2:\GEN{dc}]=[G_2:\GEN{d}]=2^\sigma$, since $\GEN{a}\GEN{b}$ is a minimal metacyclic factorization of $G$. This finishes the proof of the claim.

If $\epsilon=-1$, then $G'_2$ has index $2$ both in $\GEN{c}$ and in  $\GEN{a_2}$ and therefore $m_2=2^{\mu}$.
Conversely, if $m_2=2^{\mu}$, then $\epsilon=-1$ by \cite[Lemma~3.2]{GarciadelRioClasification}.
This proves \eqref{e-1eps-1}.

Suppose that $m_2=2^\mu$.
Since $G_2=\GEN{a_2}\GEN{b_2}$ is a metacyclic factorization and $G_2=\GEN{c}\GEN{d}$ is a minimal metacyclic factorization with $|a_2|=m_2=2^\mu=|c|$, we have $2^\sigma=[G_2:\GEN{d}]\le [G_2:\GEN{b_2}]=s_2$. Then $s_2=2^\sigma$, by the claim above. This completes the proof of \eqref{sp}.

Suppose that $m_2=2^\mu$ and $r_2\ne 2^\rho$.
Then $\Res_{m_2}(\GEN{\gamma})_2=\GEN{-1+r_2}_{m_2}$ and $T_{G_2}(\GEN{c})=\GEN{-1+2^\rho}_{m_2}$.
By \Cref{Parameters}.\eqref{Param-}, $s_2\ne r_2n_2$ and, by condition \ref{epsilon-1} of \Cref{conditionsA-D}, $\sigma\ne \rho+\nu$.
By \cite[Lemma~3.3]{GarciadelRioClasification}, the hypothesis $r_2\ne 2^\rho$ implies that $4$ divides $n_2$, $8$ divides $m_2$, $k_2<n_2$, $s_2=2^\sigma=2^{\mu-1}$ and $m_2$ divides both $2r_2$ and $2^{\rho+1}$.
Since both $r_2$ and $2^\rho$ divide $m_2$ and they are different it follows that either $m_2=r_2$ or $\rho=\mu$. This contradicts either \Cref{Parameters}.\eqref{Param-} or condition~\ref{epsilon-1}(b).
This completes the proof of \eqref{rp}.

Suppose $\epsilon=1$. We still need to prove that $s_2=2^\sigma=2$, $k_2=2=n_2$, $\mu=2$, $m_2= 2^{\nu+1}$ and $r_2=2^\nu$.
By (4a) we have $2^\mu<m_2$.
Then $[a_{\pi'},c]\ne 1$ for otherwise $\GEN{a_{2'}c}\GEN{b}$ is a metacyclic factorization of $G$ and, as $G=\GEN{a}\GEN{b}$ is a minimal metacyclic factorization, we have that $|a|=m\le |a_{2'}c|=m_{2'}2^\mu$, so that $m_2\le 2^\mu$, a contradiction.
Therefore, $k_2\ne 1$.
On the other hand, $\GEN{c^2}=G'_2\subseteq \GEN{a_2}\cap \GEN{c}= \GEN{a_2^{|a_2\GEN{c}|}}$ and therefore $[a_{\pi'},c^2]=1$. Suppose that $\GEN{a_2} = \GEN{d^i c}$ for some integer $i$ and let $u=v_2(i)$.
Then $u<\nu$ for otherwise $a_2\in \GEN{c}$ and hence $m_2=|a_2|\le |c|=2^\mu$, a contradiction.
Moreover, $|a_2\GEN{c}|=2^{\nu-u}$ and $\GEN{c^2}=G'_2\subseteq \GEN{c}\cap \GEN{a_2} = \GEN{a_2^{2^{\nu-u}}} = \GEN{c^{2^\sigma\frac{i}{2^u}+\Ese{(-1+2^\rho)^i}{2^{\nu-u}}}} \subseteq \GEN{c^2}$, since $1\le \sigma$ and $2\mid \Ese{(-1+2^\rho)^i}{2^{\nu-u}}$, by \Cref{PropEse}.\eqref{vpEse}.
Thus $2^{\mu-1}=|c^2|=|a_2^{2^{\nu-u}}|=m_2 2^{u-\nu}$. Therefore $m_2<2^{\mu}=m_22^{1+u-\nu}\le m_2$, a contradiction.
Therefore $\GEN{a_2}=\GEN{dc^{2i}}$, for an integer $i$.
Then $[d,a_{\pi'}]=[c^{-2i},a_{\pi'}]=1$.
Moreover, $|a_2\GEN{c}|=2^{\nu}$ and $a_2^{2^{\nu}}=c^{2^{\sigma}+2i\Ese{-1+2^{\rho}}{2^{\nu}}}$.
Since $2\le v_2(2i\Ese{-1+2^{\rho}}{2^{\nu}})$, $\sigma\ge \mu-1\ge 1$ and $c^2\in \GEN{a_2^{2^{\nu}}}$, necessarily $\sigma=1$ and $\mu=2$.
Then $m_2=|a_2|=2^{\nu+1}$, $n_2=2$ and $r_2=2^{\nu}$.
As $2\le k_2\le n_2$, $k_2=2$.
Moreover, $|d|=|dc|=2^{\nu+1}$ and $|G_2|=2^{\nu+2}$.
Therefore $\exp(G_2)=|d|$.
Then $G=\GEN{a}\GEN{b_{2'}d}$ is a minimal metacyclic factorization, so that $s_2=2=2^{\sigma}$.
This proves \eqref{e-1eps1} and finishes the proof of the lemma.
\end{proof}

The first statement of the next Proposition shows that $s^G$ is determined by $\Q G$. The remaining statements will be used in the proof of \Cref{EpsilonDetermined}, which shows that $\epsilon^G$ is determined by $\Q G$.

\begin{proposition}\label{sDetermined}
	Let $G$ and $H$ be metacyclic finite groups such that $\Q G\cong \Q H$ and let $p\in\pi$.
	Then
	\begin{enumerate}
		\item $s^G=s^H$.
		\item If $(k^G)_p=1$, then $(m^G)_p=(m^H)_p=m^{G_p}$ and $(r^G)_p=(r^H)_p = r^{G_p}$.
		\item If $\epsilon^{G_{\pi}}=1$, then $\epsilon^G=\epsilon^H=1$.
		\item If $\epsilon^{G_p}=1$ and either $m^{G_p}=1$ or $r^{G_p}<s^{G_p}$, then $(m^G)_p=(m^H)_p=m^{G_p}$ and $(r^G)_p=(r^H)_p=r^{G_p}$.
		\item If $\epsilon^{G_2}=-1$ and, at least one of $s^{G_2}\ne 2$, $n^{G_2}=(k^G)_2$ or $m^{G_2}\ne 4$ holds, then $\epsilon^G=\epsilon^H=-1$, $(m^G)_2=(m^H)_2$ and $(r^G)_2=(r^H)_2$.
	\end{enumerate}
\end{proposition}

\begin{proof}
Suppose that $\Q G\cong \Q H$. Then $\pi=\pi_G=\pi_H$, $\pi'=\pi'_G=\pi'_H$ and $G_{\pi}\cong H_{\pi}$, by \Cref{piIgual}.
Thus $\INV(G_p)=\INV(H_p)=(p^\mu,p^\nu,p^\sigma,\GEN{e+p^\rho}_{p^\mu})$ for every $p\in \pi$.

(1) By \Cref{DeGpAG}, $(s^G)_p=(s^H)_p=p^{\sigma}$ for every $p\in \pi$. Since, moreover $(s^G)_{\pi'}=(m^G)_{\pi'} = (G')_{\pi'} = (H')_{\pi'} = (m^H)_{\pi'} = (s^H)_{\pi'}$, we conclude that $s^G=s^H$.

Let $p\in \pi$. Recall that $k^G=k^H$, by \Cref{RDetermined}, which we denote $k$ from now on.

(2) If $k_p=1$ then $m^G_p=m^H_p=p^\mu$ by statements \eqref{e1mnomui} and \eqref{e-1} of \Cref{DeGpAG}, and hence $(r^G)_p=(r^H)_p=p^\rho$, by \Cref{DeGpAG}.\eqref{rp}.

(3) is a direct consequence of Lemma~\ref{DeGpAG}.\eqref{e12}.

(4) follows from Lemma~\ref{DeGpAG}.\eqref{e1mnomui}.

(5) is a consequence of statements \eqref{e-1} and \eqref{rp} of Lemma~\ref{DeGpAG}.
\end{proof}

\begin{proposition}\label{EpsilonDetermined}
Let $G$ and $H$ be finite metacyclic groups. If $\Q G\cong \Q H$, then $\epsilon^G=\epsilon^H$.
\end{proposition}

\begin{proof}
By means of contradiction, assume that $\Q G\cong \Q H$ and, without loss of generality suppose that $\epsilon^G=1$ and $\epsilon^H=-1$.
By \Cref{sDetermined}, $s^G=s^H$, which we denote $s$ and by \Cref{RDetermined}, $k^G=k^H$, which we denote $k$.
Recall that $G_2\cong H_2$, by \Cref{piIgual}.
By \Cref{DeGpAG}.\eqref{e12}, $\epsilon^{G_2}=\epsilon^{H_2}=-1$, and by \Cref{DeGpAG}.\eqref{e-1eps1},  $\INV(G_2)=\INV(H_2)=(4,2^{\nu},2,\GEN{-1}_4)$, with $\nu\ge 2$, $(m^G)_2=2^{\nu+1}$, $(n^G)_2=k_2=(s^G)_2=2$,
$(r^G)_2=2^{\nu}$, $(m^H)_2=(r^H)_2=4$ and $(n^H)_2=2^{\nu}$.

In the remainder of the proof, $E$ is either $G$ or $H$ and $E=\GEN{a}\GEN{b}$ is a minimal metacyclic factorization of $E$.
Moreover, we adopt the notation $m=m^E$, $n=n^E$, etc.
Since $m_2\in \{r_2,2r_2\}$ and $k_2=2$, it follows that $b_2^k\in Z(E)$ and $[b^k,a_2]=1$.

We are going to compute the number of simple components $S$ of $\Q E$ satisfying conditions \ref{B1} and \ref{B2}.

Set $L=\GEN{a,b^k}=C_E(a_{\pi'})=C_E(a_{\pi'\cup \{2\}})$.
Observe that $L$ is nilpotent and $L_p=\GEN{a_p,b_p^k}$ for every prime $p$.

By \Cref{SSPMetacyclic}, the Wedderburn components of $\Q E$ satisfying condition \ref{B1} are those of the form $\Q Ee(E,L,K)$ with $K$ a subgroup of $E$ such that \begin{equation}\label{SSPCondicion}
L \text{ is maximal in } \{B\le E : C_E(a)\le B, B'\le K \le B\} \text{ and }
L/K \text{ is cyclic}
\end{equation}
Observe that the maximality condition on $L$ is equivalent to $[b^k,a]\in K$ but $[b^{\frac{k}{p}},a]\not\in K$ for any $p\in \pi(k)$.

For every subgroup $K$ of $L$ satisfying \eqref{SSPCondicion} let $S_K=\Q Ee(E,L,K)$.

\medskip
\noindent\underline{Claim 1}. The subgroups $K$ of $L$ satisfying \eqref{SSPCondicion} and such that $S_K$ satisfies condition \ref{B2} are precisely those of the form $L_{\pi\setminus\{2\}}\times K_{\pi'}\times K_2$  with
\begin{itemize}
    \item[\mylabel{KB1}{(KB1)}] $K_{\pi'}$ is a cocyclic subgroup of $L_{\pi'}$;
    \item[\mylabel{KB2}{(KB2)}] $K_2$ a cocyclic subgroup of $L_2$; and
    \item[\mylabel{KB3}{(KB3)}] $[b^{\frac{k}{p}},a]\not\in K$ for every prime $p\mid k$.
\end{itemize}
Indeed, suppose that $K$ satisfies  \eqref{SSPCondicion} and $S_K$ satisfies condition \ref{B2}. By \Cref{PCIOdd-A2}, we have that $L_{\pi\setminus \{2\}}\subseteq \Core_G(K)$ and therefore  $K=L_{\pi\setminus \{2\}}\times K_{\pi'} \times K_2$ satisfies conditions \ref{KB1}-\ref{KB3}.
Conversely, let $K=L_{\pi\setminus \{2\}}\times K_{\pi'}\times K_2$ satisfy conditions \ref{KB1}-\ref{KB3}. Then $(L,K)$ satisfies \eqref{SSPCondicion} and hence $S_K=\Q Ge(E,L,K)$ is a simple component of $\Q E$. On the other hand by \Cref{SSPMetacyclic}, the center of $S_K$ is isomorphic to a field contained in $\Q_{[L:K]}$ and $\pi([L:K])\subseteq \pi'\cup \{2\}$. Therefore $S_K$ satisfies condition \ref{B2}. This finishes the proof of the claim.

As $s_2=2$, $k_2=2$ and $\nu\ge 2$,
$$L_2=\GEN{a_2,b_2^2} =
\begin{cases} \GEN{a_2}\cong C_{2^{\nu+1}}, & \text{if } E=G; \\
\GEN{a_2^{-1}b_2^{2^{\nu-1}}}\times \GEN{b_2^2}\cong C_2 \times C_{2^\nu}, & \text{if } E=H.\end{cases}$$
Therefore, if $E=G$, then every subgroup of $L_2$ is normal in $E$.
Otherwise, i.e. if $E=H$, then $E'_2$ is the socle of $\GEN{b_2^2}$ and hence the only subgroups of $L_2$ which are not normal in $H$ are $\GEN{a_2^{-1}b_2^{2^{\nu-1}}}$ and $\GEN{a_2b_2^{2^{\nu-1}}}$. Moreover, these two subgroups are conjugate in $E$.

\underline{Claim 2}: Let $K=L_{\pi\setminus \{2\}}\times K_{\pi'}\times K_2$ satisfy conditions \ref{KB1} and \ref{KB2}. Then $K$ satisfies condition \ref{KB3} if and only if one of the following conditions hold:
\begin{enumerate}
    \item $[b^{\frac{k}{p}},a_{\pi'}]\not\in K_{\pi'}$ for every $p\in \pi(k)$.

    \item $[b^{\frac{k}{p}},a_{\pi'}]\not\in K_{\pi'}$ for every $p\in \pi(k)\setminus \{2\}$ and $[b^{\frac{k}{2}}:a_{\pi'}]\in K_{\pi'}$ and either $E=G$ and $K_2=1$ or $E=H$ and $K_2$ is either $\GEN{a_2^{-1}b_2^{2^{\nu-1}}}$ or $\GEN{a_2b_2^{2^{\nu-1}}}$.
\end{enumerate}

Indeed, clearly, if $K$ satisfies condition (1), then it also satisfies condition \ref{KB3}. Suppose that $K$ satisfies condition (2). Then condition \ref{KB3} holds for every prime $p\ne 2$. Moreover, as $v_2(k)=1$, if $E=G$, then $[b^{\frac{k}{2}},a_2]=a_2^{2^{\nu}}\not\in K_2$ and if $E=H$, then $[b^{\frac{k}{2}},a_2]=a_2^2\not\in K_2$. Therefore $[b^{\frac{k}{2}},a]\not\in K$. Therefore $K$ satisfies condition \ref{KB3}, as desired. Finally suppose $K$ satisfy neither (1) nor (2).
Then $[b^{\frac{k}{p}},a_{\pi'}]\in K_p$ for some $p\in \pi(k)$.
Suppose that $p\ne 2$. Then $[b^{\frac{k}{p}},L_2]=1$ and $[b^{\frac{k}{p}},L_{\pi\setminus \{2\}}]\subseteq \GEN{a_{\pi\setminus \{2\}}}\subseteq K$. Then $[b^{\frac{k}{p}},L]\subseteq K$ and, in particular, $[b^{\frac{k}{p}},a]\in K$.
Therefore \ref{KB3} does not hold.
Suppose $p=2$ and $[b^{\frac{k}{p}},a_{\pi'}]\not\in K_{\pi'}$ for every $p\in \pi(k)\setminus \{2\}$. As condition (2) does not hold, either $E=G$ and $K_2\neq 1$ or $E=H$ and $K_2$ is neither $\GEN{a_2^{-1}b_2^{2^{\nu-1}}}$ nor $\GEN{a_2b_2^{2^{\nu-1}}}$.
In both cases $K_2$ contains $E'_2$ and therefore $[b^{\frac{k}{2}},a_2]\in K$.
As also $[b^{\frac{k}{2}},a_{\pi'}]\in K_{\pi'}$ and $[b^{\frac{k}{2}},a_{\pi\setminus \{2\}}] \in \GEN{a_{\pi\setminus \{2\}}}\subseteq K$,
it follows that $[b^{\frac{k}{2}},a]\in K$.
Therefore condition \ref{KB3} fails. This finishes the proof of Claim 2.

As $L$ is normal in $E$, by \cite[Problem~3.4.3]{JespersdelRioGRG1}, if $K=L_{\pi\setminus\{2\}}\times K_{\pi'}\times K_2$ and $M=L_{\pi\setminus\{2\}}\times M_{\pi'}\times M_2$ are two subgroups of $L$ satisfying conditions \ref{KB1}-\ref{KB3}, then $S_K=S_M$ if and only if $K$ and $M$ are conjugate in $E$ if and only if $K_{\pi'}$ and $M_{\pi'}$ are conjugate in $E$ and $K_2$ and $M_2$ are conjugate in $E$.

Let $d$ denote the number of conjugacy classes of cocyclic subgroups $K_{\pi'}$ of $L_{\pi'}$ which satisfy condition (1), and let $d_1$ denote the number of conjugacy classes of cocyclic subgroups of $L_{\pi'}$ which satisfy the first part of condition (2).
As $R^G=R^H$, $d$ and $d_1$ are independent of $E$.
Let $h$ denote the number of cocyclic subgroups of $K_2$.
Then
$$h=\begin{cases} \nu+2, & \text{if } E=G; \\
2(\nu+1); & \text{if } E=H.
\end{cases}$$
By Claim 2, combined with the discussion about the conjugacy classes in $E$ of subgroups of $L_2$, the number of simple components of $\Q E$ satisfying conditions \ref{B1} and \ref{B2} is
$$N_E=\begin{cases}
dh+d_1=d(\nu+2)+d_1, &  \text{if } E=G \\
d(h-1)+d_1=d(2\nu+1)+d_1, &  \text{if } E=H.
\end{cases}$$
As $\GEN{a_{\pi'}}\cap \GEN{b_{\pi'}}=1$, $K_{\pi}=\GEN{b_{\pi'}^k}$ satisfies condition (1) and hence $d\ge 1$.
Moreover, $\nu\ge 2$, and therefore $N_G<N_H$ and hence $\Q G\not\cong \Q H$, a contradiction.
\end{proof}

\section{$\Q G$ determines $m^G$, $n^G$ and $r^G$}\label{SectionmnrDetermined}

In this section we prove that $m^G$, $n^G$ and $r^G$ are determined by $\Q G$, i.e. we prove the following proposition.

\begin{proposition}\label{mnrDetermined}
Let $G$ and $H$ be finite metacyclic groups such that $\Q G\cong \Q H$. Then $m^G=m^H$, $n^G=n^H$ and $r^G=r^H$.
\end{proposition}

\begin{proof}
We will be working all the time with the group $G$ and a minimal metacyclic factorization $G=\GEN{a}\GEN{b}$. Recall that we have fixed notation $m=m^G, n=n^G, s=s^G, r=r^G, \epsilon=\epsilon^G, \dots$ and the goal is proving that $m,n$ and $r$ are determined by $\Q G$.
Recall that $k$, $\epsilon$, $s$ and the isomorphism type of $G_{\pi}$ are determined by $\Q G$. This will be used along the proof without specific mention.
We work prime by prime, i.e. we fix a prime $p$ and we have to prove that $m_p,n_p$ and $r_p$ are determined by $\Q G$.
We keep the notation for $\INV(G_p)$ as in the previous section, i.e.
    $$m^{G_p}=p^{\mu}, \quad n^{G_p}=p^{\nu}, \quad r^{G_p}=p^{\rho} \quad e=\epsilon^{G_p}.$$

We first obtain some reductions:
As $(r^G)_{\pi'}=1$ and $(m^G)_{\pi'}$ and $(n^G)_{\pi'}$ are determined by $\Q G$ (\Cref{piIgual}), we may assume that $p\in \pi$.
If $e=-1$, then $p=2$ and, by \Cref{DeGpAG},
$$(m_2,n_2,r_2)=\begin{cases}
(2^{\nu+1},2,2^{\nu}) & \text{if } \epsilon=1; \\ (2^\mu,2^\nu,2^\rho), & \text{if } \epsilon=-1
\end{cases}$$
Thus we may assume that $e=1$ and hence $\epsilon=1$.
Hence $m_pn_p=p^{\mu+\nu}$ and $\frac{m_p}{r_p}=p^{\mu-\rho}$.
Therefore, it is enough to prove that one of the three $m_p$, $n_p$ or $r_p$ is determined by $\Q G$.
For future use we express $m_p$ and $r_p$ in terms of $n_p$:
\begin{equation}\label{DenAmr}
m_p=\frac{p^{\mu+\nu}}{n_p} \qand r_p=\frac{p^{\nu+\rho}}{n_p}.
\end{equation}
Observe that $p^{\mu-\rho}=\frac{m_p}{r_p}\le n_p$, by \Cref{Parameters}.\eqref{Param+}.
If $\mu=0$, then $m_p=1$, by \Cref{sDetermined}. Thus we may assume that $\mu>0$ and hence $\rho>0$. Therefore $m_p\ge r_p>1$.
Let
$$l=\lcm(k,p^{\mu-\rho}) \qand
L=C_G(a_{\pi'\cup \{p\}}).$$
Then
$$L=\GEN{a,b^l}=\GEN{a_{\pi'}}\times \GEN{b_{\pi'}^k} \times \prod_{q\in \pi\setminus \{p\}} \GEN{a_q,b_q^k} \times \GEN{a_p,b_p^{\max(k_p,p^{\mu-\rho})}}.$$
Moreover,
$l_p\le n_p \le p^{\nu}$ and $l_p\le (n^H)_p \le p^{\nu}$.
Therefore, if $l_p=p^\nu$, then $(n^H)_p=n_p$, as desired.
Hence, we may assume that $l_p<p^\nu$.
If $k_p=1$, $\mu=0$ or $\rho<\sigma$, then $(m^H)_p=m_p$, by \Cref{sDetermined}.
Thus we may also assume that $k_p>1$, $\mu>0$ and $\rho=\sigma$.
Therefore $\frac{m_p}{s_p}=p^{\mu-\sigma} \le p^\rho = p^\sigma = s_p$, by \Cref{DeGpAG}.\ref{epsilon1}.
Moreover, $s_p\le n_p\le p^\nu$ and $s_p\le (n^H)_p\le p^\nu$, by \Cref{Parameters}.\eqref{Param+}.
Therefore, if $\rho=\nu$, then $(n^H)_p=n_p$, as desired.
Hence, we also may assume that $\rho<\nu$.
By \Cref{DeGpAG}, if $n_p\ne p^\nu$, then $n_p<p^\nu$, hence $m_p>p^\mu$, so that $r_p>p^\rho=s_p$ and thus $n_p<k_ps_p=k_pp^{\rho}$, by \Cref{Parameters}.\eqref{Param+}.
This proves that
$n_p=p^\nu$ or $n_p<\min(p^\nu,p^\rho k_p)$.

Summarizing, in the remainder of the proof we assume the following:
\begin{enumerate}
\item[\mylabel{U1}{(U1)}] $e=\epsilon=1$.
\item[\mylabel{U2}{(U2)}] $k_p>1$, $0<\mu\le 2\rho$, $1\le \rho=\sigma<\nu$, $l_p<p^\nu$, $(s^H)_p=s_p=p^\rho$  and $\max(l_p,p^\rho)\le n_p$.
\item[\mylabel{U3}{(U3)}] Either $n_p=p^\nu$ or
$n_p<\min(p^\nu,p^\rho k_p)$. In particular, if $n_p\ge l_pp^\rho$, then $n_p=p^\nu$.
\end{enumerate}

	The strategy is similar to the one in the previous section, namely we analyze how are the Wedderburn components of $\Q G$ of a certain kind.
	In this case, we consider Wedderburn components $S$ of $\Q G$ satisfying the following conditions:
	\begin{itemize}
		\item[\mylabel{C1}{(C1)}] $\Deg(S)=l$.
		\item[\mylabel{C2}{(C2)}] The center $F$ of $S$ does not contain roots of unity of order $q\in \pi\setminus \{p,2\}$.
		\item[\mylabel{C3}{(C3)}] If $p\ne 2$, then $F$ does not contain a root of unity of order $4$.
	\end{itemize}
	We denote by $N_G$ the number of Wedderburn components of $\Q G$ satisfying conditions \ref{C1}-\ref{C3}.
	We will obtain a formula for $N_G$ and use it to prove that $N_G$ determines $(n^G)_p$.
	As $N_G$ is determined by $\Q G$, this will show that so is $(n^G)_p$ as desired.

We illustrate this with the following example.

\begin{example}{\rm
Consider the groups
$$G=\GEN{a,b \mid a^{63}=1, b^{72}=a^{21}, a^b=a^{58}}, \quad
H=\GEN{a,b \mid a^{189}=1, b^{24}=a^{63}, a^b=a^{37}}.$$
For these groups $\pi'=\{7\}$, $\pi=\{2,3\}$, $k=l=3$, $G_2\cong H_2\cong C_8$ and $G_3\cong H_3 = \GEN{a,b\mid a^9=1, b^9=a^3, a^b=a^4}$.
So $\INV(G_2)=(1,2^3,1,\GEN{1}_1)$ and $\INV(G_3)=\GEN{3^2,3^2,3,\GEN{4}_9}$ and conditions (U1)-(U3) only hold for $p=3$. Moreover, the two groups take the two possible values for $n_3$, namely $n_3^G=3^2$ and $n_3^H=3$.
We now compute the Wedderburn decomposition of their rational group algebras:
\begin{eqnarray*}
	\Q G&=& 2\Q \oplus 8 \Q_3 \oplus \Q_4 \oplus \Q_8 \oplus 6 \Q_9 \oplus 4 \Q_{12} \oplus 4 \Q_{12} \oplus 4 \Q_{24} \oplus 3 \Q_{36} \oplus 3\Q_{72} \oplus \\
	&& \mathbf{2M_3(\Q_9) \oplus 2M_3(\Q_7^{2}) \oplus 2M_3(\Q_{21}^{4}) \oplus
	6 (\Q_{21},\sigma_4,\zeta_3) \oplus 2M_3(\Q_{63})}\oplus \\
	&& M_3(\Q_{72})\oplus M_3(\Q_{252})\oplus M_3(\Q_{504})\oplus M_3(\Q_{28}^{9}) \oplus M_3(\Q_{56}^{9}) \oplus M_3(\Q_{84}^{25}) \oplus \\
	&& M_3(\Q_{168}^{25})\oplus 3(\Q_{84}/\Q_{84}^{25},\zeta_3)
	\oplus 3(\Q_{168}/\Q_{168}^{25},\zeta_3) \\
	\Q H &=& 2\Q \oplus 8 \Q_3 \oplus \Q_4 \oplus \Q_8 \oplus 6 \Q_9 \oplus 4 \Q_{12} \oplus 4\Q_{24} \oplus 3\Q_{36} \oplus 3\Q_{72} \oplus \\
	&& \mathbf{2M_3(\Q_9) \oplus 2M_3(\Q_7^{2}) \oplus 2M_3(\Q_{21}^{4}) \oplus 2M_3(\Q_{63}^{37}) \oplus 2M_3(\Q_{189}^{37})} \oplus \\
	&&M_3(\Q_{36}) \oplus M_3(\Q_{72}) \oplus M_3(\Q_{28}^{9}) \oplus
	M_3(\Q_{56}^{9}) \oplus M_3(\Q_{84}^{25}) \oplus
	M_3(\Q_{168}^{25}) \oplus \\
	&& M_3(\Q_{252}^{37}) \oplus
	M_3(\Q_{504}^{289}) \oplus M_3(\Q_{756}^{37}) \oplus
	M_3(\Q_{1512}^{793}).
\end{eqnarray*}
As $\pi=\{2,3\}$ all the components satisfy condition \ref{C2}.
In both cases, the simple components in the first line have degree $1$ and the remaining components have degree $l=3$, however only the components in the second line satisfy condition \ref{C3}. So $n_G=14$ and $n_H=10$. }
\end{example}

We start characterizing the Wedderburn components of $\Q G$ satisfying conditions \ref{C1}-\ref{C3} for a fixed prime $p$ in terms of some subgroups of $L$.

\begin{lemma}\label{Componentsmp}
The Wedderburn components of $\Q G$ satisfying conditions \ref{C1}-\ref{C3} are the algebras of the form $S_K=\Q Ge(G,L,K)$ for a subgroup $K$ of $L$ satisfying the following conditions:
\begin{itemize}
\item[\mylabel{KC1}{(KC1)}] $K_{\pi\setminus \{p,2\}}=L_{\pi\setminus \{p,2\}}$.
\item[\mylabel{KC2}{(KC2)}] $K_{\pi'}$ and $K_p$ are cocyclic subgroups of $L_{\pi'}$ and $L_p$, respectively.
\item[\mylabel{KC3}{(KC3)}] $[b^{\frac{l}{q}},a_{\pi'}]\not\in K_{\pi'}$, for every $q\in \pi(k)\setminus \{p\}$.
\item[\mylabel{KC4}{(KC4)}] If $[b^{\frac{l}{p}},a_{\pi'}]\in K_{\pi'}$, then $[b^{\frac{l}{p}},a_p]\not\in K_p$.
\item[\mylabel{KC5}{(KC5)}] If $p\ne 2$, then $K_p$ is a subgroup of $L_p$ of index at most $2$.
\end{itemize}

	If $K_1$ and $K_2$ are subgroups of $L$ satisfying \ref{KC1}-\ref{KC4}, then $S_{K_1}=S_{K_2}$ if and only if $K_1$ and $K_2$ are conjugate in $G$.
\end{lemma}

\begin{proof}
By \Cref{SSPMetacyclic},  the Wedderburn components of $\Q G$ satisfying condition \ref{C1} are those of the form $S_K=\Q G e(G,L,K)$ with $K$ a cocyclic subgroup of $L$ such that $L$ is maximal in $\{B\le G: C_G(a)\subseteq B, B'\le K\le B\}$.
As $L$ is nilpotent, a subgroup $K$ of $L$ is cocyclic in $L$ if and only if $K_{\pi'}$, $K_{\pi\setminus \{p,2\}}$, $K_p$ and $K_2$ are cocyclic in $L_{\pi'}$, $L_{\pi\setminus \{p,2\}}$, $L_p$ and $L_2$ respectively.
Moreover, by \Cref{PCIOdd-A2}, if $S_K$ satisfies \ref{C2}, then $K$ satisfies \ref{KC1}. Conversely, if condition \ref{KC1} holds, then $\pi([L:K])\subseteq \pi'\cup \{p,2\}$ and as the center of $S_K$ is isomorphic to a subfield of $\Q_{[L:K]}$, $S$ satisfies condition \ref{C2}.
A similar argument, using \Cref{PCI2-A1}, shows that $S_K$ satisfies condition \ref{C3} if and only if $[L_2:K_2]\le 2$, because if $p\ne 2$, then $L_2=\GEN{a_2,b_2^k}$ and hence $L_2/\GEN{a_2^2,b_2^{2k}}$ is an elementary abelian $2$-group.
Observe that if conditions \ref{KC1} and \ref{KC5} hold and $q\in \pi(k)\setminus \{p\}$, then $[b^{\frac{l}{q}},a_{\pi\setminus \{p\}}]\in K$. Then $L$ is maximal in $\{B\le G: C_G(a)\subseteq B, B'\le K\le B\}$ if and only if for every $q\mid l$, $[b^{\frac{l}{q}},a_{\pi'\cup \{p\}}]\not\in K$, and using that $\GEN{a}$ is normal in $G$, it is easy to see that this is equivalent to the combination of conditions \ref{KC3} and \ref{KC4}. This finishes the proof of the first statement of the lemma. The last one is a direct consequence of \cite[Problem~3.4.3]{JespersdelRioGRG1} (see also \cite[Proposition~1.4]{OlivieridelRioSimon2006}).
\end{proof}

Our next goal is describing the cocyclic subgroups of $L_p$ and their normalizers in $G$. To that end we introduce the following positive integers:
$$v=\min\left(\frac{n_p}{l_p},p^\rho\right), \quad  u=\frac{|L_p|}{v}=\frac{p^{\mu+\nu}}{vl_p}, \quad t=\frac{p^{\nu+2\rho}}{v^2l_p}.$$
\begin{remarks}\label{uvt}
\begin{enumerate}
	\item\label{vrtuvt} $u \le up^{2\rho-\mu}=vt$.

	\item\label{v=prho} If $v=p^\rho$, then $n_p=p^\nu$, $r_p=p^\rho$ and $v=r_p\le t\le u$.

	\item \label{vNoprho} If $v\ne p^\rho$, then $v<u$ and $v\le \frac{r_p}{p}\le \frac{t}{p^2}$.

	\item\label{o=kvut} If $k_p\le p^{\mu-\rho}$, then $v<u$ and $u\ge t$.
\end{enumerate}
\end{remarks}

\begin{proof}
\eqref{vrtuvt} By \ref{U2}, $\mu\le 2\rho$. Thus
$vt=\frac{p^{\nu+2\rho}}{vl_p} = \frac{p^{\mu+\nu}}{vl_p}p^{2\rho-\mu}=up^{2\rho-\mu}\ge u$.

\eqref{v=prho} Suppose that $v=p^\rho$.
Then $n_p\ge l_pp^\rho$ and hence $n_p=p^\nu$, by \ref{U3}.
Then, by \eqref{DenAmr},
$r_p=p^\rho=v\le \frac{n_p}{l_p}=\frac{p^\nu}{l_p}=t\le \frac{p^{\mu+\nu-\rho}}{l_p}=u$.

\eqref{vNoprho} Suppose that $v\ne p^\rho$.
Then, as $n_p\le p^\nu$, we have $v=\frac{n_p}{l_p}< p^\rho \le p^\mu \le \frac{p^{\mu+\nu}}{n_p}=u$, and using \eqref{DenAmr} and that $l_p\le n_p$ and  $n_p<l_pp^\rho$, it follows that $v=\frac{n_p}{l_p}\le p^{\rho-1}\le \frac{r_p}{p}
\le \frac{l_pp^\rho}{n_p} \frac{r_p}{p^2}=\frac{l_pp^{\nu+2\rho}}{n_p^2p^2}=\frac{t}{p^2}$.

\eqref{o=kvut} Assume that $k_p\le p^{\mu-\rho}$. Then $l_p=p^{\mu-\rho}$.
By means of contradiction suppose that $v\ge u$. Then $v=p^\rho$, by \eqref{vNoprho}, so $n_p=p^\nu$, by \eqref{v=prho}. As $\rho<\nu$, by \ref{U2}, we get $p^{2\rho}=v^2\ge vu = \frac{p^{\mu+\nu}}{l_p}=p^{\nu+\rho} > p^{2\rho}$, a contradiction.
Again by means of contradiction, suppose that $t>u$. Then, by \eqref{v=prho}, $v=\frac{n_p}{l_p}<p^{\rho}$, by \ref{U2}, $n_p\ge p^\rho$, and, as $l_p=p^{\mu-\rho}$, it follows that $t=\frac{p^{\rho+\mu+\nu}}{n_p^2}\le \frac{p^{\mu+\nu}}{n_p}=u$, a contradiction.
\end{proof}

\begin{lemma}\label{CocyclicLp}
	Set
	$$g=\begin{cases} a_p, & \text{if } n_p\le l_pp^\rho; \\
	b_p^{l_p}, & \text{otherwise};\end{cases} \qand
	h=\begin{cases} b_p^{l_p}a_p^{-\frac{l_pp^\rho}{n_p}}, & \text{if } n_p\le l_pp^\rho; \\
	b_p^{p^{\nu-\rho}}a_p^{-1}, & \text{otherwise}.\end{cases}$$

	Then
	\begin{enumerate}
		\item\label{LpStructure} $L_p=\GEN{g}\times \GEN{h}$, $G'_p=\GEN{a_p^{r_p}} \subseteq \GEN{g}$, $|g|=u$ and  $|h|=v$.

		\item\label{LpCocyclic} Let
		$$C_{L_p} = \left\{(i,y,x) : i \in \{1,2\},\quad  1\le x\le y \qand \begin{cases} y\mid v, & \text{if } i=1; \\
		y\mid u, p\mid x \text{ and } p\mid y\mid vx, & \text{if } i=2 \end{cases} \right\}.$$
		Then
		$$(i,y,x)\mapsto K_{i,y,x}=\begin{cases} \GEN{gh^x,h^y}, & \text{if } i=1; \\
		\GEN{g^xh,g^y}, & \text{if } i=2;\end{cases}$$
		defines a bijection from $C_{L_p}$ to the set of cocyclic subgroups of $L_p$.

		\item\label{NormalizerK12} If $(i,y,x)\in C_{L_p}$, then $N_G(K_{i,y,x}) = \begin{cases} \GEN{a,b^{\frac{y}{t}}} ; & \text{if } i=2 \text{ and } y\ge t; \\ G; & \text{otherwise}.\end{cases}$.
	\end{enumerate}
\end{lemma}

\begin{proof}
\eqref{LpStructure} $L_p$ is an abelian group generated by $a_p$ and $b_p^{l_p}$,  $l_p=\max(k_p,p^{\mu-\rho})=\max(k_p,\frac{m_p}{r_p})\le n_p$, by \eqref{DenAmr}, and
$s_p=p^\rho\le r_p$ by \ref{U2}. Then $\GEN{b_p}\cap \GEN{a_p}=\GEN{b_p^{l_p}}\cap \GEN{a_p}=\GEN{b_p^{n_p}}=\GEN{a_p^{p^{\rho}}}\supseteq \GEN{a_p^{r_p}}=(G_p)'=G'_p$.

Suppose first that $n_p\le l_pp^\rho$. Then, since $G$ is given by the presentation in \eqref{Presentation}, we have $|b_p^{l_p}|=\frac{n_p}{l_p}\frac{m_p}{p^\rho}\le m_p=|a_p|$,  $\left( b_p^{l_p}a_p^{-\frac{l_pp^\rho}{n_p}}\right)^{\frac{n_p}{l_p}}=1$ and
$\GEN{a_p}\cap \GEN{b_p^{l_p}a_p^{-\frac{l_pp^\rho}{n_p}}}=1$.
Therefore
$L_p=\GEN{a_p}\times \GEN{b_p^{l_p}a_p^{-\frac{l_pp^\rho}{n_p}}} = \GEN{g}\times \GEN{h}$, $|h|=|b_p^{l_p}a_p^{-\frac{l_pp^\rho}{n_p}}|=\frac{n_p}{l_p}=v$ and $m_p=|g|=\frac{|L_p|}{v}=u$.

Otherwise, $n_p>l_pp^{\rho}\ge k_pp^\rho$
and hence $n_p=p^\nu$, by \ref{U3}.
Then $(b_p^{p^{\nu-\rho}}a_p^{-1})^{p^\rho}=1$,  $|g|=|b_p^{l_p}|=\frac{p^{\mu+\nu-\rho}}{l_p}>p^\mu\ge p^\rho$ and, as $|b_p^{p^{\nu-\rho}}a_p^{-1}\GEN{b_p}|=s_p=p^{\sigma}= p^\rho$, by \ref{U2}, it follows that
$\GEN{g}\cap \GEN{h}=\GEN{b_p^{l_p}}\cap \GEN{b_p^{p^{\nu-\rho}}a_p^{-1}}=1$.
Therefore
$L_p=\GEN{g}\times \GEN{h}$, $|h|=p^\rho=v$ and $|g|=\frac{|L_p|}{v}=u$.

\eqref{LpCocyclic} follows at once from \Cref{Cocyclic}.

\eqref{NormalizerK12} Here we use \Cref{Cocyclic} and \Cref{uvt} without specific mention.
Fix an integer such that $2\le w\le m_p+1$ such that $a_p^{b_p}=a_p^w$.
As $\Res_{m_p}(\gamma)=\GEN{1 + r_p}_{m_p}$, we have $v_p(w-1)=v_p(r)\ge \rho$.
Let $(i,y,x)\in C_{L_p}$.
As $[a,L_p]=1$ and $[b_{p'},G_p]=1$, $a,b_{p'}\in N_G(K_{i,y,x})$.
Therefore $N_G(K_{i,y,x})=\GEN{a,b^{p^{\delta}}}$ for some positive integer $\delta$.

Suppose first that $n_p\le l_pp^\rho$.
Then $v=\frac{n_p}{l_p}$, $g^b=g^w$  and $h^b=g^{-(w-1)\frac{l_pp^\rho}{n_p}}h$.
Suppose that $i=1$. Then $y\mid v$, $1\le x\le y$, $[h^{y},b] = h^{-y} (h^b)^{y} = g^{-(w-1)y \frac{l_pp^\rho}{n_p}}\in \GEN{g^{\frac{y}{x_p}}} \subseteq K_{1,y,x}$ and
$[gh^x,b] = g^{(w-1)(1-x\frac{l_pp^\rho}{n_p})} \in \GEN{g^{w-1}}\subseteq \GEN{g^{y}}\subseteq K_{1,y,x}$ because $v_p(w-1)=v_p(r)\ge \rho \ge v_p(n)-v_p(l)= v_p(v)\ge  v_p(y)$.
Thus $K_{1,y,x}$ is normal in $G$, as desired.
Suppose that $i=2$. Thus $y\mid u$ and $\max(p,\frac{y}{v})\mid x$.
Then $[g^{y},b]\subseteq \GEN{g^{y}}\subseteq K_{2,y,x}$.
Therefore  $\delta$ is the minimal non-negative integer satisfying $[g^xh,b^{p^\delta}]\in \GEN{g^y}$.
Using \eqref{Potencia}, we have $[g^xh,b^{p^\delta}] =
g^{x(w^{p^\delta}-1)-(w-1)\frac{l_pp^\rho}{n_p}\Ese{w}{p^\delta}} =
g^{(w-1)\Ese{w}{p^\delta}(x-\frac{l_pp^\rho}{n_p})}$.
On the other hand, $y\mid y \frac{l_pp^\rho}{n_p} = \frac{y p^{\rho}}{v} \mid r x$ and hence $y\mid (w-1)x$.
Therefore $g^{(w-1)\Ese{w}{p^\delta}x}\in \GEN{g^y}=\GEN{g}\cap K_{2,y,x}$, and thus, using \eqref{DenAmr}, we deduce that $[g^xh,b^{p^\delta}]\in \GEN{g^y}$ if and only if
$g^{(w-1)\Ese{w}{p^\delta}\frac{l_pp^\rho}{n_p}}\in \GEN{g^y}$ if and only if
$y\mid p^{\delta} r_p \frac{l_pp^\rho}{n_p}=\frac{p^{\delta+\nu+2\rho}l_p}{n_p^2}=t p^{\delta}$ if and only if $y\mid t$ or $\frac{y}{t}\mid p^{\delta}$.
Therefore $N_G(K_{2,y,x})=\GEN{a,b^{\max(1,\frac{y}{t})}}$, and \eqref{NormalizerK12} follows at once from this equality.

Suppose otherwise that $n_p>l_pp^\rho$. Then $v=p^\rho=r_p$ and $n_p=p^\nu$ so that $\nu-\rho>v_p(l_p)$. Hence $\GEN{a_p^{w-1}}=\GEN{a_p^{r_p}}=\GEN{b_p^{p^\nu}}=\GEN{g^{\frac{p^\nu}{l_p}}}=\GEN{g^t} \subseteq \GEN{g^{p^\rho}}$.
Moreover, $g^b=g$ and $h^b=a_p^{1-w}h$.
Suppose that $i=1$. Then $y\mid v$, $[h^y,b]=a_p^{(1-w)y}\in \GEN{g^{y\frac{p^{\nu}}{l_p}}}$ and
$[gh^x,b]=a_p^{x(1-w)} \in \GEN{g^{p^\rho}} \subseteq \GEN{g^y} \subseteq K_{1,y,x}$.
Therefore $K_{1,y,x}$ is normal in $G$, as desired.
Suppose now that $i=2$. Therefore $y\mid u$. Since $[g^y,b]=1$,
$N_G(K_{2,y,x})=\GEN{a,b^{p^\delta}}$ with $\delta$ as in the previous paragraph, i.e.  the minimal non-negative integer with $[g^xh,b^{p^{\delta}}]\in \GEN{g^y}$.
As $\GEN{[g^xh,b^{p^{\delta}}]}=\GEN{a_p^{(1-w)^{p^{\delta}}}}=\GEN{g^{tp^\delta}}$ it follows that $p^\delta=\max(1,\frac{y}{t})$. Thus $N_G(K_{2,y,x})=\GEN{a,b^{\max(1,\frac{y}{t})}}$.
\end{proof}

If $p\ne 2$ and $K$ is a subgroup of $L$, then $K$ satisfies conditions \ref{KC1}-\ref{KC5} if and only $K_{2'}$ satisfies \ref{KC1}-\ref{KC4} and $[L_2:K_2]\le 2$.
In that case $K_2$ is normal in $G$.
On the other hand, if $P$ be a subgroup of $L_{\pi'}$, then $P\subseteq L_{\pi'} \subseteq Z(L)$ and therefore $N_G(P)=\GEN{a,b^d}$ for some $d\mid l$.
We use this and \Cref{CocyclicLp} to classify the subgroups $K$ of $L$ satisfying conditions \ref{KC1}-\ref{KC4} as follows:
For each $d\mid l$ denote
	$$\mathcal{K}_{d} = \{\text{Cocyclic subgroups } P \text{ of } L_{\pi'} \text{ with } N_G(P)=\GEN{a,b^d} \text{ and } [b^{\frac{l}{q}},a_{\pi'}]\not\in P \text{ for every } q\in \pi(l)\setminus \{p\}\},$$
	$$\mathcal{K}_{d,1} = \{P \in \mathcal{K}_d : [b^{\frac{l}{p}},a_{\pi'}]\not\in P\}	\qand
\mathcal{K}_{d,2} = \{P \in \mathcal{K}_d : [b^{\frac{l}{p}},a_{\pi'}]\in P\}.$$

\begin{remark}\label{K1no1}
Observe that $\GEN{b_{\pi'}^l}\in \mathcal{K}_{1,1}$ because $b_{\pi'}^k\in Z(G)$, $\GEN{a_{\pi'}}\cap \GEN{b_{\pi'}}=1$ and if $q\in \pi(l)$, then $[b^{\frac{l}{q}},a_{\pi'}]\in \GEN{a_{\pi'}}\setminus \{1\}$.
\end{remark}

A subgroup $K$ of $L$ satisfies \ref{KC1}-\ref{KC4} if and only if $K_{\pi\setminus \{p,2\}}=L_{\pi\setminus \{p,2\}}$, $K_{\pi'}\in \mathcal{K}_d$ for some $d\mid l$,  $K_p=K_{(i,y,x)}$ for some $(i,y,x)\in C_{L_p}$ and if $K_{\pi'}\in \mathcal{K}_{d,2}$, then $[b^{\frac{l}{p}},a]\not\in K$. In that case, by \Cref{CocyclicLp}, we have
\begin{equation}\label{NGPQ}
N_G(K)
= \begin{cases}
\GEN{a,b^{\lcm(d,\frac{y}{t})}}, & \text{if } i=2 \text{ and } y>t;\\
\GEN{a,b^d}, & \text{otherwise}.
\end{cases}
\end{equation}
Combining this information with \Cref{Componentsmp} and having in mind that the number of conjugates of $K$ in $G$ is $[G:N_G(K)]$, we obtain the following formula for the  number $N_G$ of Wedderburn components of $\Q G$ satisfying conditions \ref{C1}-\ref{C3}.

\begin{equation}\label{NG}
N_G=O\sum_{d\mid l} (|\mathcal{K}_{d,1}|M(d) + |\mathcal{K}_{d,2}|N(d)).
\end{equation}
where
$$O = \begin{cases}
\text{number of subgroups of } L_2 \text{ of index at most } 2, & \text{if } p\ne 2;  \\
1, & \text{if } p=2.\end{cases}$$
and $M(d)$ and $N(d)$ are defined as follows:
first let
\begin{eqnarray*}
	M_1 &=& \text{number of elements } (1,y,x) \text{ in } C_{L_p} \qand \\
	N_1 &=& \text{number of elements } (1,y,x) \text{ in }  C_{L_p} \text{ with } [b_p^{\frac{l}{p}},a_p]\not\in K_{2,y,x}.
\end{eqnarray*}
Then for each $y\mid v$ let
\begin{eqnarray*}
	M'_y &=& \text{number of elements } (2,y,x) \text{ in }  C_{L_p} \qand \\
	N'_y &=& \text{number of elements } (2,y,x) \text{ in }  C_{L_p} \text{ with } [b_p^{\frac{l}{p}},a_p]\not\in K_{2,y,x}, \\
\end{eqnarray*}
Finally set
$$M_2 = \sum_{y\mid t} M'_y, \quad N_2 = \sum_{y\mid t} N'_y$$
and
$$M(d) = \frac{M_1+M_2}{d} +
\sum_{y,pt\mid y\mid u} \frac{M'_y}{\lcm\left(d,\frac{y}{t}\right)} \qand
N(d) = \frac{N_1+N_2}{d} +
\sum_{y,pt\mid y\mid u} \frac{N'_y}{\lcm\left(d,\frac{y}{t}\right)}.$$

The next goal consists in expressing $M(d)$ and $N(d)$ in terms of $d$ and $v$. Clearly,
	$$M_1= \frac{pv-1}{p-1}; \quad M'_y= \begin{cases} \frac{y}{p}, & \text{if } p\mid y \mid v; \\ v, & \text{otherwise}; \end{cases} \qand M_2= \sum_{p\mid y \mid v} \frac{y}{p} + \sum_{pv\mid y \mid t} v = \frac{v-1}{p-1}+v\log_p\left(\frac{t}{v}\right).$$
Moreover, by \Cref{uvt}, $v\mid t$ and therefore if $pt\mid y$, then $M'_y=v$. Thus
\begin{eqnarray*}
	\sum_{pt\mid y \mid u} \frac{M'_y}{\lcm\left(d,\frac{y}{t}\right)} &=&
	\frac{v}{d_{p'}} \left(\sum_{pd_pt\mid y \mid u} \frac{t}{y}+ \sum_{pt\mid y \mid \min(td_p,u)}\frac{1}{d_p} \right) \\
	&=&
	\begin{cases}
		\frac{vt}{d_{p'}u} \sum_{z\mid \frac{u}{ptd_p}} z + \frac{v}{d}\sum_{pt\mid y \mid td_p}1, & \text{if } td_p<u; \\
		\frac{v}{d}\sum_{pt\mid y \mid u} 1, & \text{otherwise};
	\end{cases}	\\
	&=&
	\begin{cases}
		\frac{vt}{d_{p'}u} \frac{\frac{u}{td_p}-1}{p-1}+ \frac{v}{d}(1+\log_p(d_p)), & \text{if } td_p<u; \\
		\frac{v}{d}\log\frac{u}{t}, & \text{if } t< u \le td_p; \\
		0, & \text{if } u\le t;\end{cases}	\\
	&=&
	\begin{cases}
		\frac{v}{du}\left( \frac{u-d_pt}{p-1}+ (1+\log_p(d_p))\right), & \text{if } td_p<u; \\
		\frac{v}{d}\log\frac{u}{t}, & \text{if } t\le u \le td_p; \\
		0, & \text{if } u<t.
	\end{cases}	.
\end{eqnarray*}
Therefore, if $td_p<u$, then
\begin{eqnarray*}
	dM(d)  &=& \frac{pv-1}{p-1} + \frac{v-1}{p-1} + v\log_p\left(\frac{t}{v}\right) +
	\frac{v}{u}\left( \frac{u-d_pt}{p-1}+ (1+\log_p(d_p))\right) \\
	&=& 1+(p+1) \frac{v-1}{p-1} +
	v\log_p\left(\frac{p^{\nu+2\rho}}{v^3l_p}\right) + \frac{l_pv^2}{p^{\mu+\nu}} \left( \frac{\frac{p^{\mu+\nu}}{vl_p}-d_p\frac{p^{\nu+2\rho}}{v^2l_p}}{p-1}+ (1+\log_p(d_p))\right)\\
	&=& 1+(p+1) \frac{v-1}{p-1} +
	v(\nu+2\rho-v_p(v^3l))+ \frac{v-d_pp^{2\rho-\mu}}{p-1} +
	\frac{l_pv^2}{p^{\mu+\nu}}(1+\log_p(d_p)),
\end{eqnarray*}
if $t\le u \le td_p$, then
    $$dM(d)=\frac{pv-1}{p-1} + \frac{v-1}{p-1} + v\log_p\left(\frac{t}{v}\right)+v\log_p\left(\frac{u}{t}\right)=1+(p+1) \frac{v-1}{p-1} + v(\mu+\nu-v_p(v^2l)),$$
and, if $u<t$, then
	$$dM_d=\frac{pv-1}{p-1} + \frac{v-1}{p-1} + v\log_p\left(\frac{t}{v}\right)=1+v(p+1)+v\log$$
Summarizing, we obtain the following formula for $M(d)$
	\begin{equation}\label{M}
	M(d)=\frac{1}{d}(f_d(v)+h_d(v))
	\end{equation}
where
	$$(f_d(v),h_d(v)) = \begin{cases}
	\left(v\left(\frac{p+2}{p-1} +
	\nu+2\rho-v_p(v^3l)\right),
	\frac{l_pv^2}{p^{\mu+\nu}}(1+\log_p(d_p))-\frac{2+d_pp^{2\rho-\mu}}{p-1}\right), & \text{if } td_p<u; \\
	\left( v\left( \frac{p+1}{p-1} + \mu+\nu-v_p(v^2l)\right),-2\right), & \text{if } t\le u \le td_p; \\
	(v\left(\frac{p+1}{p-1} +
	\nu+2\rho-v_p(v^3l)\right),0), & \text{if } u<t. \end{cases}
	$$

With the aim to obtain a formula for $N(d)$ we first prove the following claim:
    $$N_1=0, \qand N'_y = \begin{cases} v, & \text{if } k_p\le p^{\mu-\rho} \text{ and } y=u; \\ 0, & \text{otherwise}.\end{cases}$$
This is clear if $k_p>p^{\mu-\rho}$, because in that case $[b_p^{\frac{l}{p}},a_p]=1$.
So, suppose that $k_p\le p^{\mu-\rho}$. Then $1<l_p=p^{\mu-\rho}=\frac{m_p}{r_p}$.
Moreover, $v<u\ge t$, by \Cref{uvt}.\eqref{o=kvut}.
The first implies that $|[b_p^{\frac{l}{p}},a_p]|=|a_p^{r_p\frac{m_p}{pr_p}}|=p$ and from $v<u$ we get $N'_u=|\{x\ge 1 : \frac{u}{v}\mid x \le u\}|=v$.
Let $(i,y,x)\in C_{L_p}$.
Using $[G,a_p]=\GEN{a_p^{r_p}}\subseteq \GEN{g}$ and \Cref{Cocyclic}, it follows that $[b_p^{\frac{l}{p}},a_p]\not\in K_{i,y,x}$ if and only if $\GEN{a_p^{r_p}}\cap K_{i,y,x}=1$ if and only if $\GEN{g}\cap K_{i,y,x}=1$,  if and only if either $(i,y,x_p)=(1,u,1)$ or $(i,y)=(2,u)$.
However, if $i=1$ and $y=u$, then $u=y\mid v<u$, a contradiction.
Thus, $N_1=0$ and if $y\ne u$, then $N'_y=0$. This finishes the proof of the claim.

Having in mind that $\frac{u}{t}=p^{\mu-2\rho}v$, the previous claim yields the following formula for $N(d)$:
\begin{equation}\label{N}
N(d) = g_d(v) =
\begin{cases}
\frac{v}{d}, & \text{if } k_p\le p^{\nu-\rho} \text{ and } d_p\ge \frac{u}{t};\\
\frac{1}{d_{p'}}p^{2\rho-\mu}, & \text{if } k_p\le p^{\nu-\rho} \text{ and }
d_p< \frac{u}{t}; \\ 0, & \text{otherwise}.
\end{cases}
\end{equation}
Combining \eqref{NG}, \eqref{M} and \eqref{N} we obtain
\begin{equation}\label{NGv}
N_G=O\sum_{d\mid l} \left(\frac{|\mathcal{K}_{d,1}|}{d}(f_d(v)+h_d(v))+|\mathcal{K}_{d,2}|g_d(v)\right).
\end{equation}
Now observe that $h_d$ and $g_d$ are increasing functions for $v>0$.
Moreover, a straightforward computation shows that
	$$f_d(pv)-f_d(v) = \begin{cases} v(p-1)(\nu-v_p(vl)+2(\rho-v_p(v))), & \text{if } t\le u;\\
v(p-1)(\nu-v_p(vl)+2(\rho-v_p(v))-1), & \text{otherwise}. \end{cases}$$
If $v$ is a proper divisor of $\min\left(\frac{p^\nu}{l_p},p^\rho\right)$, then $\nu-v_p(vl)+2(\rho-v_p(v))-1>0$ and hence the previous calculation shows that  $f_d(pv)>f_d(v)$.
These shows that $f_d(v)$ is a increasing function on the set of divisors of $v$ of  $\min\left(\frac{p^\nu}{l_p},p^\rho\right)$.
This, together with formula \eqref{NGv} and the facts that $O>0$, $\mathcal{K}_{1,1}>0$ (see \Cref{K1no1}), $\mathcal{K}_{1,2}\ge 0$
and $h_d(v)$ and $g_d(v)$ are non-decreasing functions for $v>0$ shows that $v$ is determined by $N_G$ and, hence it is determined by $\Q G$.

To complete the proof of \Cref{mnrDetermined}, it only remains to show that $v$ determines $n_p$, and for that it is enough to show that $n_p$ is the unique positive integer satisfying the following conditions: $q\mid p^\nu$, $v=\min\left(\frac{q}{l_p},p^\rho\right)$ and, if $v=p^\rho$, then $q=p^\nu$. Indeed, observe that $n_p$ satisfies these conditions by \Cref{uvt}.
By means of contradiction let $q$ satisfy the conditions with $n_p\ne q$.
Then $\min(q,n_p)\ne p^\nu$  and hence $v<p^\rho$.
Thus $\min\left(\frac{\max(q,n_p)}{l_p},p^{\rho}\right) = v = \frac{\min(q,n_p)}{l_p}< \frac{\max(q,n_p)}{l_p}$ and hence $v=\min\left(\frac{\max(q,n_p)}{l_p},p^{\rho}\right)=p^\rho$, a contradiction.
This finishes the proof of \Cref{mnrDetermined}.
\end{proof}

\section{$\Q G$ determines $\Delta^G$}\label{SectionDeltaDetermined}

In this section we complete the proof of \Cref{Main} by proving the following proposition.

\begin{proposition}\label{DeltaDetermined}
    Let $G$ and $H$ be finite metacyclic groups such that $\Q G\cong \Q H$. Then $\Delta^G=\Delta^H$.
\end{proposition}

\begin{proof}
    Suppose that $G$ and $H$ satisfy the hypothesis of the proposition.
    By the results of the previous sections we can use a common notation for most of the invariants of $G$ and $H$: $m=m^G=m^H$, $n=n^G=n^H$, $s=s^G=s^H$, $r=r^G=r^H$,
    $R=R^G=R^H$, $k=k^G=k^H$, $\epsilon=\epsilon^G=\epsilon^H$ and $m'$ is defined as explained in \Cref{SSecMetacyclic}. As $m'$ is only depending on $m,n,r,s$ and $k$, it is the same for $G$ and $H$.
    We abuse the notation by  denoting with the same symbols the generators $a$ and $b$ of minimal metacyclic factorizations $G=\GEN{a}\GEN{b}$ and $H=\GEN{a}\GEN{b}$.

    As $\Delta^G$ and $\Delta^H$ are cyclic subgroups of $\U_{m'}$ to prove that they are equal we can we work prime by prime, i.e. we will prove that $(\Delta^G)_p=(\Delta^H)_p$ for every prime $p$.
    If $p\in \pi'$ or $m'_p=1$, then, by \Cref{RDetermined}, $(\Delta^G)_p=R^G_p=R^H_p=(\Delta^H)_p$. Also, if $r_p\ge m'_p$, then $(\Delta^G)_p=(\Delta^H)_p=1$ and hence we also may assume that $m'_p>r_p$. The latter implies that $r_p>1$.

    So in the remainder of the proof
    $$p\in \pi \qand m'_p>r_p$$
which implies that
$$r_p>1 \qand s_p>1,$$
    	 	 and we have to prove that $(\Delta^G)_p=(\Delta^H)_p$. We will consider three cases and in  each one the proof is going to proceed in the following way: We consider distinguished simple components of $\Q G$ (and $\Q H$) depending on the case. Each component of that kind is going to be of the form $\Q Ge(G,L,K)$ for a fixed subgroup $L$ of $G$ and various subgroups $K$ of $L$ so that $(L,K)$ satisfies the conditions of \Cref{SSPMetacyclic}. We prove that there is at least one component of that kind, parametrized by some particular $K_0$, and then we analyze some properties of the other possible $K$'s yielding such components.
    The arguments for $G$ are valid for $H$ and as $\Q G$ and $\Q H$ are isomorphic, $\Q H$ has another component $\Q He(H,L,K) \cong \Q G e(G,L,K_0)$. Using the description of these algebras in \Cref{SSPMetacyclic} we will obtain the desired conclusion with the help of the Main Theorem of Galois Theory, because $(\Delta^G)_p$ will be identified with the $p$-th part of the Galois group of a certain field extension $\Q_d/F$ where $d$ is common for $G$ and $H$ and $F$ is the center of $S$.
    As some of the arguments are similar in the different cases, we will only explain all the details in Case 1, and in Cases 2 and 3, we only elaborate arguments which are significantly different than in previous cases.

We illustrate the arguments with the following example.

\begin{example}{\rm
Consider the following groups
$$G=\GEN{a,b \mid a^{63}=1, b^3=a^{21}, a^b=a^4}, \quad
H=\GEN{a,b \mid a^{63}=1, b^3=a^{21}, a^b=a^{58}}.$$
In both cases $m_{\pi'}=\{7\}$, $\pi=\{3\}$ and $\INV(G_3)=[9,3,3,\GEN{4}_9]$. So,  $m_3=m'_3=9$ and $r_3=s_3=k=n_p=3$. So, for $p=3$, we are in Case 2 below and we look at the Wedderburn components satisfying conditions \ref{E1} and \ref{E2} below. The Wedderburn decompositions are as follows:
\begin{eqnarray*}
\Q G &=& \Q \oplus 5\Q_3 \oplus M_3(\Q_7^{2}) \oplus M_3(\Q_{21}^{4}) \oplus \mathbf{M_3(\Q_{63}^{4})}, \\
\Q H &=& \Q \oplus 5\Q_3 \oplus M_3(\Q_7^{2}) \oplus M_3(\Q_{21}^{4}) \oplus \mathbf{M_3(\Q_{63}^{58})}.
\end{eqnarray*}
They differ in exactly one component which is the only one satisfying \ref{E1} and \ref{E2}. Their centers are $\Q_{63}^{4}$ and $\Q_{63}^{58}$, respectively and their Galois correspondents are precisely $\GEN{4}_{63}=\Gal(\Q_{m'}/F^G)$ for $G$ and $\GEN{58}_{63}=\Gal(\Q_{m'}/F^H)$, which are actually the groups $\Delta^G$ and $\Delta^H$.}
    \end{example}

    We work most of the time with $G$, which is given by the presentation in \eqref{Presentation} and simplify the notation for this group by setting $\Delta=\Delta^G$.

    \medskip
    \noindent\textbf{Case 1}: Suppose that $\epsilon^{p-1}=1$ and $s_p\ge m'_p$.

    \medskip\refstepcounter{claimcounter}
    \noindent\underline{Claim \theclaimcounter}. \label{Claim1}
    In this case $m'_p=\min(m_p,k_pr_p,\max(r_p,s_p,r_p\frac{k_ps_p}{n_p}))$, $r_p\le s_p$ and one of the following hold $k_pr_p\le n_p$ or $s_p=m_p$.

    \begin{proof}
        The first equality follows from the definition of $m'$ \eqref{m'}.
        Assume $r_p > s_p$. By \Cref{Parameters}.\eqref{Param+}, $s_p\leq n_p <k_p s_p$.
        Then, $\max(r_p, s_p, r_p\frac{k_ps_p}{n_p}) = \frac{k_ps_p}{n_p}r_p \leq k_pr_p$, so $m_p' = \min(m_p, \frac{k_ps_p}{n_p}r_p)$. If $m_p' = m_p$, then $s_p \geq m_p' = m_p$ so $s_p = m_p$, which yields the contradiction $m_p \geq r_p > s_p = m_p$. Otherwise, $m_p' = \frac{k_ps_p}{n_p}r_p$ that also yields a contradiction because $s_p \geq \frac{k_ps_p}{n_p}r_p > r_p > s_p$.
        So we have proved that $s_p\geq r_p$.
        Now let us assume $k_pr_p>n_p$ and $s_p<m_p$. Then, $m_p'\neq m_p$, as $s_p\geq m_p'$. In addition, $s_p< \frac{k_ps_p}{n_p}r_p \leq k_pr_p$, so $\max(r_p,s_p,\frac{k_ps_p}{n_p}r_p) = \frac{k_ps_p}{n_p}r_p$ and $m_p'=\frac{k_ps_p}{n_p}r_p > s_p \geq m_p'$, again a contradiction.
    \end{proof}

    In this case we fix the following notation
    $$c=\lcm\left(k,\frac{s_p}{r_p}\right) \qand L=\GEN{a,b^c},$$
    and the distinguished Wedderburn components $S$ of $\Q G$ are those with a center isomorphic to a subfield $F$ of $\C$ satisfying the following conditions:
    \begin{itemize}
        \item[\mylabel{D1}{(D1)}] $F$ is contained in $\Q_{m_{\pi'}s_p}$ and $\Deg(S)=[\Q_{m_{\pi'}s_p}:F]=c$.
        \item[\mylabel{D2}{(D2)}] $F\cap \Q_{m_{\pi'}}=(\Q_{m_{\pi'}})^R$ and $F\cap \Q_{s_p}=\Q_{r_p}$.
    \end{itemize}

    We first show that such Wedderburn component occurs.

    \medskip\refstepcounter{claimcounter}
    \noindent\underline{Claim \theclaimcounter}. \label{Claim2} If $K_0=\GEN{a_{\pi\setminus \{p\}},b^{c}}$, then $(L,K_0)$ is a strong Shoda pair of $G$ and $S=\Q Ge(G,L,K_0)$ satisfies \ref{D1} and \ref{D2}.
\begin{proof}
Indeed, first of all  $[b^{c},a]=[b_{\pi}^{c},a_{\pi}]\in \GEN{a_{\pi\setminus \{p\}},a_p^{\max(k_pr_p,s_p)}} \subseteq K_0$, because $a_p^{s_p}=(b_p^{c})^{n_p/\max(k_p,s_p/r_p)}\in K_0$, by \eqref{Presentation}. This proves that $K_0$ is normal in $G$. Moreover, $L=\GEN{a,K_0}$, and hence $L/K_0$ is cyclic. In addition, $K_0\cap \GEN{a}=\GEN{a_{\pi\setminus \{p\}},a_p^{s_p}}$.
Finally, every cyclic subgroup of $G$ not contained in $L$ is generated by an element of the form $x=a^ib^t$ for $t$ a proper divisor of $c$. If $k\nmid t$ , then $1\ne [x,a_{\pi'}]\in \GEN{a_{\pi'}}$ and hence $[x,a_{\pi'}]\not\in K_0$.
Otherwise $\frac{s_p}{r_p}\nmid t$ and hence $\GEN{[x,a_p]}=\GEN{a_p^{t_pr_p}}\not\subseteq K_0$.
This proves that $(L,K_0)$ is a strong Shoda pair of $G$, by \Cref{SSPMetacyclic}.
Now we take $e=e(G,L,K_0)$ and $S=\Q Ge$. As $K_0$ is normal in $G$ and $L/K_0$ is cyclic of order $m_{\pi'}s_p$ and generated by $aK_0$, it follows that $S$ is isomorphic to a cyclic algebra $(\Q_{m_{\pi'}s_p},\Res_{m_{\pi'}s_p}(\gamma),a)$.
Therefore $\Deg(S)=[G:L]=c$ and its center is isomorphic to $F=\Q_{m_{\pi'}s_p}^{\Res_{m_{\pi'}s_p}(\gamma)}$, which clearly satisfies \ref{D1} and \ref{D2}.
\end{proof}

The next goal is to describe the Wedderburn components of $\Q G$ satisfying conditions \ref{D1} and \ref{D2}.
Let $S$ be such a component.
By \Cref{SSPMetacyclic} and condition \ref{D1}, $S=\Q Ge(G,L,K)$ for a subgroup $K$ of $L$ such that the conditions of \Cref{SSPMetacyclic} hold.
Then the following statements hold where $C=\Core_G(K)$:
\begin{itemize}
	\item[\mylabel{PCI-Cp1}{(V1)}] $a_{\pi\setminus \{p\}}^4,b_{p'}^{4c}\in \Core_G(K)$.
	\item[\mylabel{PCI-Cp2}{(V2)}] If either $v_2(|a|)\le 1$ or $\GEN{a}$ has an element of order 4 which is central in $G$, then $a_{\pi\setminus \{p\}}^2\in C$.
	\item[\mylabel{PCI-Cp3}{(V3)}] If $a_{\pi\setminus \{p\}}^2 \in C$ or $k$ is even, then $b_{p'}^{2c}\in C$.
	\item[\mylabel{PCI-Cp4}{(V4)}] $a_p^{\max(k_pr_p,s_p)}\in K$.
\end{itemize}
Indeed, suppose that $F$ has a root of unity of order $q$ with $q$ prime. The hypothesis $F\subseteq \Q_{m_{\pi'}s_p}$ implies that $q\in \pi'\cup \{2,p\}$. However, the hypothesis $F\cap \Q_{m_{\pi'}}=(\Q_{\pi'})^R$ implies that $q\not\in \pi'$. Indeed, by \cite[Lemma~3.1(2)]{GarciadelRioClasification}, $\pi(A\cap Z(G))\cap \pi'=\emptyset$ and hence $\Q(\pi')^R$ does not have roots of unity of odd prime order. Thus if $q\in \pi'$ then the hypothesis $F\cap \Q_{m_{\pi'}}=\Q_{\pi'}^R$ implies that $F$ does not have a root of unity of order $q$.
Therefore $q\in \{2,p\}$ and hence \ref{PCI-Cp1}-\ref{PCI-Cp3} follow directly from Lemma~\ref{PCIOdd-A2} and Lemma~\ref{PCI2-A1}. Statement \ref{PCI-Cp4} is easy because $a_p,b_p^{\max(k_p,s_p/r_p)}\in L$ and hence, since $L$ and $K$ satisfy the conditions of \Cref{SSPMetacyclic}, $a_p^{\max(k_pr_p,s_p)}=[b^{\max(k_p,s_p/r_p)},a_p]\in K$.

Let $D=\GEN{a_{\pi\setminus \{p\}}^4,a_p^{\max(k_pr_p,s_p)},b_{p'}^{4c}}$. Observe that $D$ is normal in $G$ because $[b_{p'}^{4c},a]=[b_{p'}^{4c},a_{\pi\setminus \{p\}}]\in \GEN{a_{\pi\setminus \{p\}}^{4}}\subseteq D$.

\medskip\refstepcounter{claimcounter}\label{Knormalnew}
\noindent\underline{Claim \theclaimcounter}.
$L_p/K_p$ is generated by either $a_pK$ or $b_p^cK$, $K$ is normal in $G$ and $\varphi([L:K])=\varphi(m_{\pi'}s_p)$.
\begin{proof}
As $L/K$ is cyclic, so is $L_p/K_p\cong (L/K)_p=L_pK/K$. Moreover, $L_p=\GEN{a_p,b_p^{c}}$, and hence $(L/K)_p=L_pK/K$ is generated by either $a_pK$ or $b_p^cK$.

We now prove that $K$ is normal in $G$. Observe that $L$ is nilpotent because $a_{\pi'}\in Z(L)$ and $[b_{\pi'},a_{\pi}]=1$.
As $D$ is normal in $G$ and $D\subseteq K$ we may assume without loss of generality that $D=1$.
Then $L_{\pi'}=\GEN{a_{\pi'}}$ and $\pi\subseteq \{2,p\}$.
Using that $L$ is nilpotent, $K\unlhd L$ and $a\in L$, it is enough to prove that $b_q$ normalizes $K_l$ for every pair of primes $q$ and $l$.
This is clear if $l\in \pi'$ because $L_{\pi'}=\GEN{a_{\pi'}}$.
It is also clear if $q\in\pi'$ because $[b_{\pi'},a_{\pi}]=1$.
Moreover as $G_{\pi}$ is nilpotent the result is also clear if $q\ne l$ and $q,l\in \pi$.
It remains to consider the case $q=l\in \pi$.

Let us first consider the case $q=l\ne p$, i.e. $q=l=2$. Then $L_2$ is either abelian of exponent at most $4$ or $L_2$ is either $D_8$, $Q_8$ or $C_4\rtimes C_4$ with $L_2'=\GEN{a_2^2}\subseteq K_2$. Suppose that $L_2$ is non-abelian and recall that $L=\GEN{a,b^c}$, so that $L_2=\GEN{a_2,b_2^c}$. Then $|a_2|=4$, $a_2^{b_2}=a_2^2$ and $b_2$ commutes with $\GEN{b_2^c}$ and normalizes $L_2$  and every subgroup of $\GEN{a_2}$.  $K_2$ is a proper subgroup of $L_2$ containing properly $\GEN{a_2^2}$. If $L_2\cong D_8$ then $K_2$ is either $\GEN{a_2}$, $\GEN{a_2^2,b_2^c}$ or $\GEN{a_2^2,a_2b_2^c}$. If $L_2\cong Q_8$ then $K_2$ is either $\GEN{a}$, $\GEN{b_2^c}$ or $\GEN{ab_2^c}$. Finally, if $L_2\cong C_4\rtimes C_4$, then $K_2$ is either $\GEN{a_2}$, $\GEN{a_2b_2^2}$ or $\GEN{a_2,b_2^2}$. In all the cases $b_2$ normalizes $K_2$.  Suppose otherwise that $L_2$ is abelian. Now, by assumption $a_2^4=1$ so it follows that $G_2$ is either abelian or $|a_2|=4$ and $a_2^{b_2}=a_2^{-1}$. In the first case clearly $b_2$ normalizes every subgroup of $L_2$.
In the second case either $a_2^2\in K$ or $c$ is even and $K_2\subseteq \GEN{a_2^2,b_2^c}\subseteq C_G(b_2)$.
In both cases $K_2$ is normalized by $b_2$.

It remains to show that $b_p$ normalizes $K_p$. Otherwise $p\mid d$, where $d=[G:N]$ and $N=N_G(K)=\GEN{a,b^d}$.
As $L/K$ is cyclic, $L=\GEN{u,K}$ for some $u$ and $S\cong M_{[G:N]}(\Q_h*\GEN{\rho})$ with $h=[L:K]$ and $\rho(\zeta_h)=\zeta_h^x$, if $u^{b^d}K=u^xK$.
Moreover, as $L_p=\GEN{a_p,b_p^{c}}$, $(L/K)_p=L_pK/K$ is generated by either $a_pK$ or $b_p^cK$. So we may assume that $u_p$ is either $a_p$ or $b_p^c$.
In the second case, $(L/K)_p$ is central in $(N_G(K)/K)_p$ and hence $h_p\le s_p$ by condition \ref{D1} and $h_p=r_p$ by condition \ref{D2}. In the first case,
$\GEN{x}_{h_p}=\GEN{(1+r_p)^d}_{h_p}$ and if $r_p<h_p$, then $v_p((1+r_p)^d-1)>v_p(r_p)$ as $p\mid d$.  Hence if $p r_p$ divides $h_p$, then $F$ has a central element of order $p r_p$ in contradiction with \ref{D2}. This proves that $h_p$ divides $r_p$. Therefore $a_p^{r_p}\in K$ so that $[b_p,a_p]\in K_p$. This implies that $[b_p,G_p]\subseteq K_p$ and hence $b_p$ normalizes $K_p$, as desired. This finishes the proof of the normality of $K$ in $G$.

As $K$ is normal in $G$, $S\cong \Q_{[L:K]}*G/L$ and $G/L=\GEN{bL}$ where the action $\rho$ of the crossed product is given by $\rho(\zeta_{[L:K]})=\zeta_{[L:K]}^x$, if $L=\GEN{u,K}$ and $u^b\in u^xK$.
Moreover, $[G:L]=\Deg(S)=c=[\Q_{m_{\pi'}s_p}:F]=\frac{\varphi(m_{\pi'}s_p)}{\dim_{\Q}(F)}$.
Thus
$$c \varphi([L:K])=[G:L]\varphi([L:K])=\dim_{\Q}(S)=c^2\dim_{\Q} F = c\varphi(m_{\pi'}s_p)$$
and therefore $\varphi([L:K])=\varphi(m_{\pi'}s_p)$, as desired.
\end{proof}

    \medskip\refstepcounter{claimcounter}
    \noindent\underline{Claim \theclaimcounter}.  \label{psDividesHKnew}
    One of the following conditions holds:
    \begin{enumerate}[label=(\roman*)]
        \item $[L:K]=m_{\pi'}s_p$ or $p\ne 2$ and $[L:K]=2m_{\pi'}s_p$,
        \item $p\ne 2$ and $[L:K]=4\frac{m_{\pi'}}{q}p^t$, with $p^t\ge s_p$, $q=1 + 2\frac{p^t}{s_p}\in \pi'$ and $v_q(m_{\pi'})=1$.
        \item $p=2$ and $[L:K]=\frac{m_{\pi'}}{q_1\cdots q_l}2^t$, with $l\ge 1$, $2^t>s_2$, $q_1,...,q_l\in \pi'$, $\varphi(q_1\cdots q_l)=\frac{2^t}{s_2}$ and $v_{q_i}(m_{\pi'})=1$ for every $i$.
    \end{enumerate}

\begin{proof}
By \ref{PCI-Cp1}, $[L:K]_{\pi'}$ divides $m_{\pi'}$.
Let $[L:K]_p=p^t$,  $[L:K]_{\pi'}=p_1^{\alpha_1}\cdots p_w^{\alpha_w}$ and $m_{\pi'}=p_1^{\beta_1}\cdots p_w^{\beta_w}q_1^{\gamma_1}\cdots q_l^{\gamma_l}$ with $p_1,\dots,p_w,q_1,\dots,q_l$ the different elements of $\pi'$.

Let us first assume that $p\neq 2$. By \ref{PCI-Cp1}, $[L:K]_2$ divides $4$.
By condition \ref{D2}, $r_p\mid p^t$ and hence $t\ge 1$.
By Claim~\ref{Knormalnew}, $\varphi([L:K])=\varphi(m_{\pi'}s_p)$ and therefore
	$$\varphi([L:K]_2)p^tp_1^{\alpha_1-1}\cdots p_w^{\alpha_w-1} = s_pp_1^{\beta_1-1}\cdots p_w^{\beta_w-1}q_1^{\gamma_1-1}\cdots q_l^{\gamma_l-1}(q_1-1)\cdots (q_l-1).$$
Then $\alpha_i=\beta_i$ for every $i=1,\dots,w$ and as $[L:K]_2\mid 4$ and $q_i$ is odd for every $i$ it follows that $s_p\le p^t$ and either $l=0$, $s_p=p^t$ and $[L:K]_2\mid 2$ or $[L:K]_2=4$, $l=1$, $q_1=1+2\frac{p^t}{s_p}$ and $\gamma_1=1$. This proves that either (i) or (ii) holds.

Now, let us consider the case when $p=2$. Then the equality $\varphi([L:K])=\varphi(m_{\pi'}s_2)$ yields
	$$2^{t}p_1^{\alpha_1-1}\cdots p_w^{\alpha_w-1}=s_2p_1^{\beta_1-1}\cdots p_w^{\beta_w-1}q_1^{\gamma_1-1}\cdots q_l^{\gamma_l-1}(q_1-1)\cdots (q_l-1).$$
Again $\alpha_i=\beta_i$ for every $i=1,...,w$ and, as each $q_i$ is odd, $\gamma_i=1$ for every $i=1,...,l$. So the equation gets reduced to:
	$$2^t=s_2(q_1-1)\cdots (q_l-1).$$
If $l=0$, then condition (i) holds, and otherwise condition (iii) holds.
\end{proof}

Now we have enough information to prove that  $(\Delta^G)_p = (\Delta^H)_p$ in this case.
Recall that both $G$ and $H$ are given by a presentation as in \eqref{Presentation} with the same parameters $m,n,s$ but now the automorphism $\gamma$ differs for $G$ and $H$.
We denote them $\gamma^G$ and $\gamma^H$ respectively.
We know that $R^G=R^H$, i.e. $\GEN{\Res_{m_{\pi'}}(\gamma^G)}=\GEN{\Res_{m_{\pi'}}(\gamma^H)}$ and hence we may assume without loss of generality that $\Res_{m_{\pi'}}(\gamma^G)=\Res_{m_{\pi'}}(\gamma^H)$.

In the remainder of the proof we will consider restrictions of $\gamma^G$ and $\gamma^H$ to several cyclotomic fields $\Q_d$ with $d\mid m$. For shortness we will abuse the notation and simplify $\Q_d^{\Res_d(\gamma^G)}$ by writing $\Q_d^{\gamma^G}$, and similarly for $H$.
We fix the Wedderburn component $S=\Q Ge(G,L,K_0)$ of $KG$ with $K_0$ as in Claim~\ref{Claim2}.
Then the center of $S$ is isomorphic to $\Q_{m_{\pi'}s_p}^{\gamma^G}$.
As $\Q G\cong \Q H$, $\Q H$ has a Wedderburn component $\Q He(H,L,K)$ isomorphic to $S$ with $(L,K)$ a strong Shoda pair of $H$.
By Claim~\ref{Knormalnew}, $K$ is normal in $H$.
Moreover, $[L:K]$ satisfies one of the conditions of Claim~\ref{psDividesHKnew} and $(L/K)_p$ is generated by either $a_pK$ or $b_p^cK$.
Let $p^t=[L:K]_p$.
If $(L/K)_p$ is generated by $b_p^cK_p$, then $\Q_{r_p} = \Q_{s_p}^{\gamma^G} = \Q_{p^t}$ and hence $m'_p>r_p = p^t\ge s_p\ge m'_p$, a contradiction.
Therefore, $(L/K)_p$ is generated by $a_pK$.
Thus $(L/K)_{\pi'\cup \{p\}}$ is generated by $a_{\pi'\cup \{p\}}K$. Therefore the center of $\Q Ge(H,L,K)$ is isomorphic to a subfield $F$ of $\Q_{[L:K]}$ such that $F\cap \Q_{[L:K]_{\pi'\cup \{p\}}}=\Q_{[L:K]_{\pi'\cup \{p\}}}^{\gamma^H}$.
Then $\Q_{m_{\pi'}s_p}^{\gamma^G}=\Q_{[L:K]_{\pi'\cup \{p\}}}^{\gamma^H}$.

We deal separately with the three cases of Claim~\ref{psDividesHKnew}.

\underline{Case (i)}.
Suppose first that either $[L:K]=m_{\pi'}s_p$ or $p\ne 2$ and $[L:K]=2m_{\pi'}s_p$.
Then, $\Q_{m_{\pi'}s_p}=\Q_{[L:K]_{\pi'\cup \{p\}}}$ and hence $\Q_{m_{\pi'}s_p}^{\gamma_G}=\Q_{m_{\pi'}s_p}^{\gamma_H}$.
Thus, Galois Theory yields $\GEN{\Res_{m_{\pi'}s_p}(\gamma^G)}=\GEN{\Res_{m_{\pi'}s_p}(\gamma^H)}$.
Since $m'_p\mid s_p$, we have $\Res_{m'_{\pi'\cup \{p\}}}((\Delta^G)_p)=\Res_{m'_{\pi'\cup \{p\}}}(\GEN{\gamma^G}_p)=\Res_{m'_{\pi'\cup \{p\}}}(\GEN{\gamma^H}_p)=\Res_{m'_{\pi'\cup \{p\}}}((\Delta^H)_p)$.
As $\Res_{m'_{\pi\setminus \{p\}}}((\Delta^G)_p)=\Res_{m'_{\pi\setminus \{p\}}}((\Delta^H)_p)=1$. We conclude that $(\Delta^G)_p=(\Delta^H)_p$, as desired.

\underline{Case (ii)}.
Suppose now that $p\ne 2$ and $[L:K]=4\frac{m_{\pi'}}{q}p^t$, with $p^t\ge s_p$, $q=1 + 2\frac{p^t}{s_p}\in \pi'$ and $v_q(m_{\pi'})=1$.
Then
$F=\mathbb{Q}_{4\frac{m_{\pi'}}{q}s_p}^{\alpha} = \mathbb{Q}_{m_{\pi'}s_p}^{\gamma^G}$, where $\Res_{\frac{m_{\pi'}}{q}p^t}(\alpha)=\Res_{\frac{m_{\pi'}}{q}p^t}(\gamma^H)$ and $\zeta_4^{\alpha}=\zeta_4^i$, if $u^b=u^i$ with $L/K=\GEN{uK}$. Thus, $F=\mathbb{Q}_{\frac{m_{\pi'}}{q}s_p}^{\gamma^H} = \mathbb{Q}_{\frac{m_{\pi'}}{q}s_p}^{\gamma^G}$, and by Galois Theory, $\GEN{\Res_{\frac{m_{\pi'}}{q}s_p}(\gamma^G)}=\GEN{\Res_{\frac{m_{\pi'}}{q}s_p}(\gamma^H)}$.
Moreover, by \Cref{SSPMetacyclic}, $\gamma^H(\zeta_4) = \alpha(\zeta_4)=\zeta_4^{-1}$, since $F$ does not have a fourth root of unity, and, as $\Res_{m_{\pi'}}(\gamma^G)=\Res_{m_{\pi'}}(\gamma^H)$, we have $\Q_q^{\gamma^H} =\Q_q^{\gamma^G}=\Q$.
The first implies that $(L/K)_2$ is not generated by $b_2^cK_2$ and hence it is generated by $a_2K_2$, so we may assume that $\alpha=\gamma^H$.

Assume first that $s_p=p^t$. Then $q=3$ and  $F=\mathbb{Q}_{\frac{m_{\pi'}}{3}s_p}^{\gamma^H} = \mathbb{Q}_{\frac{m_{\pi'}}{3}s_p}^{\gamma^G}$, and by Galois Theory, $\GEN{\Res_{\frac{m_{\pi'}}{3}s_p}(\gamma^G)}=\GEN{\Res_{\frac{m_{\pi'}}{3}s_p}(\gamma^H)}$.
Moreover, $\Q_3^{\gamma^G}=\Q_3^{\gamma^H}=\Q$, so that  $\gamma^H(\zeta_3)=\gamma^G(\zeta_3) = \zeta_3^{-1}$.
Then $\Res_{m_{\pi'}s_p}(\GEN{\gamma^G})_p=\Res_{m_{\pi'}s_p}(\GEN{\gamma^H})_p$, as $p$ is odd, and, as in the previous case, we deduce that $(\Delta^G)_p=(\Delta^H)_p$.

Now we assume $p^t>s_p$. Since also $r_p<m'_p \le s_p$, we have $\frac{p^t}{r_p}> \max(\frac{p^t}{s_p},\frac{s_p}{r_p})$. If $m_{\pi'}=q$, then $(\Q_{4p^t})^{\gamma^H}=(\Q_{qs_p})^{\gamma^G}=\Q_{r_p}$ and hence
$k_p=q-1=2\frac{p^t}{s_p}$. Then
$\max\left(\frac{p^t}{s_p},\frac{s_p}{r_p}\right)=c_p=\Deg(S)_p=[\Q_{4p^t}:\Q_{r_p}]_p=\frac{p^t}{r_p}$, a contradiction.
Thus $m_{\pi'} = xq$ with $x>1$.
By \ref{D1} and the description of $S$ (see \Cref{SSPMetacyclic}), we have $[\mathbb{Q}_{4xp^t}:F] = [\mathbb{Q}_{qxs_p}:F]=\Deg(S)=c$.
By looking at the $p$-part of these degrees we obtain:
\begin{equation}\label{maxequal}
\max\left(\bar{k}_p,\frac{p^t}{r_p}\right) = \max\left(\bar{k}_p,\frac{p^t}{s_p}, \frac{s_p}{r_p}\right)=c_p,\end{equation}
where $\bar{k} = |\Res_{x}(\gamma^G))| = |\Res_{x}(\gamma^H))|$.
As $\frac{p^t}{r_p} > \max\left(\frac{p^t}{s_p}, \frac{s_p}{r_p}\right)$, necessarily $\bar{k}_p\ge \frac{p^t}{r_p}$.
So $c_p=\bar{k}_p\ge \frac{p^t}{r_p} > \max\left(\frac{p^t}{s_p}, \frac{s_p}{r_p}\right)$.
As $\GEN{\Res_{xs_p}(\gamma^H)} = \GEN{\Res_{xs_p}(\gamma^G)}$,
$\Res_{xs_p}(\gamma^G)_p = (\Res_{xs_p}(\gamma^H)_p)^{u}$ for certain $u$ coprime with $p$ which we can even choose to be odd.
Then,
    $$\Res_x((\gamma^{H})_p) =\Res_x((\gamma^{G})_p) =
(\Res_x((\gamma^{H})_p))^{u}$$
so $u \equiv 1 \mod |\Res_x(\gamma^H)_p|$. Moreover, $|\Res_x(\gamma^H)_p)|= \bar{k}_p \geq \frac{p^t}{r_p} > \frac{p^t}{s_p}$, so
$u \equiv 1 \mod \frac{p^t}{s_p}$ and as $u$ is odd we have $u\equiv 1 \mod 2$. Thus $u \equiv 1 \mod q-1$, and as $\Res_{m_{\pi'}}(\gamma^G) = \Res_{m_{\pi'}}(\gamma^H)$ we have $\Res_q(\gamma^G) = \Res_q(\gamma^H)=\Res_q(\gamma^H)^u$.
Then, $(\gamma^G)_p=((\gamma^H)_p)^u$, because the equality happens both restricting to $q$ and to $xs_p$. So the $p$-th parts of $\gamma^G$ and $\gamma^H$ generate the same subgroup and as $s_p\ge m'_p$ we obtain the desired conclusion:
$$(\Delta^G)_p = \Res_{m_{\pi'}m'_p}(\GEN{(\gamma^G)_p})=\Res_{m_{\pi'}m'_p}(\GEN{(\gamma^H)_p})= (\Delta^H)_p.$$

\underline{Case (iii)}.
Finally suppose that $p=2$ and $[L:K]=\frac{m_{\pi'}}{q_1\cdots q_l}2^t$, with $l\ge 1$, $2^t>s_2$, $q_1,...,q_l\in \pi'$, $\varphi(q_1\cdots q_l)=\frac{2^t}{s_2}$ and $v_{q_i}(m_{\pi'})=1$ for every $i$.
Let  $x = \frac{m_{\pi'}}{q_1\cdots q_l}$.
Arguing as in the previous case we obtain
$$c_2=\bar{k}_2 \ge \frac{2^t}{r_2} > \max\left(\frac{2^t}{s_2}, \frac{s_2}{r_2}\right) \ge (q_1-1)\cdots (q_l-1),$$
and, from $2^t>s_2$ we get that $b^lK$ does not generates $L/K$ and then that $F=(\Q_{x2^t})^{\gamma^H}=(\Q_{xs_2})^{\gamma^G}$.
Having in mind the previous inequalities the same argument as in the previous case yields the desired conclusion, i.e. $(\Delta^G)_2=(\Delta^H)_2$.

\medskip
\noindent\textbf{Case 2}. Suppose that $\epsilon^{p-1}=1$ and $s_p< m'_p$.

\medskip\refstepcounter{claimcounter}
\noindent\underline{Claim \theclaimcounter}.
In this case $s_p<m_p$, $n_p<k_pr_p$ and $\frac{k_ps_p}{n_p}r_p \ge m_p'$.

\begin{proof}
    As $s_p< m_p'=\min(m_p,k_pr_p,\max(r_p,s_p,r_p\frac{k_ps_p}{n_p}))$, clearly $s_p<m_p$ and $s_p < \max(s_p, r_p, \frac{k_ps_p}{n_p}r_p)$.
    If $s_p < r_p$, then, by \Cref{Parameters}.\eqref{Param+}, we have $n_p <k_p s_p$, so $s_p < r_p < r_p \frac{k_p s_p}{n_p}$, and from this equation we easily get $n_p < k_p r_p$.
    Otherwise, $s_p < r_p \frac{k_p s_p}{n_p}$ and again we get $n_p < k_p r_p$.

Since $s_p<\max(s_p,r_p, r_p \frac{k_p s_p}{n_p})$, $\max(s_p,r_p, r_p \frac{k_p s_p}{n_p})$ is either $r_p$ or $r_p\frac{k_p s_p}{n_p}$. In the later case $m_p' = \min(m_p, k_pr_p,  r_p\frac{k_p s_p}{n_p}) \le  r_p\frac{k_p s_p}{n_p}$, as desired. Otherwise, $r_p >  r_p\frac{k_p s_p}{n_p}$. Then, $k_p s_p < n_p < k_p r_p$, so $s_p < r_p$, in contradiction with \Cref{Parameters}.\eqref{Param+}.
\end{proof}

In this case the distinguished Wedderburn components $S$ of $\Q G$ are those with center isomorphic to a subfield $F$ of $\C$ satisfying the following conditions:
\begin{itemize}
    \item[\mylabel{E1}{(E1)}] $F$ is contained in $\Q_{m_{\pi'}m'_p}$ and $\Deg(S)=[\Q_{m_{\pi'}m'_p}:F]=k$.
    \item[\mylabel{E2}{(E2)}] $F\cap \Q_{m_{\pi'}}=(\Q_{m_{\pi'}})^R$ and $F\cap \Q_{m'_p}=\Q_{r_p}$.
\end{itemize}

We first show that such Wedderburn component exists.
By \Cref{PropEse}.\eqref{vpEse}, $v_p\left(\Ese{1+r_p}{\frac{n_p}{k_p}}\right) =\frac{n_p}{k_p}$. Write $\Ese{1+r_p}{\frac{n_p}{k_p}}=z\frac{n_p}{k_p}$. Fix an integer $y$ such that
$$z\equiv y \frac{k}{k_p} \mod \frac{m_p}{s_p}.$$
As $p\nmid z$ and $s_p<m_p$, we have $p\nmid y$.

\medskip\refstepcounter{claimcounter}
\noindent\underline{Claim \theclaimcounter}.
If $L=\GEN{a,b^k}$ and $K_0=\GEN{a_{\pi\setminus \{p\}},a_p^{\frac{r_ps_pk_p}{n_p}},b^{-yk}a_p^{\frac{s_pk_p}{n_p}}, b_{p'}^k}$, then $(L,K_0)$ is a strong Shoda pair of $G$ and $S=\Q Ge(G,L,K_0)$ satisfies \ref{E1} and \ref{E2}.

\begin{proof} As $p\in \pi$, $[b_p,a_{\pi\setminus \{p\}}]=1$.
Thus $\GEN{[b_p^{-yk}a_p^{\frac{s_pk_p}{n_p}},a]}=\GEN{[b_p^{-yk},a_p]}=\GEN{[b_p^{k_p},a_p]}=\GEN{[a_p^{r_pk_p}} \subseteq \GEN{a_{p}^{\frac{r_ps_pk_p}{n_p}}} \in K_0$, since $s_p\mid n_p$.
    Also, $[b_{p'}^k, a] = [b_{\pi\setminus \{p\}}^k, a_{\pi}] \in \GEN{a_{\pi\setminus \{p\}}} \in K_0$. This proves that $K_0$ is normal in $G$. Moreover, $L=\GEN{a,K_0}$, and hence $L/K_0$ is cyclic.
    In order to prove that $L/K_0$ is maximal abelian in $G/K_0$ observe that $|b_p^{-yk}a_p^{\frac{s_pk_p}{n_p}}\GEN{a}|=|b_p^{yk}\GEN{a}|=\frac{n_p}{k_p}$ and $(b_p^{-yk}a_p^{\frac{s_pk_p}{n_p}})^{\frac{n_p}{k_p}} =b_p^{-yk\frac{n_p}{k_p}} a_p^{\frac{s_pk_p}{n_p} \Ese{1+r_p}{\frac{n_p}{k_p}}} = a_p^{-ys_p\frac{k}{k_p} + zs_p} = 1$.
    Then $|b_p^{-yk}a_p^{\frac{s_pk_p}{n_p}}\GEN{a}|=|b_p^{-yk}a_p^{\frac{s_pk_p}{n_p}}|=\frac{n_p}{k_p}$,
    and therefore $K_0\cap \GEN{a}=\GEN{a_{\pi\setminus \{p\}},a_p^{\frac{r_ps_pk_p}{n_p}}}\subseteq \GEN{a_{\pi}}$.
     Moreover, every cyclic subgroup of $G$ not contained in $L$ is generated by an element of the form $x=a^ib^t$ for $t$ a proper divisor of $k$, and for such $x$,  $1\ne [x,a_{\pi'}]\in \GEN{a_{\pi'}}$.
    Thus $[x,a_{\pi'}]\not\in K_0$.
    This shows that $(L,K_0)$ is a strong Shoda pair of $G$.

    Now we take $e=e(G,L,K_0)$ and $S=\Q Ge$. As $K_0$ is normal in $G$ and $L/K_0$ is cyclic of order $m_{\pi'}m'_p$ and generated by $aK_0$, it follows that $S$ is isomorphic to a cyclic algebra $(\Q_{m_{\pi'}m'_p},\Res_{m_{\pi'}m'_p}(\gamma^G),\zeta_{m_{\pi'}m'_p}^s)$.
    Therefore $S$ satisfies \ref{E1} and \ref{E2}.
\end{proof}

The remaining arguments in this case follow exactly as the previous one, except changing $s_p$ by $m'_p$ where it corresponds.

\medskip
\textbf{Case 3}: Suppose now that $\epsilon^{p-1}=-1$, i.e. $p=2$ and $\epsilon=-1$.
In this case, \Cref{Parameters}.\eqref{Param}
gives us the following properties:
\begin{equation}\label{CondCase3-1}
s \mid m, \quad |\Delta| \mid n, \quad  \frac{m_2}{r_2}\le n_2,  \quad m_2\le 2s_2 \qand s_2\ne n_2r_2.
\end{equation}
\begin{equation}\label{CondCase3-2}
\text{If } 4\mid n, \; 8\mid m \text{ and } k_2<n_2, \text{ then } r_2\le s_2.
\end{equation}
Having in mind that $m'_2\ne r_2$ we have
\begin{equation}\label{CondCase3-3}
4\le k_2 \qand 4r_2\le m_2.
\end{equation}
This together with $s_2\ne n_2r_2$ implies that
$$m'_2=\begin{cases}
\frac{m_2}{2}, & \text{if } k_2<n_2 \text{ and } 2s_2=m_2<n_2r_2; \\
m_2, & \text{otherwise}.
\end{cases}$$

In this case we take
$$c=\lcm\left(k, \frac{m'_2}{r_2}\right) \qand L=\GEN{a, b^c}$$
and the distinguished Wedderburn components $S$ of $\Q G$ are those with a center isomorphic to a subfield $F$ of $\C$ satisfying the following conditions:
\begin{itemize}
    \item[\mylabel{F1}{(F1)}] $F$ is contained in $\Q_{m_{\pi'}m'_2}$ and $\Deg(S)=[\Q_{m_{\pi'}m'_2}:F]=c$.
    \item[\mylabel{F2}{(F2)}] $F\cap \Q_{m_{\pi'}}=(\Q_{m_{\pi'}})^R$ and $F\cap \Q_{m'_2}=(\Q_{m'_2})^\sigma$, where $\sigma(\zeta_{m'_2})=\zeta_{m'_2}^{-1+r_2}$.
\end{itemize}

\medskip\refstepcounter{claimcounter}
\noindent\underline{Claim \theclaimcounter}.
Let
$$K_0=\begin{cases}\GEN{a_{\pi\setminus \{2\}}, b^c}, & \text{if } b_2^c \not\in \GEN{a} \\
\GEN{a_{\pi\setminus \{2\}}, b_{2'}^c}, & \text{ otherwise}. \end{cases}$$
Then $(L,K_0)$ is a strong Shoda pair of $G$ and $S=\Q G e(G,L,K_0)$ satisfies \ref{F1} and \ref{F2}.

\begin{proof}
    We claim that $[L:K_0]=m_{\pi'}m'_2$. Indeed, first of all observe that
    $$K_0\cap \GEN{a}=\begin{cases} \GEN{a_{\pi\setminus\{2\}},a_2^{\frac{m_2}{2}}}, & \text{if } b_2^c\not\in \GEN{a} \text{ and } m_2=2s_2; \\
    \GEN{a_{\pi\setminus \{2\}}}; & \text{otherwise}.\end{cases}$$
    Therefore
    $$[L:K_0]=\begin{cases} m_{\pi'}\frac{m_2}{2}, & \text{if } b_2^c\not\in \GEN{a} \text{ and } m_2=2s_2; \\
    m_{\pi'}m_2; & \text{otherwise}.\end{cases}$$
    Thus, if $m_2\ne 2s_2$ or $k_2=n_2$, then $[L:K_0]=m_{\pi'}m_2=m_{\pi'}m'_2$.
    Suppose that $m_2=2s_2$ and $k_2\ne n_2$. Then $k_2<n_2$ and, as $m_2\le n_2r_2$ we have that $m_2'=m_2$ if and only if $m_2=n_2r_2$. Then $b_2^c\in \GEN{a}$ if and only if $\frac{m'_2}{r_2}\ge n_2$ if and only if $m_2=m_2'$.

    We now prove that $K_0$ is normal in $G$. This is easy if $b_2^c\in \GEN{a}$ because $[b^c_{2'}, a]=[b_{2'}^c,a_{\pi\setminus \{2\}}]\subseteq \GEN{a_{\pi\setminus \{2\}}}\subseteq K_0$.
    Suppose otherwise that $b_2^c\not\in \GEN{a}$.
    If $[b_2^c,a_2]=1$, then $[b^{c},a]=[b_{\pi}^{c},a_{\pi}]\in \GEN{a_{\pi\setminus \{2\}}} \in K_0$.
    Otherwise, $\frac{m_2}{r_2}>c_2$ and, recalling that we are assuming that $m'_2\ne r_2$, it follows that $m'_2=\frac{m_2}{2}$, $c_2=\frac{m_2}{2r_2}$ and $s_2=\frac{m_2}{2}$. Using \Cref{PropEse}.\eqref{vpRm-1}, it follows that $[b_2^c,a]=a_2^{\frac{m_2}{2}}\in \GEN{b_2^c}\subseteq K_0$.

    Moreover, $L=\GEN{a,K_0}$, and hence $L/K_0$ is cyclic.
    In order to prove that $L/K_0$ is maximal abelian in $G/K_0$, we argue by contradiction. So we take $x\in G\setminus L$ and assume that $[x,L]\subseteq K_0$.
    We may assume that $x=a^ib^t$ for $t$ a proper divisor of $c$, as every cyclic subgroup of $G$ not contained in $L$ is generated by an element of this form.
    If $k\nmid t$, then $1\ne [x,a_{\pi'}]\in \GEN{a_{\pi'}}$ and hence $[x,a_{\pi'}]\not\in K_0$, a contradiction.
    Thus $k\mid t$ and hence $\frac{m'_2}{r_2}\nmid t$, i.e. $t_2r_2<m'_2$.
    By assumption, $[x,a_2]\in K_0$. Observe that
    $$\GEN{[x,a_2]} =\begin{cases} \GEN{a_2^2}, &\text{ if } t_2=1;\\
    \GEN{a_2^{t_2r_2}}, & \text{otherwise}. \end{cases}$$
    The assumptions $m'_2\ne r_2$ implies that $m_2>2r_2\ge 4$ and hence $a_2^2\not\in K_0$. Thus $1\ne a_2^{t_2r_2}\in K_0$ and hence $b_2^c\not\in \GEN{a}$, $m_2=2s_2$ and $t_2r_2=\frac{m_2}{2}<m'_2$. Therefore $m'_2=m_2$. Since $m'_2\ne r_2$ and $b_2^c\not\in \GEN{a}$, it follows that  $k_2<n_2$. Moreover, $4\le k_2$ and $4r_2\le m_2$, by \eqref{CondCase3-3}. Thus, by the definition of $m'_2$, we have that  $s_2\ne n_2r_2$ and $m_2\ge n_2r_2$. By \Cref{Parameters}.\eqref{Param-}, $m_2=n_2r_2$ and hence $c_2\ge n_2$, so that $b_2^c\in \GEN{a}$, a contradiction.

    Now we take $e=e(G,L,K_0)$ and $S=\Q Ge$. As $K_0$ is normal in $G$ and $L/K_0$ is cyclic of order $m_{\pi'}m'_2$ and generated by $aK_0$, it follows that $S$ is isomorphic to a cyclic algebra $(\Q_{m_{\pi'}m'_2},\Res_{m_{\pi'}m'_2}(\gamma^G),\zeta_{m_{\pi'}m'_2}^s)$.
    Therefore $\Deg(S)=[G:L]=c$ and the center of $S$ satisfies \ref{F1} and \ref{F2}.
\end{proof}

As in Case 1 we now consider an arbitrary Wedderburn component $S$ of $\Q H$ with center isomorphic to a field $F$ and satisfying conditions \ref{F1} and \ref{F2}.
As $\Deg(S)=c$, we have $S=\Q H e(H,L,K)$ for a strong Shoda pair $(L,K)$ of $H$.
As $L/K$ is abelian, $\GEN{a_2^{\max(k_2r_2,m'_2)}}=\GEN{[a_2, b_2^{c}]}\subseteq K$ and combining this with \Cref{PCIOdd-A2} it follows that $K$ contains $D=\GEN{a_{\pi\setminus \{2\}},a_2^{\max(k_2r_2,m'_2)},b_{2'}^c}$.
Observe that $D$ is normal in $H$ because $[b_{2'}^{c},a]=[b_{2'}^{c},a_{\pi\setminus \{2\}}]\in \GEN{a_{\pi\setminus \{2\}}}\subseteq D$.

We denote $C=\{g\in H: ge=e\}=\Core_H(K)$ and $N=N_H(K)=\GEN{a,b^t}$ with $t=[H:N]$. Observe that $L$ is nilpotent because $a_{\pi'}\in Z(L)$ and $[b_{\pi'},a_{\pi}]=1$.
Moreover, $S\cong M_t(\Q_{[L:K]}*N/L)$ where $N/L=\GEN{b^tL}$ and the action $\rho$ of the crossed product is given by $\rho(\zeta_{[L:K]})=\zeta_{[L:K]}^x$, if $L=\GEN{u,K}$ and $u^{b^t}\in u^xK$.
By Lemma \ref{PCI2-A1}, $a_{\pi\setminus \{2\}}\in K$ and $b_{2'}^c\in K$. Moreover, as $L/K$ is abelian, $\GEN{[a_2, b_2^{c}]} = \GEN{a_2^{\max(k_2r_2,m'_2)}}\subseteq K$.
Thus $D\subseteq K$.

\medskip\refstepcounter{claimcounter}
\noindent\underline{Claim \theclaimcounter}. \label{Knormalparnew}
$K$ is normal in $H$, $L=\GEN{a,K}$ and $\varphi([L:K])=\varphi(m_{\pi'}m'_2)$.

\begin{proof}
By condition \ref{F2}, $F\cap \Q_{m'_2} = \Q_{m'_2}^{\sigma}$ where $\sigma(\zeta_{m'_2})=\zeta_{m'_2}^{-1+r_2}$. As $r_2<m'_2$, it follows that $F\cap \Q_{m'_2}$ is not a cyclotomic field.
On the other hand,  $F\subseteq \Q_{[L:K]}\cap \Q_{m_{\pi'}m'_2}$, and hence $F\cap \Q_{[L:K]_2} = F\cap \Q_{\max([L:K]_2,m'_2)} = F\cap \Q_{m'_2}$. Thus $F\cap \Q_{[L:K]_2}$ is not a cyclotomic field.

Since $L_2K/K=\GEN{a_2K,b_2^cK}$, $(L/K)_2$ is generated by $a_2K$ or $b_2^cK$. In the latter case, $F\cap \Q_{[L:K]_2}$ is a cyclotomic field, contradicting the previous paragraph. Thus $L_2K/K=\GEN{a_2K}$ and, as $L=\GEN{K,a_{\pi'},a_2,b^c}$, it follows that $L=\GEN{a,K}$.

As in the proof of Claim~\ref{Knormalnew}, $L$ is nilpotent and to prove that $K$ is normal in $H$ it is enough to show that $[b_2,K_2]\subseteq K_2$. Otherwise, the index $t$ of $N_H(K)$ in $H$ is even and hence, as $(L/K)_2$ is generated by $a_2K$, it follows that $F\cap \Q_{[L:K]_2}=\Q_{[L:K]_2}^{\tau}$, where $\tau(\zeta_{[L:K]_2})=\zeta_{[L:K]_2}^{(-1+r_2)^t}$. However, the discussion on the cyclic $p$-subgroups of $\U_{m}$ with $m$ a $p$-th power in \Cref{SectionNumberTheory} shows that $\GEN{(-1+r_2)^t}_{[L:K]_2}=\GEN{1+u}_{[L:K]_2}$ with $u$ a multiple of $4$ dividing $[L:K]_2$. Moreover, $o_{[L:K]_2}(1+u)=\frac{o_{[L:K]_2}(-1+r_2)}{t_2}$. Then, using \Cref{PropEse}.\eqref{vpRm-1}, we derive that $\GEN{(-1+r_2)^t}_{[L:K]_2}=\GEN{1+r_2t_2}_{[L:K]_2}$. Therefore $F\cap \Q_{[L:K]_2}=\Q_{r_2t_2}$, again a contradiction with the first paragraph of the proof.

The last equality of the claim follows by the same arguments as in the last paragraph of the proof of Claim~\ref{Claim1}.
\end{proof}

As in the previous cases we take $S=\Q Ge(G,L,K_0)\cong \Q He(H,L,K)$ for some strong Shoda pair $(L,K)$ of $H$.
By Claim~\eqref{Knormalparnew}, $S\cong (\mathbb{Q}_{[L:K]}*\GEN{\rho})$ and if $Z(S)\cong F$, then
$\mathbb{Q}_{m_{\pi'}m'_2}^{\gamma^G}=F=
\mathbb{Q}_{[L:K]}^{\gamma^H}$.
As $L=\GEN{a,K}$ and $a_{\pi'}\in K$ we have that $[L:K]\mid m_{\pi'}m_2$. Combining this with $m'_2\in \left\{m_2,\frac{m_2}{2}\right\}$ and $\varphi([L:K]) = \varphi(m_{\pi'}m'_2)$, it is easy to see that either $[L:K]=m_{\pi'}m'_2$ or $m'_2=\frac{m_2}{2}$ and $[L:K]=\frac{m_{\pi'}}{3}m_2$.
In the first case, $(\Delta^G)_2=\GEN{(\sigma^G)_2}=\GEN{(\sigma^H)_2}=(\Delta^H)_2$, as desired.
In the second case, $\Res_{\frac{m_{\pi'}}{3}m'_2}(\sigma^G) = \Res_{\frac{m_{\pi'}}{3}m'_2}(\sigma^H)$ and $\sigma^G(\zeta_3)=\zeta_3^{-1}$. This implies again that $(\Delta_2)^G=\Res_{m_{\pi'}m'_2}(\sigma^G) = \Res_{m_{\pi'}m'_2}(\sigma^H) = (\Delta_2)^H$, as desired.

This finishes the proof of \Cref{DeltaDetermined} and completes the proof of \Cref{Main}.
\end{proof}

\bibliographystyle{amsalpha}
\bibliography{ReferencesMSC}	

\newcommand{\etalchar}[1]{$^{#1}$}
\providecommand{\bysame}{\leavevmode\hbox to3em{\hrulefill}\thinspace}
\providecommand{\MR}{\relax\ifhmode\unskip\space\fi MR }
\providecommand{\MRhref}[2]{%
  \href{http://www.ams.org/mathscinet-getitem?mr=#1}{#2}
}
\providecommand{\href}[2]{#2}
\begin{thebibliography}{BBCH{\etalchar{+}}24}

\bibitem[BBCH{\etalchar{+}}24]{Wedderga4.10.5}
G.~K. Bakshi, O.~Broche~Cristo, A.~Herman, O.~Konovalov, S.~Maheshwary,
  G.~Olteanu, A.~Olivieri, A.~del Río, and I.~Van~Gelder, \emph{{Wedderga},
  wedderburn decomposition of group algebras, {V}ersion 4.10.5}, \href
  {https://gap-packages.github.io/wedderga}
  {\texttt{https://gap-packages.github.io/}\discretionary
  {}{}{}\texttt{wedderga}}, Feb 2024, Refereed GAP package.

\bibitem[Dad71]{Dade71}
E.~Dade, \emph{Deux groupes finis distincts ayant la m\^{e}me alg\`ebre de
  groupe sur tout corps}, Math. Z. \textbf{119} (1971), 345--348.

\bibitem[GAP24]{GAP4}
The GAP~Group, \emph{{GAP -- Groups, Algorithms, and Programming, Version
  4.14.0}}, 2024.

\bibitem[GBdR23]{GarciadelRioClasification}
\`{A}ngel Garc\'{\i}a-Bl\'{a}zquez and \'{A}ngel del R\'{\i}o, \emph{A
  classification of metacyclic groups by group invariants}, Bull. Math. Soc.
  Sci. Math. Roumanie (N.S.) \textbf{66(114)} (2023), no.~2, 209--233.

\bibitem[GBdR24]{GarciadelRioNilpotent}
À. García-Blázquez and \'{A}. del R\'{\i}o, \emph{The isomorphism problem
  for rational group algebras of finite metacyclic nilpotent groups}, Quaest.
  Math. \textbf{47} (2024), no.~9, 1863--1885.

\bibitem[Hem00]{Hempel2000}
C.~E. Hempel, \emph{Metacyclic groups}, Communications in Algebra \textbf{28}
  (2000), no.~8, 3865--3897.

\bibitem[Her01]{Hertweck2001}
M.~Hertweck, \emph{A counterexample to the isomorphism problem for integral
  group rings}, Ann. of Math. \textbf{154} (2001), 115--138.

\bibitem[Hig40a]{Higman1940Thesis}
G.~Higman, \emph{Units in group rings}, University of Oxford, 1940, Thesis
  (Ph.D.).

\bibitem[Hig40b]{Higman1940Paper}
\bysame, \emph{The units of group-rings}, Proc. London Math. Soc. (2)
  \textbf{46} (1940), 231--248.

\bibitem[JdR16]{JespersdelRioGRG1}
E.~Jespers and {\'A}.~del R{\'{\i}}o, \emph{{Group Ring Groups. Volume 1:
  Orders and generic constructions of units}}, Berlin: De Gruyter, 2016.

\bibitem[Mar22]{Margolis2022}
Leo Margolis, \emph{The modular isomorphism problem: a survey}, Jahresber.
  Dtsch. Math.-Ver. \textbf{124} (2022), no.~3, 157--196.

\bibitem[MR19]{MargolisdelRioSurvey}
L.~Margolis and {\'A}.~del R{\'i}o, \emph{Finite subgroups of group rings: a
  survey}, Adv. Group Theory Appl. \textbf{8} (2019), 1--37.

\bibitem[OdRS04]{OlivieridelRioSimon2004}
A.~Olivieri, \'{A}. del R\'{\i}o, and J.~J. Sim\'{o}n, \emph{On monomial
  characters and central idempotents of rational group algebras}, Comm. Algebra
  \textbf{32} (2004), no.~4, 1531--1550.

\bibitem[OdRS06]{OlivieridelRioSimon2006}
A.~Olivieri, \'{A}. del R{\'{\i}}o, and J.~J. Sim\'{o}n, \emph{The group of
  automorphisms of the rational group algebra of a finite metacyclic group},
  Comm. Algebra \textbf{34} (2006), no.~10, 3543--3567.

\bibitem[Pas77]{Passman1977}
D.~S. Passman, \emph{The algebraic structure of group rings}, Pure and Applied
  Mathematics, Wiley-Interscience [John Wiley \& Sons], New York, 1977.

\bibitem[Pie82]{Pierce1982}
R.~S. Pierce, \emph{Associative algebras}, Graduate Texts in Mathematics,
  vol.~88, Springer-Verlag, New York, 1982, Studies in the History of Modern
  Science, 9.

\bibitem[PW50]{PerlisWalker1950}
S.~Perlis and G.~L. Walker, \emph{Abelian group algebras of finite order},
  Trans. Amer. Math. Soc. \textbf{68} (1950), 420--426. \MR{0034758 (11,638k)}

\bibitem[RS87]{RoggenkampScott1987}
Klaus Roggenkamp and Leonard Scott, \emph{Isomorphisms of {$p$}-adic group
  rings}, Ann. of Math. (2) \textbf{126} (1987), no.~3, 593--647. \MR{916720}

\bibitem[RT92]{RoggenkampTaylor}
K.~W. Roggenkamp and M.~J. Taylor, \emph{Group rings and class groups}, DMV
  Seminar, vol.~18, Birkh\"auser Verlag, Basel, 1992.

\bibitem[Seh78]{Sehgal1978}
S.~K. Sehgal, \emph{Topics in group rings}, Monographs and Textbooks in Pure
  and Applied Math., vol.~50, Marcel Dekker Inc., New York, 1978.

\bibitem[Whi68]{Whitcomb}
A.~Whitcomb, \emph{The {G}roup {R}ing {P}roblem}, ProQuest LLC, Ann Arbor, MI,
  1968, Thesis (Ph.D.)--The University of Chicago.

\end{thebibliography}

\end{document}